\documentclass[11pt]{article}
\usepackage{setspace}
\doublespace
\usepackage{authblk}
\usepackage{amsmath,amsthm,amssymb,color}
\usepackage{epsfig,amsfonts,latexsym, hyperref,  mathrsfs}
\usepackage{natbib}
\usepackage{geometry}
\usepackage{tikz}
\usepackage{tikz-qtree}
\usetikzlibrary{shadows,trees}

\hypersetup{
    colorlinks=true,
    linkcolor=red,
    citecolor=blue,
    urlcolor=blue,
    pdfborder={0 0 0}
}

\newtheorem{remark}{Remark}[section]
\newtheorem{theorem}{Theorem}[section]

\newtheorem{prop}[theorem]{Proposition}
\newtheorem{lemma}{Lemma}[section]

\theoremstyle{definition}

\newtheorem{Example}{Example}[section]

\numberwithin{equation}{section}

\allowdisplaybreaks

\title{Common change point estimation in panel data from the least squares and maximum likelihood viewpoints}

\date{\today}

\author[1]{Monika Bhattacharjee}
\author[2]{Moulinath Banerjee}
\author[3]{George Michailidis}
\affil[1]{Informatics Institute, University of Florida, Gainesville, USA}
\affil[2]{Department of Statistics, University of Michigan, Ann Arbor, USA}
\affil[3]{Department of Statistics \& Informatics Institute, University of Florida, Gainesville, USA}

\begin{document}

\maketitle

   \begin{abstract}  We establish the convergence rates and asymptotic distributions of the common break change-point-estimators, obtained by least squares and maximum likelihood in panel data models and compare their asymptotic variances. Our model assumptions accommodate a variety of commonly encountered probability distributions and, in particular, models of particular interest in econometrics beyond the commonly analyzed Gaussian model, including the zero-inflated Poisson model for count data, and the probit and tobit models. We also provide novel results for time dependent data in the signal-plus-noise model, with emphasis on a wide array of noise processes, including Gaussian process, MA$(\infty)$ and $m$-dependent processes. The obtained results show that maximum likelihood estimation requires a stronger signal-to-noise model identifiability condition compared to its least squares
counterpart. Finally, since there are three different asymptotic regimes that depend on the behavior of the norm difference of the model parameters 
before and after the change point, which cannot be realistically assumed to be known, we develop a novel{\em data driven adaptive procedure} that provides valid confidence intervals for the common break, without requiring a priori knowledge of the asymptotic regime the problem falls in.
      \end{abstract}
   \medskip

 \medskip

 \bigskip

%\section{Introduction} 
\section{Introduction}

The change point problem for univariate data has a long history in the econometrics and statistics literature. A broad overview of the
technical aspects of the problem is provided in \citet{Basseville1993detection, csorgo1997limit}. The problem has a wide range of applications
in economics \citep{baltagi2016estimation, RePEc:siu:wpaper:07-2015, doi:10.1080/01621459.2015.1119696} and finance 
\citep{frisen2008financial}, while other standard areas include quality monitoring and control \citep{qiu2013introduction}, as well as newer ones
such as genetics and medicine \citep{chen2011parametric} and neuroscience \citep{koepcke2016single}. On the other hand, there is little work
when it comes to panel data, despite the presence of common break  in such data as argued in \citet{bai2010common}. Further, most
of the analytical emphasis is on numerical/continuous data, although there are a lot of applications involving count data (see
\citep{cameron2013count, hsiao2014analysis} and references therein) and binary data \citep{park2011changepoint, wu2008change}
or categorical data \citep{zhang2010detecting}.

The technical literature on change point analysis for panel data focuses on the common break model given by
\begin{eqnarray}
\label{model.defn}
X_{it} & = & \mu_{i1}+\epsilon_{it}, \ \ \ t=1,2,\cdots,\tau \\ \nonumber
X_{it} & = & \mu_{i2}+\epsilon_{it}, \ \ \ t=\tau+1,\cdots,T \\
       &    & i=1,\cdots,N, \nonumber
\end{eqnarray}
where $\tau$ represents a common break point for all $N$ series, the difference $|\mu_{i1}-\mu_{i2}|$ represents the magnitude of the
shift for each series and $\epsilon_{it}$ are random noise processes that are cross-sectionally independent. 
\citet{bai2010common} employed a least squares criterion to estimate the common
change point $\tau$ and established its asymptotic distribution, while \citet{horvath2012change} developed tests for the presence of a
change point during the observation period. \citet{kim2014common} investigated estimation of the change point under cross-sectional dependence
in panels modeled by a common factor (see also \citep{baltagi2016estimation}).

As previously mentioned, the focus in the literature has been on the estimation of the common change point based on the mean shift model
using a least squares criterion. However, for other types of data, such as count data that can be modeled by Poisson or negative binomial
models and their zero inflated counterparts \citep{cameron2013count}, maximum likelihood estimation is a more suitable procedure.  The same
holds true for more complex models such as probit or Tobit models \citep{park2011changepoint}. To emphasize the latter point, consider
a zero-inflated Poisson model characterized by tthe following two parameters: (i) $\sigma$ the probability of extra zero counts and (ii) $\lambda$
the expected count of the Poisson component. The mean is given by the expession $(1-\sigma)\lambda$ and one can consider settings where
simultaneous changes in the $\sigma, \lambda$ parameters before and after the change point do not lead to changes in the corresponding mean
parameter. A least squares criterion, based on fitting different means before and after a candidate for the change point, would not be able to identify 
such structural changes, while a maximum likelihood based criterion clearly would. The same holds true for other complex models and hence a {\em comprehensive} study of the problem under the 
maximum likelihood criterion is warranted.

The key objective of this paper is to investigate the estimation of the common change point in {\em independent panel data} based both
on the least squares and the maximum likelihood criteria for a wide class of statistical models and further compare the assumptions needed to establish
consistency of the respective estimates, as well as the nature of their asymptotic distributions. To the best of our knowledge, this
is the first comprehensive treatise in the literature on maximum likelihood based estimation of the change point for panel data in a general setting.
%[Further, note that an investigation of the change point even for the single variable case under the maximum likelihood criterion has not been undertaken.]
The general setting adopted, encompasses as a special case, exponential families. Further,
for the least squares criterion for panel data, we also consider a more general setting for temporally dependent data than the one considered in previous
literature (e.g. \citep{bai2010common}).

Our results show
that maximum likelihood estimates require a {\em  stronger identifiability condition} (denoted as  SNR2 in Section \ref{sec: mle}) vis-a-vis that
for least squares estimates (denoted as SNR1 in Section \ref{sec: lse}), 
while the asymptotic distribution of the change point 
exhibits smaller variance. The singular case is for normally distributed data, where the identifiability condition needed to establish consistency
and obtain the asymptotic distribution is identical for the two criteria. 

Another key contribution of the paper is the introduction of a data based {\em adaptive inference} scheme for obtaining the asymptotic distribution
of the change point estimate in practice. As established in the literature of least squares criterion and further shown in this paper for the maximum likelihood criterion,
there are three distinct asymptotic regimes for the change point estimator that depend on the norm difference of the model parameters before and after the
change point. Since that norm difference is not {\em a priori known}, the practioner faces a dilemma of which regime to employ for the construction
of confidence intervals for the change point parameter.  
Our proposed scheme overcomes this issue and provides a unified regime that {\em self-adapts}
to the true underlying setting, thus enabling the data analyst to construct accurate confidence intervals. To the best of our knowledge, this topic has not
been  pursued in the literature before.

% Define MSE and LH method here
%\section{Model and assumptions}

\noindent
{\bf Problem Formulation:}
We consider panel data comprising $m$ series (variables), with each series  observed at $n$ time
points. The observations are denoted by $\{X_{kt}^{(n)}:\  1 \leq k \leq m,\  1 \leq t \leq n\}$. 
In general, the sequence of observations available depends on the number of time points $n$; however, for ease of exposition and to reduce notational
overhead, we shall write $X_{kt}$ for $X_{kt}^{(n)}$. Further, 
there is a single structural change common across all panels, that occurs at  $\tau_n \in (0,1)$, referred to as the common break/change point.
We assume that  $\{X_{kt}\}$ are independent over $k$. For each $k$, variables within the sets $\{X_{kt}: t \leq n\tau_n\}$  and $\{X_{kt}: t > n\tau_n\}$ are independently and identically distributed, whereas variables between these two sets are independent\footnote{The assumption of independence across time is relaxed in Section 2.3.}. Throughout this paper, we assume $0<c^{*}<\tau_n<1-c^{*}<1$ for some $c^{*}>0$ and consider estimates of $\tau_n$ which are in $(c^{*},1-c^{*})$.

We are interested in obtaining the least squares and the maximum likelihood estimates of $\tau_n$ for a wide range of statistical models, under suitable
regularity conditions. The least squares estimation problem is presented in Section \ref{sec: lse} for independent and identically distributed data,
while that of maximum likelihood estimation in Section \ref{sec: mle}. Further, extensions to time dependent data for least squares estimates are
presented in Section \ref{lse:time-dependent}. Finally, the issue of adaptive inference is examined in Section \ref{sec: adaptive}.

\noindent 
The following diagram provides a schematic road-map for the main results established, as well as illuminating examples of interest in econometrics.
In the diagram the following abbreviations are employed: \textit{indep}: independent, \textit{dep}: dependent, \text{conv}: Convergence, \text{asymp distribn}: asymptotic distribution,  \textit{Thm}: Theorem, \textit{Prop}: Proposition, \textit{Rem}: Remark, \textit{Pf}: Proof, \textit{Sec}: Section, \textit{Exm}: Example, \textit{WN}: white noise, \textit{adap inf}: adaptive inference and  \textit{Supp}: Supplementary file. By $\gamma$ we mean $\gamma_{_\text{L,LSE}}$, $\gamma_{_\text{R,LSE}}$, $\gamma_{_\text{L,LSE}}^{*}$, $\gamma_{_\text{R,LSE}}^{*}$ and $c_1$.

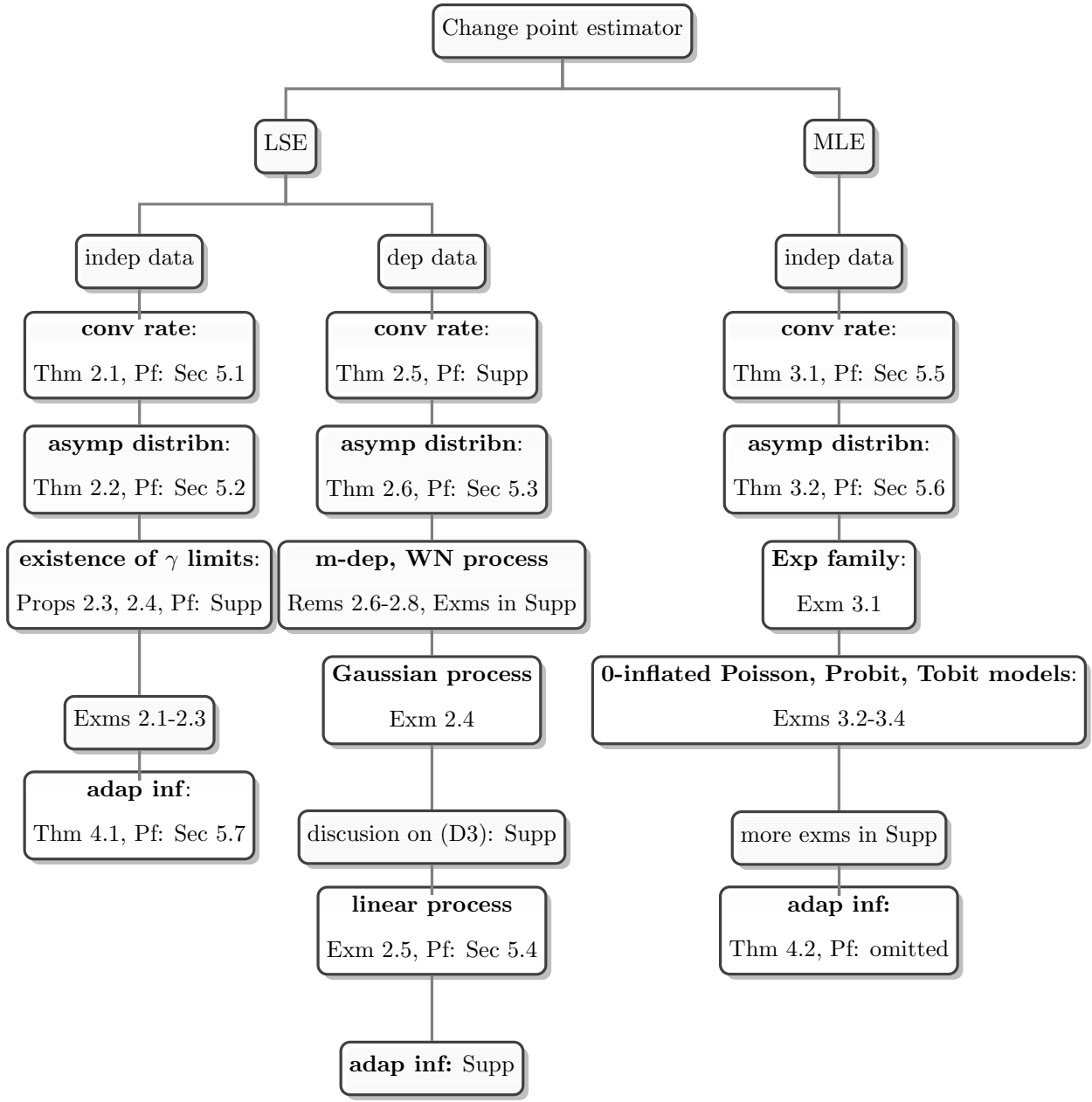
\tikzset{font=\small,
edge from parent fork down,
level distance=1.75cm,
every node/.style=
    {top color=white,
    bottom color=white,
    rectangle,rounded corners,
    minimum height=8mm,
    draw=black!75,
    very thick,
    drop shadow,
    align=center,
    text depth = 4pt
    },
edge from parent/.style=
    {draw=black!50,
    very thick,
    level 1/.style={sibling distance=16cm}}}

\begin{center}
\begin{tikzpicture}
\Tree [.{Change point estimator}
        [.{LSE}
              [.{indep data} 
                  [.{\textbf{conv rate}:\\ Thm 2.1, Pf: Sec 5.1} 
                     [.{\textbf{asymp distribn}: \\ Thm 2.2, Pf: Sec 5.2}  
                        [.{\textbf{existence of $\gamma$ limits}:\\ Props 2.3, 2.4, Pf: Supp} 
                           [.{Exms 2.1-2.3} 
                             [.{\textbf{adap inf}:\\ Thm 4.1, Pf: Sec 5.7} ] ] ] ] ] ]
                [.{dep data} 
                     [.{\textbf{conv rate}:\\ Thm 2.5, Pf: Supp}
                         [.{\textbf{asymp distribn}:\\ Thm 2.6, Pf: Sec 5.3}
                            [.{\textbf{m-dep, WN process}\\ Rems 2.6-2.8, Exms in Supp}
                               [.{\textbf{Gaussian process}\\ Exm 2.4}
                                  [.{discusion on (D3): Supp} 
                                    [.{\textbf{linear process}\\ Exm 2.5, Pf: Sec 5.4} 
                                       [.{\textbf{adap inf:} Supp} ] ] ] ] ] ] ] ] ]
        [.{MLE} 
            [.{indep data} 
                 [.{\textbf{conv rate}: \\ Thm 3.1, Pf: Sec 5.5} 
                    [.{\textbf{asymp distribn}:\\ Thm 3.2, Pf: Sec 5.6} 
                      [.{\textbf{Exp family}:\\ Exm 3.1}   
                        [.{\textbf{0-inflated Poisson, Probit, Tobit models}:\\ Exms 3.2-3.4}  
                          [.{more exms in Supp} 
                            [.{\textbf{adap inf:}\\ Thm 4.2, Pf: omitted} ] ] ] ] ] ] ] ]
 ]
\end{tikzpicture}
\end{center}

\section{Least squares estimation of the common break model parameters} \label{sec: lse}

We  present asymptotic properties of the least squares estimator of the change point $\tau_n$. 
%As discussed in previous section, this method imposes weaker conditions on the probability distribution of $\{X_{kt}\}$ and as well as on signal-to-noise ratio. 
The underlying assumption is  that the break $\tau_n$ occurs due to a change in mean parameters of $\{X_{kt}\}$, which is  equivalent to the following statement.  For each  $k \geq 1$, 
$$E(X_{kt}) = \mu_{1k}(n)I(t \leq n\tau_n) + \mu_{2k}(n)I(t>n\tau_n),$$ 
where $\mu_{1k}(n) \neq \mu_{2k}(n)$ for at least one $k$. Note that in general, $\{EX_{kt}\}$ depends on $n$. For ease of exposition, henceforth we write $\mu_{ik}$ for $\mu_{ik}(n)$ for all $k, n \geq 1$. 
\newline
\indent The least squares estimator $\hat{\tau}_{n,\text{LSE}}$ of $\tau_n$ can be obtained by optimizing the following criterion function:
\begin{eqnarray} 
 \hat{\tau}_{n,\text{LSE}} &=& \arg\max_{b \in (c^{*},1-c^{*})} M_{n}(b)\ \ \text{where} \nonumber \\
 M_{n}(b) &=& \sum_{k=1}^{m} M_{k,n}(b),\ \ 
M_{k,n}(b) = -\frac{1}{n}\bigg[ \sum_{t=1}^{nb}(X_{kt}-\hat{\mu}_{1k}(b) )^2 + \sum_{t=nb+1}^{n} (X_{kt}-\hat{\mu}_{2k}(b) )^2 \bigg], \nonumber \\
\hat{\mu}_{1k}(b) &=& \frac{1}{nb} \sum_{t=1}^{nb} X_{kt}\ \ \text{and}\ \ \hat{\mu}_{2k}(b) = \frac{1}{n(1-b)} \sum_{t=nb+1}^{n} X_{kt}. \label{eqn: optilse}
\end{eqnarray}

\bigskip\noindent
{\bf Rate of convergence for $\hat{\tau}_{n,\text{LSE}}$}. To establish our results, we consider the following assumptions. % are required:

\noindent \textbf{(A1)} $\sup_{k,n, t} E(X_{kt}- E(X_{kt}))^4 < \infty$. \vskip 3pt
\noindent Note that $(A1)$ implies  $\sup_{k,n, t} \text{Var}(X_{kt}) < \infty$. 

\noindent Set $\mu_i = (\mu_{i1},\mu_{i2},\ldots,\mu_{im}),\ i=1,2$ and consider the following {\em signal-to-noise} condition. \vskip 3pt
\noindent
\textbf{(SNR1)} $nm^{-1}||\mu_1 -\mu_2||_2^2 \to \infty$ as $n \to \infty$.

\noindent Assumption (A1) controls the 4th moment of $\{X_{kt}\}$, which arises to control the variance of the least squares quadratic criterion function posited above.  Observe that $n||\mu_1 - \mu_2||_2^2$ is the gross signal in the observed data set. Therefore,  $nm^{-1}||\mu_1 -\mu_2||_2^2$ indicates average signal per series, which is allowed to grow to $\infty$ in (SNR1). 
Given the (A1) and (SNR1) assumptions, the following rate result for $\hat{\tau}_{n,\text{LSE}}$ can be established, whose proof is given in Section \ref{subsec: proofmlelserate}. 
\begin{theorem} \textbf{Least squares convergence rate}.  \label{thm: lse1}
Suppose (A1)  and (SNR1) hold. Then,
\begin{eqnarray}
n||\mu_1 - \mu_2||_2^2 (\hat{\tau}_{n,\text{LSE}}-\tau_n) = O_{P}(1). \nonumber 
\end{eqnarray}
\end{theorem}

\bigskip\noindent
{\bf Asymptotic distribution of $\hat{\tau}_{n,\text{LSE}}$}. 
For the panel data setting, there are three different regimes, as shown in \citet{bai2010common} : (a) $\lim_{n \to \infty}  ||\mu_1 - \mu_2||_2 \to \infty$, (b) $\lim_{n\to\infty}  ||\mu_1 - \mu_2||_2 \to 0$ and (c) $\lim_{n\to\infty}  ||\mu_1 - \mu_2||_2 \to c >0$. Asymptotic distributions of the change point
estimate are different in these three regimes.  %For the first case, rate of convergence for $\hat{\tau}_{n,\text{LSE}}$ is very high and therefore $(\hat{\tau}_{n,\text{LSE}} - \tau_n)$ is degenerated at $0$ (see Theorem \ref{thm: lse2}(a)).  In the second and third regimes,  signals are weak and therefore under some additional assumptions, we obtain non-degenerate asymptotic distribution of $n||\mu_1 - \mu_2||_2^2 (\hat{\tau}_{n,\text{LSE}}-\tau_n)$.
%Before stating the assumptions needed to establish the main result, r
\newline \indent Recall that
in the presence of a single panel, the following two results have been established in the literature:  (i) if $|\mu_1-\mu_2|\rightarrow
0$ (Regime (b)) at an appropriate rate as a function of the sample size $n$, then the asymptotic distribution of the change point is given by the maximizer of
a Brownian motion with triangular drift  (for details see \cite{bhattacharya1994}); and (ii) if $|\mu_1-\mu_2|\rightarrow c$ (Regime (c)), then the
asymptotic distribution of the change point, in the random design setting,  is given by the maximizer of a two-sided compound Poisson process (for details see Chapter 14 of the book by \cite{kosorok2008}).  As previously mentioned and will be established rigorously
next, in the panel data setting analogous regimes emerge, with the modification that in the case of (ii) since we are dealing with a fixed design, the
process becomes a two-sided generalized random walk; in addition there exists a third one (Regime (a)), where the asymptotic distribution of the change point becomes degenerate at the true value.
\newline
\indent Next, we introduce  assumptions needed to establish these results.  In Regime (a),  the asymptotic distribution of the change point can be derived under the same assumptions as in Theorem \ref{thm: lse1}. On the other hand, in the second and third regimes,  a non-degenerate limit distribution can be obtained under the following additional  assumptions.  Detailed comments on these assumptions will be provided after stating the results. 

\vskip 5pt
\noindent \textbf{Regime (b): $||\mu_1 - \mu_2||_2 \to 0$, assumptions}. 
Note that  existence of $\text{Var}(X_{kt})$ is guaranteed by (A1).  Denote
\begin{eqnarray}
\text{Var}(X_{kt}) = \sigma_{1k}^2(n) I(t \leq n\tau_n) + \sigma_{2k}^{2}(n)I(t>n\tau_n),\ \  \forall k \geq 1.
\end{eqnarray}
\noindent For ease of presentation,  we write $\sigma_{ik}$ for $\sigma_{ik}(n)$. Let,
\begin{eqnarray} \label{eqn: gammaLRnew}
\gamma_{_{\text{L, LSE}}}^2 = \lim_{n \to \infty} \frac{\sum_{k=1}^{m} (\mu_{1k} - \mu_{2k})^2 \sigma_{1k}^2}{||\mu_1 - \mu_2||_2^{2}}\ \ \text{and}\ \ \ \gamma_{_{\text{R, LSE}}}^2 = \lim_{n \to \infty} \frac{\sum_{k=1}^{m} (\mu_{1k} - \mu_{2k})^2 \sigma_{2k}^2}{||\mu_1 - \mu_2||_2^{2}}.
\end{eqnarray}

\noindent We then require, \\
\noindent \textbf{(A2)} $\gamma_{_{\text{L, LSE}}}$ and $\gamma_{_{\text{R, LSE}}}$ exist. \\
\noindent \textbf{(A3)} There exists an $\epsilon >0$, such that $\inf_{k,t,n} \text{Var}(X_{kt})>\epsilon$.
\vskip 5pt

\noindent \textbf{Regime (c): $||\mu_1-\mu_2||_2 \to c>0$, assumptions}.
\noindent %For a set $A$, let $A^c$ be its complement set.   
Consider the following disjoint and exhaustive subsets of $\{1,2,3,\ldots, m(n)\}$:
\begin{align} \label{eqn: partition}
\mathcal{K}_0 &= \{k:\ 1 \leq k \leq m(n),\ \lim (\mu_{1k} - \mu_{2k}) \neq 0\} \ \ \text{and} \\
\mathcal{K}_n &= \mathcal{K}_0^c = \{k:\ 1 \leq k \leq m(n),\ \lim (\mu_{1k} - \mu_{2k}) = 0\},\ \forall n \geq 1. \nonumber
\end{align}
Clearly, $\mathcal{K}_0$ is a finite set. We consider the following assumptions on $\mathcal{K}_0$.
\vskip 3pt
\noindent \textbf{(A4)}  $\mathcal{K}_0$ does not vary with $n$.  
\vskip 3pt
\noindent \textbf{(A5)} For some $\tau^{*} \in (c^{*},1-c^{*})$, $\tau_n \to \tau^{*}$ as $n \to \infty$. Moreover,  there is a collection of independent random variables $\{X_{ik}^{*}: k \in \mathcal{K}_0, i=1,2\}$ such that for all $0<f<1$,
\begin{align}
X_{k\lfloor nf \rfloor} \stackrel{\mathcal{D}}{\to} X_{1k}^{*}I(f \leq \tau^{*}) + X_{2k}^{*}I(f > \tau^{*}).
\end{align}
Let $E(X_{ik}^{*}) = \mu_{ik}^{*}$ (say). For all $k \in \mathcal{K}_0$ and $i=1,2$, we have
$\mu_{ik}(n) \to \mu_{ik}^{*}$.
%For all $t \geq 1$, the probability distribution of $\{X_{kt}: k \in \mathcal{K}_0\}$ does not change with $n$. 
%\vskip 3pt
%\noindent As a consequence,  $\mu_{1k}(n) = \mu_{1k}(n+i)$ and  $\mu_{2k}(n) = \mu_{2k}(n+i)$ for all $k \in \mathcal{K}_0$ and $n,i \geq 1$.
\vskip 3pt
\noindent We consider the following assumptions on $\mathcal{K}_n$. Let
\begin{eqnarray}
c_1^2 &=& \lim_{n \to \infty} \sum_{k \in \mathcal{K}_n} (\mu_{1k} - \mu_{2k})^2, \\ 
\gamma_{_{\text{L, LSE}}}^{*2} &=& \lim_{n \to \infty} \sum_{k \in \mathcal{K}_n} (\mu_{1k} - \mu_{2k})^2 \sigma_{1k}^2, \ \ \ 
\gamma_{_{\text{R, LSE}}}^{*2} = \lim_{n \to \infty} \sum_{k \in \mathcal{K}_n}  (\mu_{1k} - \mu_{2k})^2 \sigma_{2k}^2. \label{eqn: gammastarLRnew}
\end{eqnarray}
\noindent \textbf{(A6)} $c_1$,  $\gamma_{_{\text{L, LSE}}}^{*}$ and $\gamma_{_{\text{R, LSE}}}^{*}$ exist. \\
\noindent \textbf{(A7)} $\sup_{k \in \mathcal{K}_n} |\mu_{1k}-\mu_{2k}| \to 0$.
\vskip 10pt

Let $X \stackrel{d}{=} Y$ denote equality in distribution for random variables $X$ and $Y$.  The following theorem provides the limiting distribution of $\hat{\tau}_{n,\text{LSE}}$. Its proof is given in Section \ref{subsec: lse2}.
\begin{theorem} \textbf {Least squares asympotic distributions}. \label{thm: lse2}
Suppose (A1)  and (SNR1) hold. Then, the following statements are true.

\noindent
{\bf (a)} If $||\mu_1-\mu_2||_2 \to \infty$, then 
$$\lim_{n \to \infty} P(n||\mu_1-\mu_2||_2^2 (\hat{\tau}_{n,\text{LSE}}-\tau_n)=0) = \lim_{n \to \infty} P(\hat{\tau}_{n,\text{LSE}}=\tau_n) =1.$$

\noindent
{\bf (b)} If (A2) and (A3) hold and  $||\mu_1-\mu_2||_2 \to 0$, then
\begin{eqnarray}
n||\mu_1-\mu_2||_2^2 (\hat{\tau}_{n,\text{LSE}}-\tau_n) \stackrel{\mathcal{D}}{\to}  \arg \max_{h \in \mathbb{R}} (-0.5|h| + \gamma_{_{\text{L,LSE}}}B_hI(h\leq 0) + \gamma_{_{\text{R,LSE}}}B_h I(h>0)),\ \ 
\end{eqnarray}
where $B_h$ denotes the standard Brownian motion.

\noindent
{\bf (c)} Suppose (A3)-(A7) hold and $||\mu_1 -\mu_2||_2 \to c >0$, then
\begin{eqnarray}
n||\mu_1 - \mu_2||_2^2 (\hat{\tau}_{n,\text{LSE}} -\tau_n) &\stackrel{\mathcal{D}}{\to}& \arg \max_{h \in c^2\mathbb{Z}} (D(h) + C(h) + A(h)) \nonumber 
%&=& c^2 \arg \max_{c^2h \in c^2\mathbb{Z}} (D(c^2 h) + C(c^2 h) + A(c^2 h)), \nonumber 
\end{eqnarray}
where  for each $\tilde{h} \in \mathbb{Z}$,
\begin{eqnarray}
D(c^2 (\tilde{h}+1))-D(c^2 \tilde{h}) &=& 0.5 {\rm{Sign}}(-h) c_1^2, \\
C(c^2 (\tilde{h}+1)) - C(c^2 \tilde{h}) &=& (\gamma_{_{\text{L,LSE}}}^{*}I(\tilde{h}\leq 0) + \gamma_{_{\text{R,LSE}}}^{*}I(\tilde{h}> 0)) W_{\tilde{h}},\ \ W_{\tilde{h}} \stackrel{\text{i.i.d.}}{\sim} \mathcal{N}(0,1), \ \ \ \ \ \ \  \ \label{eqn: msethm2c}\\
A(c^2 (\tilde{h}+1)) - A(c^2 \tilde{h}) &=& \sum_{k \in \mathcal{K}_0} \bigg[(Z_{k\tilde{h}} -\mu_{1k}^{*})^2 -(Z_{k\tilde{h}} -\mu_{2k}^{*})^2 \bigg],\  %\nonumber %\\
  %Z_{kh} &\stackrel{d}{=}& X_{1k}I(h \leq 0) + X_{nk}I(h > 0) 
  \label{eqn: msethm2a}
\end{eqnarray}
 $\{Z_{k\tilde{h}}\}$ are independently distributed with $Z_{k\tilde{h}} \stackrel{d}{=} X_{1k}^{*}I(\tilde{h} < 0) + X_{2k}^{*}I(\tilde{h} \geq  0)$ for all $k \in \mathcal{K}_0$. % the above increments are independently  distributed.
\end{theorem}

\noindent \textbf{Discussion of Theorem \ref{thm: lse2}}.  Next, we provide comments on the assumptions and how they relate to the three regimes established in Theorem \ref{thm: lse2}.
In the first regime, the signal for $\hat{\tau}_{n,\text{LSE}}$ is  high and therefore the difference
$(\hat{\tau}_{n,\text{LSE}} - \tau_n)$ becomes a point
mass at $0$. On the other hand, in the second and third regimes,  the total signal is weak and moderate, respectively, and a non-degenerate limit distribution can be obtained under additional  assumptions.
%therefore additional assumptions
%are required to obtain the results.
\newline
\indent
Under the last two regimes, the results are based on the weak convergence of the process 
\begin{eqnarray} 
M_{n}^{*}(h) &=& \sum_{k=1}^{m} M_{k,n}^{*}(h)\ \ \text{where}\ \  M_{k,n}^{*}(h) = n(M_{k,n}(\tau_n + n^{-1}||\mu_1-\mu_2||_2^{-2}h)-M_{k,n}(\tau_n)),  \nonumber \\
&&h \in ||\mu_1-\mu_2||_2^{2}\{-(n-1), -(n-2),\ldots, -1,0,1,\ldots, (n-2), (n-1)\}. \hspace{1 cm} \label{eqn: normalizationlse}
\end{eqnarray}
%This is due to the fact that u
Under  appropriate conditions (as mentioned in Theorem \ref{thm: lse2}), the $\arg \max_{h} M_{n}^{*}(h) = n||\mu_1 - \mu_2||_2^2 (\hat{\tau}_{n,\text{LSE}}-\tau_n)$  converges weakly to the unique $\arg \max$ of the limiting process. For more details see Lemma \ref{lem: wvandis1}.  
%Moreover, in both  cases, It is easy to see that
%\begin{eqnarray}
%M_{n}^{*}(h) = A_{n}^{I}(h) + A_{n}^{II}(h)\ \text{where}\ A_{n}^{I}(h) = o_P(1).
%\end{eqnarray}
%Therefore, it is enough to deal with the weak convergence of the process $A_{n}^{II}(h)$. 

\noindent \textbf{Regime (b): $||\mu_1-\mu_2||_2 \to 0$}. In the second regime, the asymptotic variance of $M_{n}^{*}(h)$ is proportional to $\gamma_{_{\text{L, LSE}}}$ when $h<0$ and $\gamma_{_{\text{R, LSE}}}$ if $h > 0$ and
hence the need for assumption (A2). Discussion of specific models and conditions under which (A2) is satisfied are given in Proposition \ref{prop: 1} and in Examples \ref{example: berlse}-\ref{example: normallse}. Finally, (A3) is required for establishing the non-degeneracy of the asymptotic distribution.

\noindent \textbf{Regime (c): $||\mu_1-\mu_2||_2 \to c>0$}. Under the third regime,  the limiting process has two components based on the partition of  $\{1,2,\ldots, m(n)\}$ into sets $\mathcal{K}_0$ and $\mathcal{K}_n = \mathcal{K}_0^c$, defined in (\ref{eqn: partition}).   %and $\mathcal{K}_n = \mathcal{K}_0^c$, the complement set of $\mathcal{K}_0$. %, where
%\begin{eqnarray} \label{eqn: partition}
%\mathcal{K}_n = \{k:\ 1 \leq k \leq m(n),\ \lim (\mu_{1k} - \mu_{2k}) = 0\}.
%\end{eqnarray}
Observe that $\mathcal{K}_n$ is the collection of all such  indices whose corresponding variables eventually have the same mean before and after the change point. In the second regime, $\mathcal{K}_0$ is the empty set. On the other hand, under the third regime,  $\mathcal{K}_0$ may not be empty,  but can be at most a finite set. Hence, $\mathcal{K}_n$ is necessarily an infinite set. 

These two sets in the partition contribute  differently to the limit. Note that
\begin{eqnarray}
M_{n}^{*}(h) = \sum_{k \in \mathcal{K}_n} M_{k,n}^{*}(h) + \sum_{k \in \mathcal{K}_0} M_{k,n}^{*}(h) =: M_n^{I}(h) + M_{n}^{II}(h),\ \text{say}.
\end{eqnarray}
Let $\tilde{M}_n^{II}(c^2\tilde{h}) = M_{n}^{II}(||\mu_1 - \mu_2||_2^2 \tilde{h})$ for all $\tilde{h} \in \mathbb{Z}$. Then  
\begin{align}
& \tilde{M}_{n}^{II}(c^2 (\tilde{h}+1)) - \tilde{M}_{n}^{II}(c^2 \tilde{h}) = M_{n}^{II}(||\mu_1 - \mu_2||_2^2 (\tilde{h}+1)) - M_{n}^{II}(||\mu_1 - \mu_2||_2^2 \tilde{h}) \nonumber \\
=  & \sum_{k \in \mathcal{K}_0} [ (X_{kt} - \mu_{1k})^2 - (X_{kt} - \mu_{2k})^2) ] I(t = n\tau_n + (\tilde{h}+1)).
\end{align}
Thus, since $\mathcal{K}_0$ is a finite set, by (A4) and (A5),  $M_{n}^{II}(h)$ converges weakly to the process $A(h)$, described in Theorem \ref{thm: lse2}(c). %As mentioned in Introduction, for each $k \geq 1$, the variables within the sets $\{X_{kt}: t \leq n\tau_n \}$ and $\{X_{kt}:  t > n\tau_n\}$ are independent and identically distributed. Therefore, the distribution of $\{X_{k\ n\tau_n +1}: k \in \mathcal{K}_0, n \geq 3\}$ does not vary with $n$.  This in conjunction with (A4) and (A5) makes the right side of (\ref{eqn: msethm2a}) to be free of $n$. 
\newline
\indent As mentioned before, in the second regime $\mathcal{K}_n = \{1,2,\ldots,m(n)\}$. Therefore, in the third regime, the set $\mathcal{K}_n$ can be treated in the same way as in the second regime. Thus, $M_{n}^{I}(h)$ converges weakly to an appropriately scaled  Gaussian process  on $c^2\mathbb{Z}$ with a triangular drift, as given by the definitions of $C(h)$ and $D(h)$ in Theorem \ref{thm: lse2}(c). Hence, it becomes obvious the need for assumptions (A3) and (A6).
Discussion of specific models and conditions under which (A6) is satisfied are given in Proposition \ref{prop: 2} and  subsection Illustrative Examples \ref{example: berlse}-\ref{example: normallse}. \newline
\indent (A7) is a technical assumption. Following the proof of Theorem \ref{thm: lse2}(c),  at some point we  seek to establish  the  asymptotic normality of 
\begin{align}
\sum_{k \in \mathcal{K}_n} (\mu_{1k}-\mu_{2k})(X_{kt} - E(X_{kt}))\ \forall t \geq 1. \label{eqn: asymnorlse}
\end{align}
Note that for $t \geq 1$, $\{(\mu_{1k}-\mu_{2k})(X_{kt} - E(X_{kt})): k \in \mathcal{K}_n\}$ is a collection of infinitely many independent and centered random variables. To apply Lyapunov's central limit theorem to (\ref{eqn: asymnorlse}),  we require
\begin{align}
\frac{\sum_{k \in \mathcal{K}_n} |\mu_{1k}-\mu_{2k}|^3 E|X_{kt} - E(X_{kt})|^3}{\sum_{k \in \mathcal{K}_n}(\mu_{1k}-\mu_{2k})^2 E(X_{kt} - E(X_{kt}))^2 }  \to 0.  \label{eqn: lyaplse}
\end{align}
By (A1) and (A3), the left side of (\ref{eqn: lyaplse}) is dominated by
\begin{align}
C\frac{\sum_{k \in \mathcal{K}_n} |\mu_{1k}-\mu_{2k}|^3}{\sum_{k \in \mathcal{K}_n}(\mu_{1k}-\mu_{2k})^2  } \label{eqn: lyaplsedom}
\end{align}
for some $C>0$.  (A7) is a sufficient condition for (\ref{eqn: lyaplsedom}) to converge to $0$. %be a null sequence. 
We do not need such an assumption for the second regime, since (A7)  under the second regime  is automatically satisfied.

Based on Theorem \ref{thm: lse2}, the following remarks are immediate consequences.

\begin{remark} \label{rem: lserem4}
Note that in (\ref{eqn: normalizationlse}), we use the normalization $n^{-1}||\mu_1 - \mu_2||_2^{-2}$ for $h$ for the purpose of making a unified statement.  In the third regime, we can also use the normalization $n^{-1}$ for $h$. Using the latter, we have the following restatement of Theorem \ref{thm: lse2}(c).
\vskip 3pt
Suppose (A1), (A3)-(A7) and (SNR1) hold and $||\mu_1 - \mu_2||_2 \to c>0$, then 
\begin{eqnarray}
 n (\hat{\tau}_{n,\text{LSE}} -\tau_n) &\stackrel{\mathcal{D}}{\to}& \arg \max_{h \in \mathbb{Z}} (D_1(h) + C_1(h) + A_1(h)),
\end{eqnarray}
where $D_1(h) = D(c^2h)$, $C_1(h) = C(c^2h)$, $A_1(h) = A(c^2h)$ for all $h \in \mathbb{Z}$. 
\end{remark}
Let $B_h$ be the standard Brownian motion. In the following remarks, we use the relation that for any function $f: \mathbb{R} \to \mathbb{R}$ and $C \in \mathbb{R}$,  $\arg \max_{h} f(C^2 h) = C^{-2}\arg \max_{h} f(h)$ and the processes $B_{C^2 h}$ and $CB_h$ have identical distributions.
\begin{remark} \label{rem: lse1}
Suppose $\sup_{k}|\sigma_{1k} -\sigma_{2k}| \to 0$. Then,
\begin{eqnarray}
\gamma_{_{\text{L,LSE}}} = \gamma_{_{R,LSE}}= \gamma_{_{\text{LSE}}}\ \ \text{(say)}\ \text{and}\ \ \gamma_{_{\text{L,LSE}}}^{*} = \gamma_{_{R,LSE}}^{*} = \gamma_{_{\text{LSE}}}^{*}\ \ \text{(say)}. \label{eqn: gammalse} 
\end{eqnarray}

Further suppose (A1), (A2), (A3) and (SNR1) hold and  $||\mu_1-\mu_2||_2 \to 0$, then
\begin{eqnarray}
n||\mu_1-\mu_2||^2_2 (\hat{\tau}_{n,\text{LSE}}-\tau_n) &\stackrel{\mathcal{D}}{\to}&  \arg \max_{h \in \mathbb{R}} (-0.5|h| + \gamma_{_{\text{LSE}}}B_h) \nonumber \\
 &=& \gamma_{_{\text{LSE}}}^2 \arg \max_{h \in \mathbb{R}} (-0.5\gamma_{_{\text{LSE}}}^2 |h| +\gamma_{_{\text{LSE}}}B_{\gamma_{_{\text{LSE}}}^2 h}) \nonumber \\
 \hspace{0cm} & \stackrel{d}{=} & \gamma_{_{\text{LSE}}}^2 \arg \max_{h \in \mathbb{R}} (-0.5\gamma_{_{\text{LSE}}}^2 |h| +\gamma_{_{\text{LSE}}}^2 B_{h}) \nonumber \\
& = & \gamma_{_{\text{LSE}}}^2 \arg \max_{h \in \mathbb{R}} (-0.5 |h| + B_{h}).
\end{eqnarray}
%where $B_h$ is the standard Brownian motion. 
%\newline
\indent Moreover, the conclusion of Theorem \ref{thm: lse2}(c) holds  under (A1), (A3)-(A7) and (SNR1) with
%Moreover by (\ref{eqn: msethm2c}),
%\begin{eqnarray}
$C(c^2 (h+1)) - C(c^2 h) = \gamma_{_{\text{LSE}}}^{*} W_h$ where $W_h \stackrel{\text{i.i.d.}}{\sim} \mathcal{N}(0,1)$. 
%\end{eqnarray}
\end{remark}

\begin{remark} \label{rem: lserem2}
Suppose $||\mu_1 - \mu_2||_2 \to c>0$,  $\sup_{k} |\sigma_{1k} - \sigma_{2k}| \to 0$ and  $\mathcal{K}_0$ is the empty set. Then, 
in Theorem \ref{thm: lse2}$(c)$, $A(h) = 0\ \forall h \in c^2\mathbb{Z}$, $c_1 = c$ and $ \gamma_{_{\text{L,LSE}}}^{*} = \gamma_{_{\text{R,LSE}}}^{*} = \gamma_{_{\text{LSE}}}^{*}$. Further, suppose (A1), (A3)-(A7) and (SNR1) hold. % and $B_h$ denotes the standard Brownian motion. 
Then
\begin{eqnarray}
n||\mu_1 - \mu_2||_2^2 (\hat{\tau}_{n,\text{LSE}} -\tau_n) & \stackrel{\mathcal{D}}{\to}&  \arg \max_{h \in c^2\mathbb{Z}} (-0.5c^2|h| + \gamma_{_{\text{LSE}}}^{*} B_h) \nonumber \\
&\stackrel{d}{=} &  \arg \max_{h \in c^2 \mathbb{Z}} (-0.5c^2|h| + c^{-1}\gamma_{_{\text{LSE}}}^{*} B_{c^2h}) \nonumber \\
&=& c^2 \arg \max_{h \in  \mathbb{Z}} (-0.5|h| + c^{-1}\gamma_{_{\text{LSE}}}^{*} B_{h}) \nonumber \\
&\stackrel{d}{=} & \gamma_{_{\text{LSE}}}^{*2} \arg \max_{h \in \mathbb{Z}} (-0.5 |h| + B_{h}).\ \ \ 
\end{eqnarray}
In other words,
%\begin{eqnarray}
$n (\hat{\tau}_{n,\text{LSE}} -\tau_n)  \stackrel{\mathcal{D}}{\to} c^{-2}\gamma_{_{\text{LSE}}}^{*2} \arg \max_{h \in \mathbb{Z}} (-0.5 |h| + B_{h}).$
%\end{eqnarray}
\end{remark}

\begin{remark}\label{rem: lserem3}
Suppose $c_1^2 = \lim \sum_{k \in \mathcal{K}_n} (\mu_{1k} - \mu_{2k})^2 =0$ and (A1) holds. Then, (A6) and (A7) hold and $\gamma_{_{\text{L,LSE}}}^{*} = \gamma_{_{\text{R,LSE}}}^{*} = 0$. Hence,  if further (A4), (A5),  (SNR1) hold and $||\mu_1 - \mu_2|| \to c>0$, then
%\begin{eqnarray}
$n||\mu_1 - \mu_2||_2^2 (\hat{\tau}_{n,\text{LSE}} -\tau) \stackrel{\mathcal{D}}{\to} \arg \max_{h \in c^2 \mathbb{Z}} A(h)$, where $A(h)$ is given in (\ref{eqn: msethm2a}). %\nonumber 
%\end{eqnarray}
\end{remark}

\subsection{Sufficient conditions for (A2) and (A6)}
Conditions (A2) and (A6) which guarantee the existence of certain limits are not satisfied without some restrictions on the mean and variance parameters in the panel data model. For example, these limits may fail to exist if the variance parameters $\{\sigma_{ik}(n)\}$ oscillate over $n$. The examples in 
Section 1 of the Supplementary file % Examples  
%\ref{counter1}, \ref{counter2} and \ref{counter3} 
examine concrete scenarios where the limits do not exist. 
\newline
\indent Propositions \ref{prop: 1} and \ref{prop: 2}, given below,  provide sufficient conditions for (A2) and (A6) to hold.  Their proofs  are respectively given in Sections \ref{subsec: prop1} and \ref{subsec: prop2}. 

\begin{prop} \label{prop: 1}
(A2) holds if (a)-(e)  described below,  are satisfied for some $\mu_{ik}^{*} \in \mathbb{R}$ and $\sigma_{ik}^{*}>0$, which are free of $n$.
\vskip 3pt
\noindent (a) $\sup_{k \in \mathcal{K}_n} |\sigma_{ik}^2(n)-\sigma_{ik}^{*2}| \to 0$ for all $i =1,2$.
\vskip 3pt
\noindent (b) $\sup_{k} |\mu_{ik}(n) - \mu_{ik}^{*}| = O((m(n))^{-1/2} ||\mu_1(n)-\mu_2(n)||_2)$ for all $i=1,2$.
\vskip 3pt
\noindent (c) $\sum_{k=1}^{m(n)} | \mu_{1k}(n) - \mu_{2k}(n)| = o(\sqrt{m(n)} ||\mu_1(n)-\mu_2(n)||_2)$.
\vskip 3pt
\noindent (d) $\frac{||\mu_{1}(n+1)-\mu_{2}(n+1)||}{||\mu_{1}(n)-\mu_{2}(n)||} \to 1$.
 \vskip 3pt
\noindent (e) $\frac{m(n+1)}{m(n)} \to 1$.
\vskip 3pt
\noindent  Suppose $\sigma_{ik}^2(n) = \sigma_{i}^2(n)$  $\forall k,n \geq 1$ and $i=1,2$. Then (A2) holds if $\sigma_i^2(n) \to \sigma_i^2$ as $n \to \infty$.  Moreover, if $\sigma_{ik}^2(n) = \sigma_{i}^2$  for all $k,n \geq 1$ and $i=1,2$, then (A2) is always satisfied.
\end{prop}

\begin{prop} \label{prop: 2}
(A6) holds if (a), (e)  described in Proposition \ref{prop: 1} and (f), (g)  described below, are satisfied for some $\mu_{ik}^{*} \in \mathbb{R}$, %and $\sigma_{ik}^{*}>0$, 
which are free of $n$.
%Sufficient conditions for the existence of the above limits are (a)-(d) as described below.
\vskip 3pt
\noindent (f) $\sup_{k \in \mathcal{K}_n} |\mu_{ik}(n) - \mu_{ik}^{*}| = O(m(n)^{-1/2})$ for all $i=1,2$.
\vskip 3pt
\noindent (g) $\sum_{k \in \mathcal{K}_n} | \mu_{1k}(n) - \mu_{2k}(n)| = o(\sqrt{m(n)})$.
 \end{prop}
% Examples \ref{counter1}, \ref{counter2} and \ref{counter3} show that 
Conditions (a), (c) in Proposition \ref{prop: 1} and, Conditions (a), (g)  in Proposition \ref{prop: 2} seem to be necessary for (A2) and (A6) to hold. It is justified by the examples given in Section 1 of the Supplementary file. % \ref{counter1}-\ref{counter3}.
%{\bf Subsume these sentences in an appropriate remark following the propositions.} In addition to (a), (b) and (f), we also need some technical conditions.  (c) and (g) provide appropriate growth  rate of the $L_1$-norm and,  (d) and (e) control over the growth of  $||\mu_1 - \mu_2||_2$ and $m(n)$.  
%\newline
\vskip 3pt
\indent The following remark is immediate by the mean value theorem. 
\begin{remark} \label{rem: prop}
Condition (a) is implied by (b) or (f) in the above propositions in the presence of a smooth mean-variance relationship, say, if for all $k,n \geq 1$ and $i=1,2$, $\sigma_{ik}^2(n) = g(\mu_{ik}(n))$, where $g(\cdot)$ is a continuous function with bounded first derivative.
\end{remark}

\subsection{Illustrative Examples}

We discuss selective examples, which are widely encountered in practice. %, where Assumptions (A1)-(A7) are satisfied.  For all the following examples, we assume (A4) and (SNR1) hold. 
Suppose (SNR1) holds for all the following examples. Conditions (b)-(e) and (e)-(g) in Propositions \ref{prop: 1} and \ref{prop: 2} are respectively assumed for Regimes (b) and (c) (after replacing $\{\mu_{ik}: k \geq 1, i=1,2\}$ by the analogous mean parameters). Further, (f) in Proposition \ref{prop: 2} implies (A7). For all the examples below, we also assume (A4)  for regime (b). 

%Moreover, assume (b)-(e) in Proposition \ref{prop: 1} after replacing ??. Assume ??. 

\begin{Example} \label{example: berlse}
\textbf{Bernoulli data}. A random variable $X$ follows a $\mbox{Ber}(p)$ distribution if $P(X=1) = 1- P(X=0)= p$.  Suppose  $\{X_{kt}\}$ are independently distributed and for all $k \geq 1$, 
\begin{align}
X_{kt} \sim \mbox{Ber}(p_{1k}(n))I(t \leq n\tau_n) + \mbox{Ber}(p_{2k}(n))I(t > n\tau_n) %,\ \ %p_k(n) \neq q_k(n)\ \text{for at least one $k$}. 
\label{eqn: berlse1}
\end{align}
and $\ p_{1k}(n) \neq p_{2k}(n)$ for at least one $k$.  Then,  (A1) is satisfied and 
the conclusions of Theorems \ref{thm: lse1} and \ref{thm: lse2}(a) hold for the model in  (\ref{eqn: berlse1}). 
\newline
 \indent Now, (A3) is satisfied if  for some $0 < C <1$, we have 
\begin{align} \label{eqn: berlse2}
0<C < p_{ik}(n) < 1-C<1\ \ \forall k,n \geq 1\ \text{and}\ i=1,2.
\end{align}
Also, (A5) holds if for some constants $\{p_{ik}^{*}: k \in \mathcal{K}_0, i=1,2\}$ and as $n \to \infty$,  
\begin{eqnarray} \label{eqn: extraber}
\tau_n \to \tau^{*}\ \text{and}\ \ p_{ik}(n) \to p_{ik}^{*},\ \ \forall k \in \mathcal{K}_0\ \text{and}\ i=1,2.
\end{eqnarray}
Conditions (b)-(e) and (e)-(g) in Propositions \ref{prop: 1} and \ref{prop: 2}, respectively, imply (A2) and (A6). By Remark \ref{rem: prop}, Condition (a) is satisfied by Conditions (b) or (f) as
in this example, $\sigma_{ik}^{2}(n) = p_{ik}(n)(1-p_{ik}(n)) = g(p_{ik}(n)) = g(\mu_{ik}(n))$ where $g(x) = x(1-x)$  is continuous with bounded first derivative for all $0 < x <1$.  %Therefore, %Then by Remark \ref{rem: prop}, Condition (a) is satisfied by Conditions (b) or (f). 
\newline
\indent Therefore, under (\ref{eqn: berlse2}), the conclusion of Theorem \ref{thm: lse2}(b)   holds for the model in (\ref{eqn: berlse1}). Moreover, the conclusion of Theorem \ref{thm: lse2}(c)   holds for  (\ref{eqn: berlse1}) if we further assume (\ref{eqn: extraber}).

%WE NEED ADDITIONAL DISCUSSION AND CONDITIONS FOR THE EXISTENCE OF THE GAMMA LIMITS.
\end{Example}

\begin{Example} \label{example: poilse}
\textbf{Poisson data}. A discrete random variable $X$ is $\mbox{Poi}(\lambda)$ distributed, if for all $x = 0,1,2,\ldots$, we have $P(X=x) = \frac{1}{x!}e^{-\lambda}\lambda^x$.  Suppose  $\{X_{kt}\}$ are independently distributed and for all $k \geq 1$, 
\begin{align}
X_{kt} \sim \mbox{Poi}(\lambda_{1k}(n))I(t \leq n\tau_n) + \mbox{Poi}(\lambda_{2k}(n))I(t > n\tau_n)  \label{eqn: poilse1}
\end{align}
and  $\lambda_{1k}(n) \neq \lambda_{2k}(n)$ for at least one $k$.  Suppose for some $C_1,C_2>0$, 
\begin{align} \label{eqn: poilse2}
C_1 < \inf_{i,k,n} \lambda_{ik}(n) \leq \sup_{i,k,n} \lambda_{ik}(n) < C_2.
\end{align}
The last inequality in (\ref{eqn: poilse2}) implies (A1). Therefore,  the conclusions of Theorems \ref{thm: lse1} and \ref{thm: lse2}(a)  hold for  the model in  (\ref{eqn: poilse1}).
\newline
\indent (A3) is satisfied if the first inequality in (\ref{eqn: poilse2}) holds. Conditions (b)-(e) and (e)-(g) in Propositions \ref{prop: 1} and \ref{prop: 2} respectively imply (A2) and (A6).  %Moreover, as mean and variance are equal in this example, Condition (a) is satisfied by Conditions (b) or (f). 
Also, (A5) holds if (\ref{eqn: extraber}) is satisfied after replacing $p$ by $\lambda$.
\newline
\indent Hence,  under  (\ref{eqn: poilse2}), the conclusion of Theorem \ref{thm: lse2}(b)  holds for the model in (\ref{eqn: poilse1}).  Also the conclusion of Theorem \ref{thm: lse2}(c)  holds for  (\ref{eqn: poilse1}) if we further assume (\ref{eqn: extraber}) with $p$ replaced by $\lambda$.
\end{Example}

\begin{Example} \label{example: normallse}
\textbf{Normal data}.  $X$  follows a $\mathcal{N}(\mu,\sigma^2)$,  if its probability density function is 
\begin{align}
f_X(x) &= \frac{1}{\sigma\sqrt{2\pi}}e^{-\frac{1}{2}\left( \frac{x-\mu}{\sigma}\right)^2}\ \ \forall\ x, \mu \in \mathbb{R}\ \mbox{and}\ \ \sigma >0. \label{eqn: normallsea} 
\end{align}
Suppose $\{X_{kt}\}$ are independently distributed and for all $k \geq 1$,
\begin{align}
X_{kt} \sim \mathcal{N}(\mu_{1k}(n),\sigma_{1k}^2(n))I(t \leq n\tau_n) + \mathcal{N}(\mu_{2k}(n),\sigma_{2k}^2(n))I(t > n\tau_n),  
\label{eqn: normallse1}
\end{align}
 $\mu_{1k}(n) \neq \mu_{2k}(n)$ for at least one $k$ and for some $C>0$,
%\begin{eqnarray} \label{eqn: norlseex1}
$\sup_{k,n} \sigma_{ik}^2(n) < C,\ \forall i=1,2$.
%\end{eqnarray}
\newline
\indent Then, (A1) is satisfied and the conclusions of Theorems \ref{thm: lse1} and \ref{thm: lse2}(a)  hold for  (\ref{eqn: normallse1}). 
\newline
\indent Next, suppose Condition (a) in Proposition \ref{prop: 1} is satisfied. Then, by Propositions \ref{prop: 1} and \ref{prop: 2}, (A2) and (A6) hold.  Also suppose (A3) holds. 
\newline
%(A2) is satisfied by Conditions (a)-(e) in Proposition \ref{prop: 1}  and hence 
\indent Under the above assumptions, the conclusion of Theorem \ref{thm: lse2}(b)  holds for the model in (\ref{eqn: normallse1}).  Moreover,  Conditions (a)-(e)  can be relaxed if $\sigma_{ik}^2(n) = \sigma_i^2(n)$ and $\sigma_i^2(n) \to \sigma_i^2$ or $\sigma_i^2(n) = \sigma_i^2$ for all $k,n \geq 1$ and $i=1,2$. 
\newline
\indent Additionally, assume that (\ref{eqn: extraber}) holds after replacing $p$ by $\mu$ and $\sigma$.  Then, (A5) is satisfied and the conclusion of Theorem \ref{thm: lse2}(c)  holds for the model in (\ref{eqn: normallse1}).

%Suppose $\sigma_{ik}^2(n) = \sigma^2$ for all $k,n \geq 1$ and $i=1,2$. Then 

%Then, the conclusion of Example \ref{example: poilse} hold for the model (\ref{eqn: normallse1}) if (\ref{eqn: poilsea})-(\ref{eqn: poilse3}) hold by replacing $\lambda_{ik}(n)$ and $\lambda_{ik}$ respectively by $\sigma_{ik}^2(n)$ and $\sigma_{ik}^2$. 
\end{Example}

\subsection{Extensions to Time Dependent Data}\label{lse:time-dependent}
%{\bf Comment: We need to flesh out this pre-amble. Also, why do we start with this specific kind of stationarity and dependence. Need to say something about these being in some sense minimalistic assumptions we require to establish our results and verify them in interesting special cases. It may not be a bad idea to drop a reference to the fact that we are going to provide a treatment of the dependent model considered in Bai's paper later in the section. We also need to mention what we are relegating to the supplement.} 
%\newline
Next, we introduce dependence over the index $t$.  We consider stationary time series models, since they are widely encountered in applications.
There are different notions of stationarity established in the literature, including weak, strict, $r$-th order $(r > 2)$  and moment stationarity. 
In this section, we employ the latter two types of stationarity, defined below.
%For this task, we need the following two notions of stationarity.
 
\noindent {\bf Definition:} A process $\{Y_t: t \in \mathbb{Z}\}$ is called centered  \textit{$4$-order stationary} if the following statement holds.\\
\noindent \textbf{(S1)} For all $t,t_1,t_2,\ldots,t_r \geq 1$ and $r=1,2,3,4$,  $$E(Y_t) = 0, \ E|Y_t|^r < \infty$$  and $\text{Cum}(X_{kt_1}, X_{kt_2},\ldots,X_{kt_r}) $ depends only on the lags $t_2-t_1,t_3-t_1, \ldots, t_r-t_1$.  

\noindent A process $\{Y_t: t \in \mathbb{Z}\}$ is called centered \textit{moment stationary} if (S1) holds for all $r \geq 1$. 

\vskip 10pt
 \noindent To introduce dependence, we assume that for all $k$, $\{X_{kt}\}$ is driven by a stationary process. \\
\noindent 
\textbf{(D1)} For each $k$ and $t$, 
\begin{eqnarray}
X_{kt} = \mu_{1k}I(t \leq n\tau_n) + \mu_{2k}I(t> n\tau_n) + Y_{kt}
\end{eqnarray}
where for each $k$, $\{Y_{kt}: t \in \mathbb{Z}\}$ is a centered $4$-order stationary process and not observable.  Also $\{Y_{kt}\}$ are independent over $k$.  In general, $\{\mu_{ik}: k \geq 1, i=1,2\}$ may depend on the sample size $n$, whereas in this section, $\{Y_{kt}\}$ do not depend on $n$. 

\noindent  We also assume that the cumulants are summable.  \vskip 3pt
\noindent \textbf{(D2)} $\displaystyle \sum_{\stackrel{t_i \in \mathbb{Z}}{1 \leq i \leq r}} \sup_{k}|\text{Cum}(X_{k1}, X_{k(t_1+1)},X_{k(t_2+1)},\ldots, X_{k(t_r+1)})| < \infty,\ \ \forall r=1,2,3.$
\vskip 3pt
\noindent Note that if $\{Y_{kt}\}$ are all independent, then  (D2)  implies (A1).  Note that (D1) and (D2) are the minimal assumptions required to establish our results. Later in Remarks \ref{rem: uncorrelated}-\ref{rem: both} and in Example \ref{example: dep3},  we shall establish that these assumptions are satisfied by a large class of stationary processes. 
\vskip 5pt
Given assumptions (D1), (D2) and (SNR1), we can then establish the following result. Its proof is similar to the proof of Theorem \ref{thm: lse1} and is given in the supplement. % whose proof is provided in Section \ref{subsec: lse3}.
\begin{theorem} \textbf {Least squares convergence rate for dependent data}. \label{thm: lse3}
Suppose (D1), (D2) and (SNR1) hold. Then, 
\begin{eqnarray}
n||\mu_1 - \mu_2||_2^2 (\hat{\tau}_{n,\text{LSE}}-\tau_n) = O_{P}(1).
\end{eqnarray}
\end{theorem}

Next, we establish the analogue of Theorem \ref{thm: lse2}  for the time dependent setting. 

In Regime (a): $||\mu_1 - \mu_2||_2 \to \infty$, the conclusion of Theorem \ref{thm: lse2}(a) continues to hold under (D1), (D2) and (SNR1). Additional assumptions are required to obtain non-degenerate asymptotic distributions under the other two regimes.

Due to  dependence amongst observations, we need  the following stronger assumption. \\
\noindent \textbf{(D3)} For each $k \geq 1$, $\{Y_{kt}\}$ is a centered moment stationary process and 
 $$\sum_{t_1,t_2,\ldots, t_r = -\infty}^{\infty} \sup_{k}|\text{Cum}(X_{k1}, X_{k(t_1+1)},X_{k(t_2+1)},\ldots, X_{k(t_r+1)})| < \infty,\ \ \forall r \geq 1.$$

Later in Remarks \ref{rem: uncorrelated}-\ref{rem: both} and in Example \ref{example: dep3},  we shall see (D3)  being satisfied by many stationary processes.  For a linear process with i.i.d. innovations, this assumption can be relaxed using an appropriate truncation on the innovation process; see Example \ref{rem: lsema}.  %(??? I have not tried yet)

\noindent \textbf{Regime (b): $||\mu_1 - \mu_2||_2 \to 0$, assumptions}.
Recall $\gamma_{_\text{L,LSE}}$ and $\gamma_{_\text{R,LSE}}$ from (\ref{eqn: gammaLRnew}). By (D1),   $\sigma_{1k}^{2} = \text{Var}(X_{k1}) = \text{Var}(Y_{k1}) = \text{Var}(Y_{kn(\tau_n +1)}) = \text{Var}(X_{kn(\tau_n +1)}) = \sigma_{2k}^2$ and hence $\gamma_{_\text{L,LSE}} = \gamma_{_\text{R,LSE}} = \gamma_{_{\text{LSE}}}$, as defined in (\ref{eqn: gammalse}). But as $\gamma_{_{\text{LSE}}}$ uses only the marginal distributions of $\{X_{kt}\}$,  it cannot provide us the asymptotic variance of $n||\mu_1 - \mu_2||_2^2 (\hat{\tau}_{n,\text{LSE}}-\tau_n)$ when $\{X_{kt}\}$ are dependent over $t$.  Next, we introduce the counterpart of $\gamma_{_{\text{LSE}}}$ which exploits dependency among $\{X_{kt}\}$. 
\vskip 3pt
 \noindent For all $h_1,h_2 \in  \mathbb{R}$, let $\gamma_{(h_1,h_2),\text{DEP,LSE}}$ equal
\begin{eqnarray}
  \lim  \sum_{t_1=0 \wedge [h_1 ||\mu_1-\mu_2||_2^{-2}] }^{0 \vee [h_1||\mu_1-\mu_2||_2^{-2}] }\ \  \sum_{t_2=0 \wedge [h_2 ||\mu_1-\mu_2||_2^{-2}] }^{0 \vee [h_2||\mu_1-\mu_2||_2^{-2}] } \left(\sum_{k=1}^{m}(\mu_{1k}-\mu_{2k})^2 \text{Cum}(X_{kt_1},X_{kt_2}) \right).
\end{eqnarray}
\noindent \textbf{(D4)} For all $h_1, h_2 \in \mathbb{R}$,$\ \gamma_{(h_1,h_2),\text{DEP,LSE}}$ exists and $\gamma_{(h_1,h_1),\text{DEP,LSE}} >0$. 
%comparison with Bai and others

We assume (D4) to obtain the asymptotic distribution of $n||\mu_1 - \mu_2||_2^2 (\hat{\tau}_{n,\text{LSE}}-\tau_n)$ when $||\mu_1-\mu_2||_2 \to 0$.    If $\{Y_{kt}\}$ are all independent/uncorrelated, then $\gamma_{(h_1,h_2),\text{DEP,LSE}} = \min(|h_1|,|h_2|) \gamma_{_{\text{LSE}}}^2$ and (D4) is equivalent to (A2) and (A3).  
Also it is easy to see that if (a)-(e) in Proposition \ref{prop: 1} are satisfied with
\begin{eqnarray}
\sigma_{ik}^2 (n) &=& ||\mu_1 - \mu_2||_2^{2}  \sum_{t_1=0 \wedge [h_1 ||\mu_1-\mu_2||_2^{-2}] }^{0 \vee [h_1||\mu_1-\mu_2||_2^{-2}] }\ \  \sum_{t_2=0 \wedge [h_2 ||\mu_1-\mu_2||_2^{-2}] }^{0 \vee [h_2||\mu_1-\mu_2||_2^{-2}] }  \text{Cum}(X_{kt_1},X_{kt_2})\ \forall k,n,i, \nonumber 
\end{eqnarray}
and for some $C>0$
\begin{eqnarray}
\inf_{k} \sum_{h=-\infty}^{\infty} \text{Cum}(X_{k1},X_{k(h+1)}) > C,
\end{eqnarray}
then (D4) holds.
\vskip 5pt
\noindent \textbf{Regime (c): $||\mu_1 - \mu_2||_2 \to c>0$, assumptions}. Recall the partition of $\{1,2,3,\ldots, m(n)\}$ into $\mathcal{K}_0$ and $\mathcal{K}_n = \mathcal{K}_0^c$, described in (\ref{eqn: partition}). We assume (A4) and the following assumption on $\mathcal{K}_0$.
\vskip 3pt
\noindent \textbf{(D5)} There exist $\tau^{*} \in (c^{*},1-c^{*})$ and $\mu_{ik}^{*} \in \mathbb{R}$ such that $\tau_n \to \tau^{*}$ and $\mu_{ik}(n) \to \mu_{ik}^{*}$ for all $k \in \mathcal{K}_0$ and $i=1,2$.
\vskip 3pt
\noindent Unlike (A5), here we do not require weak convergence of $\{Y_{kt}\}$, since  due to (D1) they do not depend on $n$.  If $\{Y_{kt}\}$ are independent, then (D5) is equivalent to (A5). \\
%\vskip 5pt
\indent Next, we consider assumptions on $\mathcal{K}_n$.  Recall $\gamma_{_{L,\text{LSE}}}^*$ and $\gamma_{_{R,\text{LSE}}}^*$ from  (\ref{eqn: gammastarLRnew}). Like Regime (b), here also 
$\gamma_{_{L,\text{LSE}}}^* = \gamma_{_{R,\text{LSE}}}^* = \gamma_{_{\text{LSE}}}^*$, as defined in (\ref{eqn: gammalse}) and we need to 
%Recall $\gamma_{_{\text{LSE}}}^*$ from (\ref{eqn: gammalse}). Now we
 introduce the counterpart of $\gamma_{_{\text{LSE}}}^*$ which exploits dependency among $\{X_{kt}\}$. \\
\indent For all $t_1,t_2 \in \mathbb{Z}$, define
\begin{eqnarray}
\gamma^{*}_{(t_1,t_2),\text{DEP,LSE}} = \lim \sum_{k \in \mathcal{K}_n} (\mu_{1k} -\mu_{2k})^2 \text{Cum}(X_{kt_1},X_{kt_2}).
\end{eqnarray}
Recall $c_1^2 = \lim \sum_{k \in \mathcal{K}_n^c} (\mu_{1k} - \mu_{2k})^2$. We assume (A7) and the following assumption. 
%\begin{align*}
%c_1^2 = \lim \sum_{k \in \mathcal{K}_n^c} (\mu_{1k} - \mu_{2k})^2.
%\end{align*}

\noindent \textbf{(D6)} $c_1$ exists.  For all $t_1,t_2 \in \mathbb{Z}$, $\gamma^{*}_{(t_1,t_2),\text{DEP,LSE}}$ exists and $\gamma^{*}_{(t_1,t_1),\text{DEP,LSE}} >0$.  
\vskip 5pt
\noindent  If $\{Y_{kt}\}$ are all independent/uncorrelated, then $\gamma^{*}_{(t_1,t_2),\text{DEP,LSE}} = \gamma_{_{\text{LSE}}}^{*2}I(t_1=t_2)$ and (D6) is equivalent to (A3) and (A6).  Moreover, it is easy to see that if (e)-(g) in Proposition \ref{prop: 2} are satisfied and  $\inf_{k} \text{Var}(X_{kt})>C>0$, then (D6)  holds. 

\vskip 5pt
%Recall that $c_1^2 = \lim \sum_{k \in \mathcal{K}_n} (\mu_{1k}-\mu_{2k})^2$. 
The following Theorem establishes the asymptotic distribution of 
$n||\mu_1 - \mu_2||_2^2 (\hat{\tau}_{n,\text{LSE}}-\tau_n)$ under appropriate dependence on $\{X_{kt}\}$. Its proof is given in Section \ref{subsec: lse4}. 

\begin{theorem} \textbf{Asymptotic distributions under temporal dependence}. \label{thm: lse4}
Suppose (D1), (D2), (SNR1) hold. Then the following statements hold.

\noindent {\bf (a)} If $||\mu_1-\mu_2||_2 \to \infty$, then $$\lim_{n \to \infty} P(\hat{\tau}_{n,\text{LSE}}=\tau_n) =1.$$

\noindent {\bf (b)} Further, if (D3) and (D4) hold and  $||\mu_1-\mu_2||_2 \to 0$, then
\begin{eqnarray}
n||\mu_1 - \mu_2||_2^2 (\hat{\tau}_{n,\text{LSE}}-\tau_n) \stackrel{\mathcal{D}}{\to}  \arg \max_{h \in \mathbb{R}} (-0.5|h| + B^{*}_h).
\end{eqnarray}
where for all $h_1,h_2,\ldots, h_r \in \mathbb{R}$ and $r \geq 1$, 
\begin{eqnarray}
(B_{h_1}^{*},B_{h_2}^{*},\ldots, B_{h_r}^{*}) \sim \mathcal{N}_{r}(0,\Sigma)\ \ \text{where}\ \  
\Sigma = ((\gamma_{(h_1,h_2),{\rm{DEP,MSE}}} ))_{1 \leq h_1,h_2 \leq r}.
\end{eqnarray}

\noindent {\bf (c)} Suppose (A4), (A7),  (D3), (D5) and (D6)  hold and $||\mu_1 -\mu_2||_2 \to c >0$, then
\begin{eqnarray}
n (\hat{\tau}_{n,\text{LSE}}-\tau_n) \stackrel{\mathcal{D}}{\to} \arg \max_{h \in \mathbb{Z}} (D^{*}(h) + C^{*}(h) + A^{*}(h))
\end{eqnarray}
where for each $h, t_1,t_2,\ldots,t_r \in \mathbb{Z}$ and $r \geq 1$,
\begin{eqnarray}
&& D^{*}(h) = -0.5  c_1^2 |h|,\ \ \ C^{*}(h) = \sum_{t=0 \wedge h}^{0 \vee h} W^{*}_t, \nonumber \\
 &&  (W_{t_1}^{*}, W_{t_2}^{*},\ldots, W_{t_r}^{*}) \sim \mathcal{N}_{r}(0,\Sigma^{*}),\ \ \Sigma^{*} = ((\gamma^{*}_{(t_1,t_2),\text{DEP,LSE}}))_{1 \leq t_1, t_2 \leq r}, \nonumber \\
 && A^*(h) = \sum_{t=0 \wedge h}^{0 \vee h} \sum_{k \in \mathcal{K}_0} \bigg[(Y_{kt} + (\mu_{2k}^{*}-\mu_{1k}^{*}){\rm{sign}}(h))^2 -Y_{kt}^2 \bigg]{\rm{sign}}(h).\  \nonumber 
%  y_k &\stackrel{d}{=}& X_{1k}I(h \leq 0) + X_{nk}I(h > 0).\ \ \ \ 
\end{eqnarray}
\end{theorem}

The following remarks and examples are direct consequences of  Theorem \ref{thm: lse4}.

\begin{remark} \label{rem: uncorrelated}
A process $\{Y_t\}$ is called white noise, if ${\rm{Cum}}(Y_{t},Y_{t^\prime}) =0\  \forall t \neq t^\prime$. Suppose $\{X_{kt}\}$ satisfy all the assumptions in Theorem \ref{thm: lse4}
and  for each $k \geq 1$, $\{Y_{kt}: t \in \mathbb{Z}\}$ is a white noise process. Then, in Theorem \ref{thm: lse4}, we can replace (D4) and (D6) by (A2), (A3) and (A3), (A6), respectively. Moreover, in this case, the asymptotic distributions in Theorem \ref{thm: lse4} are identical to those posited in Theorem \ref{thm: lse2}.

%Suppose (D1)  holds and  for each $k \geq 1$, $\{Y_{kt}: t \in \mathbb{Z}\}$ is a white noise. Then %(D4) and (D6) are satisfied if (A2) and (A6) hold respectively. Thus 
%the conclusions of Theorem \ref{thm: lse4} hold once we replace (D4) and (D6) by (A2), (A3) and (A3), (A6) respectively.  Moreover, in this case, the asymptotic distributions in Theorem \ref{thm: lse4} are identical to those posited in Theorem \ref{thm: lse2}.
%If $\{X_{kt}\}$ satisfy all the assumptions in Theorem \ref{thm: lse4} and suppose
%\begin{eqnarray}
%{\rm{Cum}}(X_{kt},X_{kt^\prime}) =0\ \ \ \forall k\ \ \text{and}\ \ t \neq t^\prime, \label{eqn: whnoise}
%\end{eqnarray}
%then the asymptotic distributions in Theorem \ref{thm: lse4} are identical to those posited in Theorem \ref{thm: lse2}. Any $m$-dependent moment stationary white noise process has the above displayed property and satisfy (D2) and (D3).  For more illustration see Examples \ref{example: dep1} and \ref{example: dep2}.  
\end{remark}

\begin{remark} \label{rem: mdeprem}
A process $\{Y_t\}$ is called $m$-dependent, if for all $1 \leq i \neq j \leq r$, $r \geq 2$ and $|t_i - t_j| > m$, $\{Y_{t_i}: 1 \leq i \leq r\}$ are independently distributed. Suppose (D1) and (SNR1) hold and for each $k \geq 1$, $\{Y_{kt}:  t \in \mathbb{Z}\}$ is an $m$-dependent process. Then, the infinite sum in (D2)  reduces to a finite sum. Hence (D2) and consequently the conclusion of Theorem \ref{thm: lse4}(a) hold if  $\sup_{k} E|Y_{kt}|^4 < \infty$.  Further, suppose $\sup_{k} E|Y_{kt}|^r < \infty\ \forall r \geq 1$. Then, (D3) is satisfied and the conclusions of Theorem \ref{thm: lse4}(b) and (c) hold, respectively under (D4) and,  (A4), (A7), (D5) and (D6).  
\end{remark}

\begin{remark} \label{rem: both}
Suppose (D1) and (SNR1) hold and for each $k \geq 1$, $\{Y_{kt}:  t \in \mathbb{Z}\}$ is an $m$-dependent white noise process. Then, by Remarks \ref{rem: uncorrelated} and \ref{rem: mdeprem}, the conclusion of Theorem \ref{thm: lse4}(a) holds if $\sup_{k} E|Y_{kt}|^4 < \infty$. Next, suppose  $\sup_{k} E|Y_{kt}|^r < \infty\ \forall r \geq 1$. Then, the conclusion of Theorem \ref{thm: lse4}(b) holds under (A2) and (A3). Also the conclusion of Theorem \ref{thm: lse4}(c) holds when (A3)-(A7) are satisfied. In this case, the asymptotic distributions in Theorem \ref{thm: lse4} are identical to those posited in Theorem \ref{thm: lse2}.
\end{remark}
Specific examples of $m$-dependent white noise error processes are given in Section $2$ of the Supplementary file. Next, we discuss an example which is neither an $m$-dependent, nor a white noise process.
\begin{Example} \textbf{Gaussian process}. \label{example: dep3}
 A process $\{Y_t: t \in \mathbb{Z}\}$ is called a centered stationary Gaussian process with covariance kernel $g(\cdot)$, if  $g: \mathbb{R} \to \mathbb{R}$ is a symmetric function with unique global maximum at $0$, $g(0) >0$ and for all $t,t_1,t_2 \in \mathbb{Z}$,
%\begin{align*}
$E(Y_t) = 0\ \ \text{and}\ \ \text{Cum}(X_{t_1}, X_{t_2}) = g(t_1-t_2) = g(|t_1-t_2|)$.
%\end{align*}
Suppose for all $k,t \geq 1$,
\begin{eqnarray} \label{eqn: gaup}
 X_{kt} &=& \mu_{1k}I(t \leq n\tau_n) + \mu_{2k}I(t>n\tau_n) + Y_{kt}
 \end{eqnarray}
 where $\mu_{1k} \neq \mu_{2k}$ for at least one $k$, $\{Y_{kt}\}$ are independent over $k$ and  for each $k$, $\{Y_{kt}: t \in \mathbb{Z}\}$ is a centered stationary Gaussian process with covariance kernel $g_k(\cdot)$ satisfying  
 \begin{align} \label{eqn: covkerd2}
 \sum_{h=-\infty}^{\infty} \sup_{k} |g_{k}(|h|)| < \infty.
 \end{align}
 Then, (D1) and (D2) are satisfied. Further, under  (SNR1) the conclusions of Theorems \ref{thm: lse3} and \ref{thm: lse4}(a) hold for (\ref{eqn: gaup}).  \\
 \indent As all cumulants of $\{X_{kt}\}$ of order more than $2$ are zero,  (D2) and (D3) are equivalent.  \\
 \indent The conclusion of Theorem \ref{thm: lse4}(b)  holds for (\ref{eqn: gaup}) under (SNR1) and (D4). (D4) holds if for some $C>0$, $\inf_{k} \sum_{h=-\infty}^{\infty} g_k (h) > C$ and (a)-(e) in Proposition \ref{prop: 1} are satisfied with
\begin{eqnarray} \label{eqn: gaupsig}
\sigma_{ik}^2 (n) &=& ||\mu_1 - \mu_2||_2^{2}  \sum_{\stackrel{t_i=0 \wedge [h_i ||\mu_1-\mu_2||_2^{-2}]}{i=1,2} }^{0 \vee [h_i||\mu_1-\mu_2||_2^{-2}] }\ %\  \sum_{t_2=0 \wedge [h_2 ||\mu_1-\mu_2||_2^{-2}] }^{0 \vee [h_2||\mu_1-\mu_2||_2^{-2}] } 
g_{k}(|t_1-t_2|)\ \forall k,n,i. \label{eqn: siggaup}
\end{eqnarray}
\indent The conclusion of \ref{thm: lse4}(c) holds for (\ref{eqn: gaup}) under (SNR1), (A4), (A7), (D5) and (D6). (D6) holds  if (e)-(g) in Proposition \ref{prop: 2} are satisfied and for some $C>0$, $\inf_k g_k(0) > C$.
\end{Example}
$\{Y_{kt}: t \in \mathbb{Z}\}$ in (D1) may not always  be an $m$-dependent or a Gaussian process.  This implies that (D3) may not always be satisfied, even if $\{Y_{kt}: t \in \mathbb{Z}\}$ is a moment stationary process.  In Section $2$ of the Supplementary file, we discuss a wide class of $\alpha$-mixing processes for which (D3) holds. 
%The following remark provides a wide class of processes for which (D3) holds though.

The next example relaxes Assumption (D3) and considers a weaker condition when $\{Y_{kt}\}$ is a linear process, as was considered in  \citet{bai2010common}.  Its proof is based on an appropriate truncation on the innovation process and is given in Section \ref{subsec: lsema}.

\begin{Example} \textbf{Linear Error Process.}\label{rem: lsema}
Suppose  $\{\varepsilon_{k,t}\}$ are independent and identically distributed over $t$ and independent over $k$ with mean $0$, variance $\sigma_{k\varepsilon}^2$ and $\sup_{k} E\varepsilon_{k,t}^4 < \infty$. %Moreover, $\{\varepsilon_{k,t}\}$ are independent over $k$.  
\noindent 
Suppose for each $k, t\geq 1$,
\begin{eqnarray} \label{eqn: lsedpma}
&& X_{kt} = \mu_{1k}I(t \leq n\tau_n) + \mu_{2k}I(t>n\tau_n) + Y_{kt}\ \ \text{where} \ \ \\ 
 && Y_{kt} = \sum_{j=0}^{\infty} a_{k,j} \varepsilon_{k,t-j}\ \ \text{and}\ \ 
  \sup_{k}\sum_{j=0}^{\infty} |a_{k,j}| < \infty.\ \ \ \  \nonumber 
\end{eqnarray}
Then,  (D1) and (D2) are satisfied and the following statements hold.
\vskip 5pt
\noindent (a) If (SNR1) holds and $||\mu_1 - \mu_2||_2 \to \infty$, then  $\lim_{n \to \infty} P(\hat{\tau}_{n,\text{LSE}}=\tau_n) =1$.
\vskip 5pt
\noindent (b) Further suppose (D4) holds for $||\mu_1 - \mu_2||_2 \to 0$ and  (A4), (A7), (D5) and (D6) hold for $||\mu_1-\mu_2||_2 \to c>0$. Then  the conclusions of Theorem \ref{thm: lse4}(b) and (c)   hold with
$$Cum(X_{kt_1},X_{kt_2}) = \sigma_{k\varepsilon}^2 \bigg(\sum_{j=0}^{\infty} a_{k, j}a_{k, j+|t_2-t_1|}\bigg)\ \ \forall t_1,t_2 \in \mathbb{Z}\ \text{and}\ k \geq 1.$$
%Note that (D1) and (D2) are satisfied by the model (\ref{eqn: lsedpma}).  Moreover,  
(D3)  holds if and only if all moments of $\{\varepsilon_{k,t}\}$ are  finite.  But as stated above, the asymptotic distributions in Theorem \ref{thm: lse4}  still hold for the model (\ref{eqn: lsedpma}), when $E|\varepsilon_{k,t}|^r = \infty$ for some $k \geq 1$ and $r \geq 5$. 
\end{Example}

\subsubsection{Connections to the results presented in \citet{bai2010common}} \label{rem: bai}
\noindent
{\bf (A)} Next, we compare the results previously established with those in the paper by \citet{bai2010common} that posited that
data $\{X_{kt}\}$ are generated according to model (\ref{eqn: lsedpma}) and considered the following assumptions.

\begin{enumerate}
\item
$\sup_{k} \sum_{j=0}^{\infty} j |a_{k,j}| < \infty$ 
\item
$m^{-1/2} \sum_{k=1}^{m}(\mu_{1k} - \mu_{2k})^2 \to \infty$ %\ \ OR\ \ 
\item
$||\mu_1 - \mu_2||_2 \to \infty\ \ \text{and}\ \ n^{-1}m\log( \log (n)) \to 0$
\end{enumerate}

The key result established in that paper is that assuming (1)-(3) we get
$$\lim_{n \to \infty} P(\hat{\tau}_{n,\text{LSE}}=\tau_n) =1.$$ 
Details are presented in Theorems $3.1$ and $3.2$ in \citet{bai2010common}.

In comparison,  we assume in Example \ref{rem: lsema} that $\sup_{k} \sum_{j=0}^{\infty}  |a_{k,j}| < \infty$ which is clearly weaker than (1).
Further, observe that assumptions (SNR1) and (2) above indicate two different regimes, since none of them implies the other one. % neither (2) implies (SNR1) nor the other way around.
Moreover, note that assumption (3) above is {\em stronger} than the posited (SNR1). Therefore, Example \ref{rem: lsema}(a) implies 
\citet{bai2010common}'s result under assumptions (1) and (3). 

\noindent
{\bf (B)} Recall the quantity $\gamma_{_{\text{LSE}}}^{*}$ from Remark \ref{rem: lserem2} and let $B_h$ denote the standard Brownian motion. 
Suppose $\{Y_{kt}\}$ in (\ref{eqn: lsedpma}) are uncorrelated,  $||\mu_1 - \mu_2||_2 \to c>0$, $n^{-1}m\log( \log (n)) \to 0$, and assumptions
(A3) and (A6) hold. Then, \citet{bai2010common} also established in Theorem $4.2$ that
\begin{eqnarray} \label{eqn: 24}
n (\hat{\tau}_{n,\text{LSE}} -\tau) \stackrel{\mathcal{D}}{\to}  c^{-2}\gamma_{_{\text{LSE}}}^{*2} \arg \max_{h \in \mathbb{Z}} (-0.5 |h| + B_{h}).
\end{eqnarray}

To derive (\ref{eqn: 24}),  one needs to establish the asymptotic normality of $\sum_{k=1}^{m}(\mu_{1k} - \mu_{2k})Y_{kt}$, presented at the end of the first column on page $90$ in \citet{bai2010common}. For doing so, assumptions need to be imposed on $\mu_1$ and $\mu_2$ in addition to $||\mu_1 - \mu_2||_2 \to c>0$, as we have already discussed in the current study around (\ref{eqn: asymnorlse})-(\ref{eqn: lyaplsedom}). % (see discussion around (\ref{eqn: ct123})-(\ref{eqn: ct1234})). 
However, such assumptions are missing in the presentation of \citet{bai2010common}. % did not mention those conditions.  
%For details see the proof of Theorem $4.2$  on page $90$ in \cite{bai2010common}. 

Finally, consider all the assumptions stated in the last paragraph.  Further, assume the weaker condition $mn^{-1} \to 0$ instead of $n^{-1}m\log( \log (n)) \to 0$.  Recall the sets $\mathcal{K}_0$ and $\mathcal{K}_n$ from (\ref{eqn: partition}).  By  Remarks \ref{rem: lserem2}, \ref{rem: uncorrelated} and Example \ref{rem: lsema}(b),  we additionally need (A7) and $\mathcal{K}_0$ as the empty set  for (\ref{eqn: 24}) to hold.  
Example \ref{rem: lsema}(b) provides a more general result for the model (\ref{eqn: lsedpma}) i.e. when $\{Y_{kt}\}$ are not necessarily uncorrelated.

\section{Maximum likelihood estimation of the common break model parameters} \label{sec: mle}

In this section, we discuss maximum likelihood estimation (MLE) for the change point $\tau_n$. As will become clear below, stronger assumptions will
be needed to establish both the rate and derive the asymptotic distribution of the MLE due to the possible lack of adequate smoothness of the likelihood function. \\
\indent The problem formulation is as follows: let $\{\mathbb{P}_{\lambda}: \lambda \in \Lambda \subset \mathbb{R}^{d}\}$ ($d$ being a finite positive integer) be a family of probability density/mass functions satisfying assumptions (B1)-(B11) described next.  \\
\indent Since $d$ is finite,  we define a sequence of  $d$-dimensional vectors or $d \times d$ matrices to be convergent if they converge entry wise. Binary operators such as $\leq , < , \geq$ and $>$ between two $d$-dimensional vectors or $d \times d$ matrices are also applicable in an entry wise manner. Modulus, power, exponential, log, expectation and variance functions also operate component wise.  Let $J_d$ and $1_d$ be respectively the $d \times d$ matrix and $d \times 1$ vector whose entries are all
equal to $1$.  Finally, we consider $\frac{\partial }{\partial \lambda} \log \mathbb{P}_{\lambda}(X)$  to be a column vector. 

Next, we postulate assumptions needed to establish the results. 

\noindent \textbf{(B1)}  Probability distributions in $\{\mathbb{P}_{\lambda}: \lambda \in \Lambda\}$  are distinct for different $\lambda$; i.e., $\mathbb{P}_{\lambda_1} = \mathbb{P}_{\lambda_2}$, if and only if $\lambda_1 = \lambda_2$.

\noindent \textbf{(B2)} The support of $\mathbb{P}_{\lambda}$ does not depend on $\lambda$. 

\noindent \textbf{(B3)} The parameter space $\Lambda$ contains an open set of which the true parameter value is an interior point. 

Suppose that for almost all $X$, $\mathbb{P}_{\lambda}(X)$ is differentiable with respect to $\lambda \in \Lambda$.  Further, suppose that
$X_1, X_2,\ldots, X_n$ are i.i.d. from $\mathbb{P}_{\lambda_0}$ for an unknown $\lambda_0 \in \Lambda$.  Then, the maximum likelihood estimator (MLE) of $\lambda_0$ is a root of the likelihood equation   
\begin{eqnarray} \label{eqn: likelihood equation}
\sum_{i=1}^{n}\frac{\partial }{\partial \lambda} \log \mathbb{P}_{\lambda}(X_i) = 0.
\end{eqnarray}

Note that the above set of assumptions is sufficient to establish that there is a (measurable) root $\hat{\lambda}_n$ of the likelihood equation (\ref{eqn: likelihood equation})  so that $\hat{\lambda}_n \stackrel{P}{\to} \lambda_0$.  It is easy to see that if there is a unique solution of (\ref{eqn: likelihood equation}) for almost all $X_1,X_2,\ldots,X_n,\ldots$ and for all sufficiently large $n$ (which may depend upon the sequence in consideration) and if $\Lambda$ is an open set, then this solution will be the MLE and also consistent for $\lambda_0$.  If there is more than one root of (\ref{eqn: likelihood equation}) which maximizes the joint log-likelihood $\sum_{i=1}^{n}\log \mathbb{P}_{\lambda}(X_i)$, then we select the consistent solution.  For more details see \cite{lehmann1998estimation}.  \\
\indent To obtain asymptotic results for $\hat{\lambda}_n$,  assumptions are needed to control the second derivative of the joint log-likelihood with respect to $\lambda$, given next. 

\noindent \textbf{(B4)} $\frac{\partial^2}{\partial \lambda^2}  \log P_\lambda (X)$ exists for all $\lambda \in \Lambda$ and almost everywhere in $X \sim P_a$, $a \in \Lambda$.

\noindent \textbf{(B5)} For some $0<C_1 \leq C_2<\infty$ and measurable function $G_2(\cdot)$ on $\mathbb{R}^{d \times d}$,
\begin{eqnarray} \label{eqn: sdv1}
0< C_1 G_2(x) \leq \sup_{\lambda \in \Lambda} \bigg| \frac{\partial^2}{\partial \lambda^2} \log \mathbb{P}_{\lambda}(x)\bigg| \leq C_2 G_2(x) < \infty\ \ \forall x.
\end{eqnarray}
 For some $0< \epsilon_1 \leq \epsilon_2 < \infty$ and $X \sim \mathbb{P}_{\lambda}$,
\begin{eqnarray} \label{eqn: sdv2}
0 < \epsilon_1 J_d \leq \inf_{\lambda \in \Lambda} E G_2(X) \leq \sup_{\lambda \in \Lambda} (E G_2^4 (X))^{1/4} \leq \epsilon_2 J_d <\infty.
\end{eqnarray}

Analogously to the discussion preceding assumption (A1) in Section  \ref{sec: lse},  we require control over the variance of an estimator to establish its probability convergence. For that
%This can be done by assuming finite $4$-th order moment of the variables involved in the estimator.   Thus 
we need $\sup_{\lambda \in \Lambda} EG_2^{4}(X) < \infty$ in (B5) and additionally the following assumption. 

 \noindent \textbf{(B6)}  For $X \sim P_\lambda$,
$\sup_{\lambda \in \Lambda} E\bigg[ \bigg( \frac{\partial}{\partial \lambda} \log P_\lambda (X) \bigg)^\prime \bigg( \frac{\partial}{\partial \lambda} \log P_\lambda (X) \bigg) \bigg]^4 < \infty$.
\vskip 5pt
\noindent %Later we shall see that if 
If $G_2(x) = C\ \forall x$, then (B6) (in Theorems \ref{thm: mle1g} and \ref{thm: mle2g}) can be relaxed to the weaker assumption given below. 
\vskip 5pt
 \noindent \textbf{(B7)}  For $X \sim P_\lambda$,
$\sup_{\lambda \in \Lambda} E\bigg[ \bigg( \frac{\partial}{\partial \lambda} \log P_\lambda (X) \bigg)^\prime \bigg( \frac{\partial}{\partial \lambda} \log P_\lambda (X) \bigg) \bigg]^2 < \infty$.
\vskip 5pt
\noindent The condition $G_2(x) = C\ \forall x$ holds for a wide class of probability distributions, such as the one parameter natural exponential family, which is 
examined in detail in Example \ref{example: expfam} below.
\vskip 5pt

Recall the posited setting based on data $\{X_{kt}:\ 1 \leq k \leq m,\ 1 \leq t \leq n\}$ and where for each $k \geq 1$, there is a common break $\tau_n$, so that
\begin{eqnarray}
X_{kt} \sim \mathbb{P}_{\theta_k(n)}I(t \leq n\tau_n) + \mathbb{P}_{\eta_k(n)}I(t > n\tau_n),\ \ \ \theta_k(n), \eta_k(n) \in \Lambda \subset \mathbb{R}^d,
\end{eqnarray}
and $\theta_k(n) \neq \eta_k(n)$ for at least some $k$.  For ease of presentation, we shall use $\theta_k$ and $\eta_k$ respectively for $\theta_k(n)$ and $\eta_k(n)$. \\
\indent The maximum likelihood estimator $\hat{\tau}_{n, \text{MLE}}$ is obtained as
\begin{eqnarray}
&& \hat{\tau}_{n, \text{MLE}} = \arg \max_{b \in (c^{*},1-c^{*})} \sum_{k=1}^{m} L_{k,n}(b)\ \ \text{where} \label{eqn: taumle} \\
&& L_{k,n}(b) = \frac{1}{n}\sum_{t=1}^{nb}\log \mathbb{P}_{\hat{\theta}_k(b)}(X_{kt}) + \frac{1}{n}\sum_{t=nb+1}^{n}\log \mathbb{P}_{\hat{\eta}_{k(b)}}(X_{kt}), \nonumber \\
&& \sum_{t=1}^{nb} \frac{\partial }{\partial \theta} \log \mathbb{P}_{{\theta}}(X_{kt}) \bigg|_{\theta = \hat{\theta}_k (b)} = \sum_{t=nb+1}^{n} \frac{\partial }{\partial \theta}\log \mathbb{P}_{\hat{\eta}_{k(b)}}(X_{kt}) \bigg|_{\eta = \hat{\eta}_k(b)} = 0. \nonumber
\end{eqnarray}
%such that $\hat{\theta}_{k}(b)$ and $\hat{\eta}_{k}(b)$  are respectively consistent for $\theta_k$ and $\eta_k$. 
 Existence of $\{\hat{\theta}_{k}(b), \hat{\eta}_{k}(b) \}$ is guaranteed by Assumptions (B1)-(B3). 

\noindent
{\bf Rate of convergence for $\hat {\tau}_{n,\text{MLE}}$}. %To investigate the asymptotic properties of $\hat{\tau}_{n,\text{MLE}}$, 
To establish the result, we typically need to deal with the second derivative of the joint log-likelihood at the random points $\hat{\theta}_{k}(b)$ and $\hat{\eta}_{k}(b)$ or intermediate points. This can be handled if $\hat{\theta}_k(b), \hat{\eta}_{k}(b) \in \Lambda$  for all $1 \leq k \leq m$. The assumptions (B8) and (B9), below, ensure this. As we deal simultaneously with all $k \leq m$, we also require an appropriate growth rate for $m = m(n)$. When $||\theta - \eta||_2 \to \infty$, we assume $\log m(n) = o(n)$. For the other two regimes i.e. when $||\theta - \eta||_2 \to 0$ or $c > 0$, (SNR2) (stated below) implies $m=o(\sqrt{n})$, a substantially slower rate of growth than $\log m(n) = o(n)$. \\
% This is given in (B10). Also note that if $\frac{\partial^2}{\partial \lambda^2} \log \mathbb{P}_{\lambda}(x) = C$ for all $\lambda$ and $x$, then we do not need the following assumptions.
\indent A centered  $d$-dimensional random vector $X$ is called marginally sub-Gaussian, if for all $\epsilon >0$ and some $C_1,C_2>0$,  $P(|X| \geq \epsilon 1_d ) \leq C_1e^{-C_2 \epsilon^2}1_d$, or equivalently if there exists $b \in \mathbb{R}$ such that $E(e^{tX}) \leq e^{0.5 t^2 b^2} 1_d$ for all 
$t \in \mathbb{R}$.  This definition of sub-Gaussian also holds for a $d \times d$ random matrix $X$, if we replace $1_d$ by $J_d$. 

\noindent \textbf{(B8)}  For all $X \sim \mathbb{P}_\lambda$, $\frac{\partial}{\partial \lambda} \log \mathbb{P}_{\lambda}(X)$ is  a marginally sub-Gaussian random variable.

\noindent \textbf{(B9)}  For all $X \sim \mathbb{P}_\lambda$, %$\frac{\partial^2}{\partial \lambda^2} \log \mathbb{P}_{\lambda}(X) + I(\lambda)$ 
$G_2(X) - E(G_2(X))$ is  a marginally sub-Gaussian random variable.
\vskip 5pt
\noindent Let $||\theta -\eta||_2^2 = \sum_{k=1}^{m} ||\theta_k - \eta_k||_2^2$. 
  %\theta = (\theta_1,\theta_2,\ldots,\theta_m)^\prime$,  $\eta = (\eta_1,\eta_2,\ldots, \eta_m)^\prime$. 
In this section, we consider the following  signal-to-noise condition.

\noindent \textbf{(SNR2)} $\sqrt{n}m^{-1}||\theta -\eta||_2^2 \to \infty$ as $n \to \infty$.

Given these assumptions the following rate result for $\hat{\tau}_{n,\text{MLE}}$ can be established, whose proof is 
given in Section \ref{subsec: mle1proof}.

\begin{theorem} \textbf{MLE convergence rate}. \label{thm: mle1g}
Suppose (B1)-(B6), (B8), (B9),  (SNR2) hold and $\log m(n) = o(n)$. Then,
\begin{eqnarray} \label{eqn: msethm1g}
n||\theta - \eta||_2^2 (\hat{\tau}_{n,\text{MLE}}-\tau_n) = O_{P}(1).
\end{eqnarray}
%Further, if  $G_2(x)$ in (B5) does not depend on $x$, i.e. it is a constant function, then (B6) can be relaxed to (B7). 
\end{theorem}

\begin{remark}\label{remark:Gaussian-case}
Suppose  that $\frac{\partial^2 }{\partial \lambda^2} \log \mathbb{P}_{\lambda} (x) = -\Sigma$ for all $\lambda, x$ and that for some positive definite matrix $\Sigma \in \mathbb{R}^{d \times d}$, which does not depend on $\lambda$ and $x$. This is equivalent to positing 
that for each $k \geq 1$,
%\begin{eqnarray}
$X_{kt} \sim \mathcal{N}_d(\theta_k, \Gamma)I(t \leq \tau_n) + \mathcal{N}_d(\eta_k, \Gamma)I(t > \tau_n)$, $\theta_k \neq \eta_k$ for at least one $k$
%end{eqnarray}
and for some known $d \times d$ positive definite matrix $\Gamma$. Then,  the result in (\ref{eqn: msethm1g}) continues to hold under the weaker assumption (B7) (or equivalently (A1) in Section \ref{sec: lse}) and (SNR1). 
%\end{theorem}
Therefore, for the Gaussian likelihood, the rate of convergence of the maximum likelihood change point estimate can be established 
under weaker assumptions.
\end{remark}

\noindent
{\bf Asymptotic distribution of $\hat{\tau}_{n,\text{MLE}}$}.
Next, we present results regarding the asymptotic distribution of $n||\theta - \eta||_2^2 (\hat{\tau}_{n,\text{MLE}}-\tau_n)$. Under the same assumptions  as in Theorem \ref{thm: mle1g}, $(\hat{\tau}_{n,\text{MLE}}-\tau_n)$ is degenerate at $0$ if $||\theta - \eta||_2 \to \infty$.  
Analogously to the results in Section \ref{sec: lse}, additional assumptions are needed to obtain the asymptotic distribution for the
cases $||\theta - \eta||_2 \to 0$ or $c>0$, given next. \\
\indent For $X \sim \mathbb{P}_{\lambda}$, define
\begin{eqnarray} \label{eqn: I1}
I(\lambda) = E \bigg[\left(\frac{\partial}{\partial \lambda} \log \mathbb{P}_{\lambda}(X) \right) \left(\frac{\partial}{\partial \lambda} \log \mathbb{P}_{\lambda}(X) \right)^\prime \bigg]\ \ \forall \lambda \in \Lambda,
\end{eqnarray}
which exists by (B6) or (B7). Moreover, by (B4) and (B5), for some $C_1,C_2>0$
\begin{eqnarray} \label{eqn: I2}
 I(\lambda) = - E \left(\frac{\partial^2}{\partial \lambda^2} \log \mathbb{P}_{\lambda}(X) \right)\ \forall \lambda \in \Lambda,\ \  
 0< C_1 J_d \leq \inf_{\lambda \in \Lambda} I(\lambda) \leq \sup_{\lambda \in \Lambda} I(\lambda) \leq C_2 J_d < \infty. \label{eqn: I3} 
\end{eqnarray}
 Recall the sets $\mathcal{K}_0$ and $\mathcal{K}_n = \mathcal{K}_0^c$ in (\ref{eqn: partition}).  
\noindent   Let,
\begin{eqnarray} \label{eqn: gammamleg}
\gamma_{_{\text{MLE}}}^2 = \lim \frac{\sum_{k=1}^{m}(\theta_k - \eta_k)^\prime I(\theta_k) (\theta_k - \eta_k)}{||\theta-\eta||_2^2}\ \text{and}\ %\nonumber \\ 
\gamma_{_{\text{MLE}}}^{*2} = \lim \sum_{k \in \mathcal{K}_n}(\theta_k - \eta_k)^\prime I(\theta_k) (\theta_k - \eta_k). \nonumber 
\end{eqnarray}
Note that  (\ref{eqn: I3}) implies $\gamma_{_{\text{MLE}}} >0$. Moreover, $\gamma_{_{\text{MLE}}}^* >0$ if and only if $\lim \sum_{k \in \mathcal{K}_n} ||\theta_k - \eta_k ||_2^2 >0$.  Existence of $\gamma_{_{\text{MLE}}}$ and $\gamma_{_{\text{MLE}}}^*$ are required respectively for $||\theta - \eta||_2 \to 0$ and $c>0$.  However, this is guaranteed by the conditions in Propositions \ref{prop: 1} and \ref{prop: 2} when $\mu_{1k}$, $\mu_{2k}$, $\sigma_{1k}^2$ and $\sigma_{2k}^2$ are respectively replaced by $\theta_k$, $\eta_k$,  $I(\theta_k)$ and $I(\eta_k)$. 
%Here we assume existence of $\gamma_{_{\text{MLE}}}$ and $\gamma_{_{\text{MLE}}}^*$ respectively for $||\theta - \eta||_2 \to 0$ or $c>0$. 
As $\mathcal{K}_0$ may not be the empty set under $||\theta - \eta||_2 \to c>0$,  we consider  (A4) and (B10), given below,  on $\mathcal{K}_0$.
 \vskip 5pt
 \noindent \textbf{(B10)} (i) There is $\tau^{*} \in (c^{*},1-c^{*})$ such that $\tau_n \to \tau^{*}$. \\
 (ii) $\mathbb{P}_{\lambda}(x)$ is continuous in both $\lambda$ and $x$. \\
 %Moreover, 
 (iii) Let $\{X_{ik}^{*}: k \in \mathcal{K}_0, i=1,2\}$ be a collection of independent random variables such that for all $k \in \mathcal{K}_0$ and $0<f<1$,
$X_{k\lfloor nf \rfloor} \stackrel{\mathcal{D}}{\to} X_{1k}^{*}I(f \leq \tau^{*}) + X_{2k}^{*}I(f > \tau^{*})$. \\
% $t \geq 1$, $X_{kt} \stackrel{\mathcal{D}}{\to} X_{1k}^{*}I(t \leq n\tau^{*}) + X_{2k}^{*}I(t > n\tau^{*})$. 
(iv) Let  $X_{1k} \sim \mathbb{P}_{\theta_k^{*}}$ and $X_{2k} \sim \mathbb{P}_{\eta_k^{*}}$. Then for all $k \in \mathcal{K}_0$, $\theta_{k}(n) \to \theta_k^{*}$ and $\eta_{k}(n) \to \eta_k^{*}$
%\begin{align}
%X_{kt} \stackrel{\mathcal{D}}{\to} X_{1k}^{*}I(t \leq n\tau^{*}) + X_{2k}^{*}I(t > n\tau^{*}),\ \ \theta_{k}(n) \to \theta_k^{*}\ \ \ \text{and}\ \ \eta_{k}(n) \to \eta_k^{*}. \nonumber
%\end{align}

The next assumption is on the third derivative of the log-likelihood that takes values in $\mathbb{R}^{d \times d \times d}$.  Note that the operations of modulus, $\sup$, $ \leq$ and expectation on $d \times d \times d$ cubes  are component wise.   

\noindent \textbf{(B11)} $\frac{\partial^3}{\partial \lambda^3} \log P_\lambda (X)$ exists for all $\lambda \in \Lambda$ and almost everywhere in 
$X \sim P_a$, $a \in \Lambda$. 
Moreover, for some measurable function $G_3(\cdot) \in \mathbb{R}^{d \times d \times d}$
\begin{eqnarray}
\sup_{\lambda \in \Lambda} \bigg| \frac{\partial^3}{\partial \lambda^3} \log P_\lambda (x) \bigg| \le G_3 (x)\ \ \forall x
\end{eqnarray}
 such that  $E(G_3(X)) < \infty$ for any $X \sim P_a$, $a \in \Lambda$.
\vskip 5pt
%\noindent For Gaussian likelihood (B3) is always true as there $\beta(\cdot)$ is a quadratic function.
%\vskip 5pt

We are now ready to establish the asymptotic distribution of $n||\theta - \eta||_2^2 (\hat{\tau}_{n,\text{MLE}}-\tau_n)$. The proof of the following theorem is given in Section \ref{subsec: mle2}.
\begin{theorem} \textbf{MLE asymptotic distributions}. \label{thm: mle2g}
Suppose (B1)-(B6), (B8), (B9) and (SNR2) hold. We then have

\noindent $(a)$ If $||\theta-\eta||_2 \to \infty$ and $\log m(n) = o(n)$, then $\lim_{n \to \infty} P(\hat{\tau}_{n,\text{MLE}}=\tau_n) =1$

\noindent $(b)$ If $\gamma_{_{\text{MLE}}}$ exists, (B11) holds and $||\theta -\eta||_2 \to 0$, then
\begin{eqnarray}
n||\theta - \eta||_2^2 (\hat{\tau}_{n,\text{MLE}}-\tau) \stackrel{\mathcal{D}}{\to}  \arg \max_{h \in \mathbb{R}} (-0.5\gamma_{_{\text{MLE}}}^{2}|h| + \gamma_{_{\text{MLE}}}B_h) = \gamma_{_{\text{MLE}}}^{-2}\arg \max_{h \in \mathbb{R}} (-0.5|h| + B_h)\ \ \ \ 
\end{eqnarray}
where $B_h$ corresponds to a standard Brownian motion.

\noindent $(c)$ If $\gamma_{_{\text{MLE}}}^*$ exists, (A4), (B10) and (B11) hold, $\sup_{k \in \mathcal{K}_n}|\theta_k -\eta_k |_2 \to 0$ and $||\theta - \eta||_2 \to c >0$, then
\begin{eqnarray}
n (\hat{\tau}_{n,\text{MLE}} -\tau) &\stackrel{\mathcal{D}}{\to}& \arg \max_{h \in \mathbb{Z}} (D_2(h) + C_2(h) + A_2(h)), \nonumber 
%&=& c^2 \arg \max_{c^2h \in c^2\mathbb{Z}} (D(c^2 h) + C(c^2 h) + A(c^2 h)) \nonumber 
\end{eqnarray}
where for each $h \in \mathbb{Z}$,
\begin{eqnarray}
D_2 ( h+1)-D_2 ( h) &=& -0.5 \text{Sign}(h)  \gamma_{_{\text{MLE}}}^{*2}, \\
C_2 (h+1) - C_2 ( h) &=& \gamma_{_{\text{MLE}}}^{*} W_h,\ \ W_h \stackrel{\text{i.i.d.}}{\sim} \mathcal{N}(0,1), \label{eqn: msethm2cg}\\
A_2 (h+1) - A_2 ( h) &=& \sum_{k \in \mathcal{K}_0}  \left(\log \mathbb{P}_{\eta_k^*}(Z_{kh}) - \log \mathbb{P}_{\theta_k^*}(Z_{kh}) \right)
  \label{eqn: msethm2ag}
\end{eqnarray}
and $\{Z_{kh}\}$ are independently distributed with $Z_{kh} \stackrel{d}{=} X_{1k}^{*}I(h \leq 0) + X_{2k}^{*}I(h > 0)$. % the above increments are independently  distributed.
%\\
%\indent Further, if  $G_2(x)$ in (B5) does not depend on $x$, i.e. it is a constant function, then (B6) can be relaxed to (B7). 
\end{theorem}

%\begin{remark}
%Note that for regimes (b) and (c) when $||\theta - \eta||_2 \to 0$ or $c > 0$, then (SNR2) implies $m=o(\sqrt{n})$,
%which is stronger than $\log m(n) = o(n)$. 
%\end{remark}

\begin{remark}
For the Gaussian case (see remark \ref{remark:Gaussian-case}), the results in Theorem \ref{thm: mle2g} 
continue to hold under the weaker assumption (B7) (or equivalently (A1) in Section \ref{sec: lse}) and (SNR1). 
\end{remark}

\begin{remark} \label{rem: varcompmlelse}
Suppose $d=1$. Therefore $\theta_k, \eta_k \in \mathbb{R}\ \forall k$. Also suppose, for all $k \geq 1$, $E(X_{kt}) = \theta_k I(t \leq n\tau_n) + \eta_k I(t > n\tau_n)$.  Let $B_h$ be the standard Brownian motion and denote $V ={\rm{Var}}(\arg\max_{h} (-0.5|h|+B_h))$.   Then under the assumptions in Theorem \ref{thm: mle2g}(b), if $||\theta -\eta||_2 \to 0$, the asymptotic variance $V_{\text{MLE}}$ of $n||\theta - \eta||_2(\hat{\tau}_{n,\text{MLE}}-\tau_n)$ is
\begin{align*}
V_{\text{MLE}} = \bigg(\lim \frac{\sum_{k=1}^{m}(\theta_k-\eta_k)^2}{\sum_{k=1}^{m}(\theta_k-\eta_k)^2 I(\theta_k)}\bigg)^2 V = \bigg(\lim \frac{\sum_{k=1}^{m}(\theta_k-\eta_k)^2}{\sum_{k=1}^{m}(\theta_k-\eta_k)^2 I(\eta_k)}\bigg)^2 V. 
\end{align*}
Recall that ${\rm{Var}}(X_{kt}) = \sigma_{1k}^2I(t \leq n\tau_n)+\sigma_{2k}^2 I(t > n\tau_n)$ for all $k \geq 1$. Suppose $\sigma_{1k}^2 = g(\theta_k)$ and $\sigma_{2k}^2 = g(\eta_k)$ for some continuous function $g(\cdot)$ with bounded first derivative.  Then $||\theta - \eta||_2 \to 0$ implies $\sup_{k} |\sigma_{1k} - \sigma_{2k}| \to 0$. Under the assumptions given in Theorem \ref{thm: lse2}(b), if $||\theta - \eta||_2 \to 0$, the asymptotic variance $V_{\text{LSE}}$ of $n||\theta - \eta||_2(\hat{\tau}_{n,\text{LSE}}-\tau_n)$ is
\begin{align*}
V_{\text{LSE}} = \bigg(\lim \frac{\sum_{k=1}^{m}(\theta_k-\eta_k)^2 \sigma_{1k}^2}{\sum_{k=1}^{m}(\theta_k-\eta_k)^2} \bigg)^2 V = \bigg(\lim \frac{\sum_{k=1}^{m}(\theta_k-\eta_k)^2\sigma_{2k}^2}{\sum_{k=1}^{m}(\theta_k-\eta_k)^2 }\bigg)^2 V. 
\end{align*}
As the arithmetic mean is bigger than the harmonic mean and by  Cram\'{e}r-Rao lower bound, $\sigma_{1k}^2 \geq (I(\theta_k))^{-1}$, $\sigma_{2k}^2 \geq (I(\eta_k))^{-1}$ for all $k \geq 1$, we have $V_{\text{MLE}} \leq V_{\text{LSE}}$. A similar conclusion holds when $||\theta -\eta||_2 \to c$ and  $\mathcal{K}_0$ is the empty set.
\end{remark}

\subsection{Illustrative Examples}
We showcase the asymptotic behavior of $\hat{\tau}_{n,\text{MLE}}$, when the data generating mechanism for $\{X_{kt}\}$ follows specific probability distributions, including those in the exponential family (full rank as well as curved) and some models of particular interest in econometrics, namely, the $0$-inflated Poisson for count data, as well as the Probit and Tobit models. We illustrate how Assumptions (B1)-(B10) hold for these models and also provide comparisons between the least squares and maximum likelihood estimators of the change point for these examples. Apart from the one parameter full rank exponential family, together with the $0$-inflated Poisson, Probit and Tobit models, all other ones are discussed in the Supplementary file in the interests of space. 
%The next two examples deal with models extensively used in applied econometrics work.

\begin{Example} \label{example: expfam1}
\textit{\textbf{Exponential family}}.
A random variable $X$ belongs to the one parameter natural exponential family, if its probability density/mass function has the form
\begin{eqnarray} \label{eqn: defnexp}
f_{\lambda}(x) = e^{\lambda x -\beta(\lambda)+h(x)},\ \ x \in \mathbb{R}\ \text{and}\ \lambda \in \Lambda \subset \mathbb{R},
\end{eqnarray}
where $\beta(\lambda) = \log \int_{\mathbb{R}} e^{\lambda x + h(x)} dx$ is an infinitely differentiable convex function.  Note that in this case, $E(X) = \beta^\prime (\lambda)$ and $\text{Var}(X) = \beta^{\prime \prime}(\lambda)$. Since $\beta(\cdot)$ is a convex function, $\beta^\prime(\cdot)$ is a strictly increasing function and therefore its inverse exists. \\ %Further, let $\alpha = \beta^{\prime -1}$. \\
\indent This example assumes that for each $k$ and $t$, the probability distribution of $X_{kt}$ belongs to the one parameter natural exponential family and the break  occurs due to change in the value of the parameter. This is equivalent to positing that  %the following assumption.
%\vskip 5pt
%\noindent \textbf{(B1)} P
the probability density/mass function of $X_{kt}$ is 
\begin{eqnarray} \label{eqn: exppdfch}
f_{kt}(x) = f_{\theta_k}(x)I(t \leq \tau_n) + f_{\eta_k}I(t > \tau_n),
\end{eqnarray}
where $\theta_k \neq \eta_k$ for at least one $k \geq 1$. \\
\indent Following the developments in Section \ref{sec: mle}, one can establish the conclusions of Theorem \ref{thm: mle1g} and Theorem \ref{thm: mle2g}(a)  when the second derivative of $\beta(\lambda)$ for $\lambda \in \Lambda$  is bounded away from both $0$ and $\infty$, the $4$-th moment of $X_{kt}$ is uniformly bounded above and (SNR2) holds.\\
\indent Further, suppose that the third derivative of $\beta(\cdot)$ is bounded.
  %For more details see Theorem \ref{thm: mle1}.
%Suppose $\gamma_{_\text{MLE}}$ and $\gamma_{_\text{MLE}}^{*}$ exists with $I(\theta_k)$ and $I(\eta_k)$ replaced by   
Then, the conclusion of Theorem \ref{thm: mle2g}(b) holds for the model in (\ref{eqn: exppdfch},) if $\gamma_{_\text{MLE}}$  exists. In addition,
the conclusion of Theorem \ref{thm: mle2g}(b) holds for (\ref{eqn: exppdfch}), if $\gamma_{_\text{MLE}}^{*}$  exists and, (A4), (B10) and $\sup_{k \in \mathcal{K}_n} |\theta_k - \eta_k| \to 0$ hold. \\
\indent Note that in this example, $E(X_{kt}) \neq \theta_k I(t \leq n\tau_n) + \eta_k I(t > n\tau_n)$  and therefore, we can not not  apply Remark \ref{rem: varcompmlelse} directly. However, using the structure of the exponential family, one can  establish similar variance comparisons as given in Remark \ref{rem: varcompmlelse}  for the model  (\ref{eqn: exppdfch}).  \\
\indent For the Gaussian case, i.e., when $f_{\lambda}(x) = (\sqrt{2\pi \sigma^2})^{-1} \exp \{-(x-\lambda)^2 /\sigma^2\}$ for $\lambda \in \mathbb{R}$ and known constant $\sigma >0$,  Condition (SNR2) can be relaxed to (SNR1). Also in this case, $\beta(\cdot) = C\lambda^2$ for some constant $C>0$ and hence, all the requirements on $\beta(\cdot)$,  as stated above, hold  naturally. \\
\indent Details are given in Example $6.5$ of the supplement due to space constraint. 
\end{Example}

\begin{Example} \label{example: 0infpois}
\textit{\textbf{0-inflated Poisson distribution}}. A random variable $X$ follows a $0$-inflated Poisson distribution with parameter $(\sigma,\lambda)$ ($0 < \sigma <1$ and $\lambda >0$), if $X$ has the following probability mass function:
\begin{eqnarray}
P(X=x) = (\sigma + (1-\sigma)e^{-\lambda})I(x=0) + (1-\sigma) e^{-\lambda} \frac{\lambda^x}{x !} I(x =1,2,\ldots).
\end{eqnarray}
In this model,  the maximum likelihood method is  recommended over least squares because the latter method relies on the means to detect the change point.
However, for this model  $E(X) = (1-\sigma)\lambda$ and it is easy to come up with scenarios where two different pairs of $(\sigma, \lambda)$ (e.g. $(0.5,2)$ and $(2/3,3)$) end up with the same (or very similar)  mean(s). In that case, the least squares based method would fail to detect the change point, while
the maximum likelihood one would not, provided that the other conditions required and previously discussed hold.

It can then easily be seen that assumptions (B1) and (B2) hold for this example. The log-likelihood of $(\sigma,\lambda)$ is given by
\begin{eqnarray}
L(\sigma,\Lambda) &=& (\log (\sigma + (1-\sigma)e^{-\lambda}))I(x=0) + (\log(1-\sigma) - \lambda + x \log \lambda - \log(x!))I(x =1,2,\ldots) \nonumber \\
&=& (\log (\sigma + (1-\sigma)e^{-\lambda}))I(x=0) + (\log(1-\sigma) - \lambda)I(x =1,2,\ldots)  \nonumber \\
&& \hspace{2.5 cm}+ x (\log \lambda) I(x =1,2,\ldots) - (\log(x!))I(x =1,2,\ldots).
\end{eqnarray}
Thus, (B4)  holds.  Moreover, 
\begin{eqnarray}
\frac{\partial^2 L(\sigma,\lambda)}{\partial \lambda^2} &=& \frac{(1-\sigma)e^{-\lambda}}{\sigma + (1-\sigma)e^{-\lambda}} \bigg(1- \frac{(1-\sigma)e^{-\lambda}}{\sigma + (1-\sigma)e^{-\lambda}} \bigg) I(x=0) - \lambda^{-2} x I(x=1,2,\ldots), \nonumber \\
\frac{\partial^2 L(\sigma,\lambda)}{\partial \sigma \partial \lambda} &=& \bigg[ \frac{(1-\sigma)e^{-\lambda} (1-e^{-\lambda})}{(\sigma + (1-\sigma)e^{-\lambda})^2} + \frac{e^{-\lambda}}{(\sigma + (1-\sigma)e^{-\lambda})}\bigg] I(x=0), \nonumber \\
\frac{\partial^2 L(\sigma,\lambda)}{\partial \sigma^2} &=& -\bigg(\frac{ (1-e^{-\lambda})}{(\sigma + (1-\sigma)e^{-\lambda})} \bigg)^2 I(x=0)  - \frac{1}{(1-\sigma)^2} I(x = 1,2,\ldots).
\end{eqnarray}
Consider the following three measurable functions on the set of all non-negative integers:
\begin{eqnarray}
G_{11}(x) = I(x=0) + xI(x=1,2,\ldots),\ G_{12}(x) = I(x=0)\ \text{and}\ \ G_{22}(x) = 1. 
\end{eqnarray}
Note that the above functions satisfy (\ref{eqn: sdv2}). For some fixed  $0<c_1<1$ and $0<c_2<c_3<\infty$, define the restricted parameter space
\begin{eqnarray} \label{eqn: 0poispara}
\Lambda = \{(\sigma,\lambda): \ \sigma \in (c_1,1-c_1),\ \lambda \in (c_2,c_3)\}.
\end{eqnarray}
 Then, for some $0<C_1 \leq  C_2< \infty$,
\begin{eqnarray}
&& 0 < C_1 G_{11}(x) \leq \sup_{(\sigma,\lambda) \in \Lambda} \bigg |\frac{\partial^2 L(\sigma,\lambda)}{\partial \lambda^2} \bigg | \leq  C_2 G_{11}(x) < \infty, \nonumber \\
&& 0 < C_1 G_{12}(x) \leq \sup_{(\sigma,\lambda) \in \Lambda} \bigg |\frac{\partial^2 L(\sigma,\lambda)}{\partial \sigma \partial \lambda} \bigg| \leq  C_2 G_{12}(x) < \infty, \nonumber \\
&& 0 < C_1 G_{22}(x) \leq \sup_{(\sigma,\lambda) \in \Lambda} \bigg|\frac{\partial^2 L(\sigma,\lambda)}{\partial \sigma^2} \bigg| \leq  C_2 G_{22}(x) < \infty. \nonumber 
\end{eqnarray}
Hence, (B5) holds for this example. Moreover, (B6), (B8), (B9) and (B11) hold for the parameter space $\Lambda$ defined in (\ref{eqn: 0poispara}). 

Suppose the data $\{X_{kt}\}$ are generated from the $0$-inflated Poisson with parameter 
$$(\sigma_{1k},\lambda_{1k})I(t \leq n\tau_n) + (\sigma_{2k},\lambda_{2k})I(t > n\tau_n)\ \ \forall k \geq 1,$$
where $(\sigma_{1k},\lambda_{1k}) \neq (\sigma_{2k},\lambda_{2k})$ for at least one $k$. We obtain $\hat{\tau}_{n,\text{MLE}}$ by (\ref{eqn: taumle}). 

Suppose $(\sigma,\lambda) \in \Lambda$,  (SNR2) holds and $\log m(n) = o(n)$. Then, 
\begin{eqnarray} %\label{eqn: msethm1g}
n \bigg( \sum_{k=1}^{n} (\sigma_{1k} - \sigma_{2k})^2 + \sum_{k=1}^{n} (\lambda_{1k} - \lambda_{2k})^2 \bigg) (\hat{\tau}_{n,\text{MLE}}-\tau_n) = O_{P}(1).
\end{eqnarray}
Under the above assumptions and $\big(\sum_{k=1}^{n} (\sigma_{1k} - \sigma_{2k})^2 + \sum_{k=1}^{n} (\lambda_{1k} - \lambda_{2k})^2 \big) \to \infty$, 
\begin{eqnarray}
P(\hat{\tau}_{n,\text{MLE}} = \tau_n) \to 1\ \ \text{as $n \to \infty$}.
\end{eqnarray}
Further, suppose $\gamma_{_{\text{MLE}}}$ and  $\gamma_{_{\text{MLE}}}^*$ exist respectively, when $\big( \sum_{k=1}^{n} (\sigma_{1k} - \sigma_{2k})^2 + \sum_{k=1}^{n} (\lambda_{1k} - \lambda_{2k})^2 \big) \to 0$ and $C>0$. Moreover, suppose $\sup_{k \in \mathcal{K}_n}(|\sigma_{1k} - \sigma_{2k}| + |\lambda_{1k} - \lambda_{2k}|) \to 0$, (A4) and (B10) hold for the latter case. Then, the conclusions of Theorem \ref{thm: mle2g}(b) and (c) continue to hold.  Note that in the last two regimes, (SNR2) implies $m(n) = o(\sqrt{n})$, which is stronger than $\log m(n) = o(n)$. 
\end{Example}

\begin{Example} \label{example: probit}
\textit{\textbf{Probit model}}. Suppose a response variable $X$ is binary, that is it can have only two possible outcomes which we will denote as 1 and 0. We also have a predictor vector $Y \in \mathbb{R}^{d}$, which is assumed to influence the outcome $X$. The probit model is then defined as
%\begin{eqnarray}
$P(X=1) = \Phi(Y^\prime \beta)$,
%\end{eqnarray}
where $\Phi(\cdot)$ is the distribution function of the standard Gaussian variable and $\beta \in \mathbb{R}^{d}$ is the parameter vector of interest. Clearly this model satisfies (B1), (B2) and (B3). 
\\ \indent The log-likelihood of $\beta$ is given by
%\begin{eqnarray}
$L(\beta) = X \log \Phi(Y^\prime \beta) + (1-X) \log (1-\Phi(Y^\prime \beta))$. 
%\end{eqnarray} 
Therefore, (B4) holds. For all $x \in \mathbb{R}$, let $\Phi^\prime(x) = \frac{\partial}{\partial x} \Phi(x)$ and  $\Phi^{\prime \prime}(x)=\frac{\partial^2}{\partial x^2} \Phi(x)$.  Therefore, \\
\begin{align*}
\frac{\partial^2}{\partial \beta^2} L(\beta) = - \left(\frac{x-\Phi(Y^\prime \beta)}{\Phi(Y^\prime \beta)(1-\Phi(Y^\prime \beta))} \Phi^\prime (Y^\prime \beta)\right)^2 YY^\prime + \frac{x-\Phi(Y^\prime \beta)}{\Phi(Y^\prime \beta)(1-\Phi(Y^\prime \beta))} \Phi^{\prime \prime}(Y^\prime \beta)YY^\prime.
\end{align*}
Moreover, as $\Phi^{\prime \prime}(x) = -x\Phi^\prime(x) $, on simplification, we get 
\begin{align*}
\frac{\partial^2}{\partial \beta^2} L(\beta) = - \left(\frac{x-\Phi(Y^\prime \beta)}{\Phi(Y^\prime \beta)(1-\Phi(Y^\prime \beta))} \Phi^\prime (Y^\prime \beta)\right) \left( x + \frac{x-\Phi(Y^\prime \beta)}{\Phi(Y^\prime \beta)(1-\Phi(Y^\prime \beta))} \Phi^\prime (Y^\prime \beta) \right)YY^\prime.
\end{align*}
Suppose both $Y$ and $\beta$ belong to a compact subset $\Lambda$ of $\mathbb{R}^d$, so that $0<C_1 \leq Y^\prime \beta \leq C_2 < \infty$.  Then,
it is easy to see that for some $C_1, C_2>0$
\begin{align*}
0 < C_1 J_d \leq \inf_{y, \beta \in \Lambda} \bigg|\frac{\partial^2}{\partial \beta^2} L(\beta)\bigg| \leq \sup_{y, \beta \in \Lambda} \bigg|\frac{\partial^2}{\partial \beta^2} L(\beta)\bigg| < C_2 J_d < \infty.
\end{align*}
Therefore, (B5) holds with $G_2 = CJ_d$ for some $C>0$. 
\\ \indent Further, since $X$ is a bounded random variable and $y, \beta$ belong to a compact subset of $\mathbb{R}^d$, Assumptions (B6) and (B7) are satisfied. In addition, (B8) and (B9) hold, since $X$ is Sub-Gaussian and $G_2$ is a constant function. Analogously, it is see that (B11) holds.
\\ \indent Suppose $\{X_{kt}\}$ are independently generated from the Probit model with parameter
\begin{eqnarray}
\beta_{1k}(n)I(t \leq n\tau_n) + \beta_{2k}(n) I(t > n\tau_n)\ \ \forall k \geq 1. \nonumber 
\end{eqnarray}
For ease of presentation, we shall write $\beta_{1k}$ and $\beta_{2k}$, respectively, for $\beta_{1k}(n)$ and $\beta_{2k}(n)$. 
Further, $\beta_{1k} \neq \beta_{2k}$ for at least one $k$. Let
%\begin{eqnarray}
$||\beta_1 - \beta_2||_2^2 = \sum_{k=1}^{m} ||\beta_{1k} - \beta_{2k}||_2^2$.
%\end{eqnarray}
Suppose $Y,\beta \in \Lambda$,  (SNR2) holds and $\log m(n) = o(n)$. Then,
%\begin{eqnarray} \nonumber 
$n  ||\beta_1 - \beta_2||_2^2 (\hat{\tau}_{n,\text{MLE}}-\tau_n) = O_{P}(1)$ 
%\end{eqnarray}
%Under the above assumptions 
and when $||\beta_1 - \beta_2||_2^2 \to \infty$, 
%\begin{eqnarray}
we have $P(\hat{\tau}_{n,\text{MLE}} = \tau_n) \to 1$ as $n \to \infty$.
%\end{eqnarray}
\\ \indent Further, suppose  $\gamma_{_{\text{MLE}}}$ and  $\gamma_{_{\text{MLE}}}^*$ exist respectively,  when  $||\beta_1 - \beta_2||_2^2 \to 0$ and $C>0$. 
In addition, suppose $\sup_{k \in \mathcal{K}_n}|\beta_{1k} - \beta_{2k}|_2 \to 0$, (A4) and (B10) hold for the latter case. Then, the conclusions of Theorem \ref{thm: mle2g}(b) and (c) continue to hold.  Again as in previous examples, for the last two regimes,  (SNR2) implies the stronger assumption $m(n) = o(\sqrt{n})$ than $\log m(n) = o(n)$. 
\end{Example}

\begin{Example} \label{example: Tobit}
\textit{\textbf{Tobit model}}.  In this model, the response variable $X$ depends on a $d \times 1$ vector predictor $Y$ as %follows:
%\begin{align*}
$X = (Y^\prime \beta + \varepsilon) I(\varepsilon > -Y^\prime \beta)$,
%\end{align*}
where $\varepsilon \sim \mathcal{N}(0,1)$ and $\beta$ is a $d\times 1$ parameter vector. Though Assumptions (B1), (B2) and (B3) hold for this model,  since $X$ has neither a probability density, nor a mass function, there may not always exist a consistent solution of the joint log-likelihood equation (\ref{eqn: likelihood equation}). \citet{amemiya1973regression} established that (\ref{eqn: likelihood equation}) has a consistent sequence of solutions, if for a given data set $\{(X_i,Y_i):\ 1 \leq i \leq n\}$,  the following (T1) and (T2) conditions hold. 

\noindent (T1) The empirical distribution of $\{Y_i:\ 1 \leq i \leq n\}$ converges weakly to some probability distribution.

\noindent  (T2) $\lim_{n \to \infty}\frac{1}{n} \sum_{i=1}^{n} Y_i Y_i^\prime$ is positive definite.

%Let $\Phi(\cdot)$ denote the standard Gaussian distribution function. 
The log-likelihood function of $\beta$ is given by
\begin{align*}
L(\beta) = (1- \Phi(Y^\prime \beta))I(X=0) - 0.5 (X-Y^\prime \beta)^2 I(X>0) + C
\end{align*}
for some constant $C$. Thus, Assumption (B4) holds. Let $\Phi^{\prime \prime}(x) =\frac{\partial^2}{\partial x^2} \Phi(x)$. Then,
\begin{align*}
\frac{\partial^2}{\partial \beta^2} L(\beta) &=  - \Phi^{\prime \prime}(Y^\prime \beta) YY^\prime I(X=0) -  YY^\prime I(X>0).
\end{align*}
Suppose both $Y$ and $\beta$ belong to a compact subset $\Lambda$ of $\mathbb{R}^d$, so that $0<C_1 \leq Y^\prime \beta \leq C_2 < \infty$. Therefore, (B5) holds with $G_2(x) = C$ for all $x$ and for some $C>0$. Similarly (B11) holds.  Moreover, it is easy to see that (B8) and (B9) are satisfied for this model.  \\
\indent Suppose $\{X_{kt}\}$ are independently generated with 
\begin{eqnarray}
X_{kt} = \begin{cases} (Y_{kt}^\prime \beta_{1k}(n) + \varepsilon_{kt}) I(\varepsilon_{kt} > -Y_{kt}^\prime \beta_{1k}(n)),\ \ \text{if $t \leq n\tau_n$} \\
(Y_{kt}^\prime \beta_{2k}(n) + \varepsilon_{kt}) I(\varepsilon_{kt}(n) > -Y_{kt}^\prime \beta_{2k}(n)),\ \ \text{if $t \leq n\tau_n$}
\end{cases}
\end{eqnarray}
where $\{\varepsilon_{kt}\}$ are i.i.d. standard Gaussian variables and $\beta_{1k}(n) \neq \beta_{2k}(n)$ for at least one $k$.  We shall write $\beta_{1k}$ and $\beta_{2k}$ respectively, for $\beta_{1k}(n)$ and $\beta_{2k}(n)$. Suppose for each $k \geq 1$, $\{Y_{kt}:\  t \leq n\tau_n\}$ and  $\{Y_{kt}:\  t \geq n\tau_n\}$ satisfy  (T1) and (T2) and $Y, \beta \in \Lambda$. 
Then, under (SNR2) and $\log m(n) = o(n)$,  we have
%\begin{eqnarray} \nonumber 
$n  ||\beta_1 - \beta_2||_2^2 (\hat{\tau}_{n,\text{MLE}}-\tau_n) = O_{P}(1)$
%\end{eqnarray}
%Under the above assumptions 
and when $||\beta_1 - \beta_2||_2^2 \to \infty$,  we have
%\begin{eqnarray}
$P(\hat{\tau}_{n,\text{MLE}} = \tau_n) \to 1$ as $n \to \infty$.
%\end{eqnarray}
\\ \indent Further, suppose  $\gamma_{_{\text{MLE}}}$ and  $\gamma_{_{\text{MLE}}}^*$ exist respectively,  when  $||\beta_1 - \beta_2||_2^2 \to 0$ and $C>0$. Moreover, suppose $\sup_{k \in \mathcal{K}_n}|\beta_{1k} - \beta_{2k}|_2 \to 0$, (A4) and (B10) hold for the latter case. Then, the conclusions of Theorem \ref{thm: mle2g}(b) and (c) continue to hold.  Again as in previous examples, for the last two regimes,  (SNR2) implies the more strong assumption $m(n) = o(\sqrt{n})$ than the required $\log m(n) = o(n)$. 
\end{Example}

\begin{remark}
In this work, we do not pursue the investigation of ML estimation of the change point under dependence, since the likelihood will depend on
the temporal dependence posited and can become exceedingly complicated.
\end{remark}

\section{Adaptive inference for the asymptotic distribution of the change-point estimate} \label{sec: adaptive}
In Sections \ref{sec: lse} and \ref{sec: mle}, we derived point estimates $\hat{\tau}_{n,\text{LSE}}$ and $\hat{\tau}_{n,\text{MLE}}$ of the change point $\tau_n$  and established their convergence rates and asymptotic distributions, respectively.  
However, the results in Theorems \ref{thm: lse2} and \ref{thm: mle2g} identify three different limiting regimes depending on the behavior 
of the norm difference of the parameters before and after the change point. The latter norm difference is {\em not a priori known},
and hence the practitioner is left with the dilemma of which regime to use for construction of confidence intervals. Next, we present
{\em a data based adaptive procedure} to determine  the quantiles of the asymptotic distribution, irrespective of the specific regime pertaining to the
data at hand. 

\subsection{Adaptive inference for the least squares estimator} \label{sub: adaptivelse}
Recall the observed data set $\{X_{kt}: k,t \geq 1\}$. Let $\mathbb{P}_{\mu,\sigma^2,\theta}$ be a probability distribution which is fully characterized by its mean $\mu$, variance $\sigma^2$ and the $d \times 1$ parameter vector $\theta$. Therefore, $$\mu = \int_{\mathbb{R}} x d\mathbb{P}_{\mu,\sigma^2, \theta}\ \ \text{and}\ \ \ \sigma^2 = \int_{\mathbb{R}}(x-\mu)^2 d\mathbb{P}_{\mu,\sigma^2, \theta}.$$ 
$\mu$, $\sigma^2$ and $\theta$ may not be functionally independent.  We denote the pre- and post-change point probability distributions of $X_{kt}$ by $\mathbb{P}_{\mu_{1k}(n), \sigma_{1k}^2(n), \theta_{1k}(n)}$ and $\mathbb{P}_{\mu_{2k}(n), \sigma_{2k}^2(n), \theta_{2k}(n)}$, respectively.  For ease of exposition, we shall write  $\mu_{ik}$, $\sigma_{ik}^2$ and $\theta_{ik}$, respectively, for $\mu_{1k}(n)$, $\sigma_{ik}^2(n)$ and $\theta_{ik}(n)$. Let $\hat{\tau}_{n,\text{LSE}}$ be the least squares estimator of the change point $\tau_n$,
\begin{align*}
& \hat{\mu}_{1k} = \frac{1}{n\hat{\tau}_{n,\text{LSE}}} \sum_{t=1}^{n\hat{\tau}_{n,\text{LSE}}} X_{kt},\ \ \hat{\sigma}_{1k}^2 = \frac{1}{n\hat{\tau}_{n,\text{LSE}}} \sum_{t=1}^{n\hat{\tau}_{n,\text{LSE}}} (X_{kt} - \hat{\mu}_{1k})^2, \nonumber \\
& \hat{\mu}_{2k} = \frac{1}{n(1-\hat{\tau}_{n,\text{LSE}})} \sum_{t=n\hat{\tau}_{n,\text{LSE}}+1}^{n} X_{kt},\ \ \hat{\sigma}_{2k}^2 = \frac{1}{n(1-\hat{\tau}_{n,\text{LSE}})} \sum_{t=n\hat{\tau}_{n,\text{LSE}}+1}^{n} (X_{kt} - \hat{\mu}_{2k})^2\nonumber
\end{align*}
and $\hat{\theta}_{ik}$ be an estimator of $\theta_{ik}$, such that $\hat{\theta}_{ik} - {\theta}_{ik} \stackrel{\text{P}}{\to} 0,\ \ \forall k \geq 1, i=1,2$. If $\sigma^2= g(\mu)$ and $\theta = h(\mu,\sigma^2)$ for some functions $g$ and $h$, then we consider $\hat{\sigma}_{ik}^2 = g(\hat{\mu}_{ik})$ and $\hat{\theta}_{ik} = h(\hat{\mu}_{ik},\hat{\sigma}_{ik}^2)$. Moreover, if $\theta = h(\mu,\sigma^2,\eta)$ for some function $h$ and $\eta$ does not depend on $\mu$ and $\sigma^2$, then $\hat{\theta}_{ik} = h(\hat{\mu}_{ik},\hat{\sigma}_{ik}^2,\hat{\eta}_{ik})$ where $\hat{\eta}_{ik} - {\eta}_{ik} \stackrel{\text{P}}{\to} 0,\ \ \forall k \geq 1, i=1,2$.

Generate independent random variables $\{\tilde{X}_{kt,\text{LSE}}: k,t \geq 1\}$, where for each $k \geq 1$, 
\begin{align}
& \tilde{X}_{kt,\text{LSE}} \sim \mathbb{P}_{\hat{\mu}_{1k},\hat{\sigma}_{1k}^2,\hat{\theta}_{1k}}I(t \leq n\hat{\tau}_{n,\text{LSE}}) + \mathbb{P}_{\hat{\mu}_{2k},\hat{\sigma}_{2k}^2,\hat{\theta}_{2k}}I(t > n\hat{\tau}_{n,\text{LSE}}).
\end{align}
The least squares criterion function is given by
\begin{eqnarray} \label{eqn: tildetau}
\tilde{M}_{n}(h) =-\frac{1}{n} \sum_{k=1}^{m} \bigg[\sum_{t=1}^{n\hat{\tau}_{n,\text{LSE}}+h} (\tilde{X}_{kt,\text{LSE}} - \hat{\mu}_{1k})^2 + \sum_{t=n\hat{\tau}_{n,\text{LSE}}+h+1}^{n} (\tilde{X}_{kt,\text{LSE}} - \hat{\mu}_{2k})^2\bigg],
\end{eqnarray}
and define $\displaystyle \tilde{h}_{n,\text{LSE}} = {\arg\max}
%{n\hat{\tau}_{n,\text{LSE}}+h \in (nc^{*},n(1-c^{*}))} 
\{\tilde{M}_n(h): h \in [n(c^* - \hat{\tau}_{n,\text{LSE}}), n(1-c^* -\hat{\tau}_{n,\text{LSE}})]\}$.   Note that $\tilde{h}_{n,\text{LSE}}$ can take both positive and negative values as $c^{*}<\hat{\tau}_{n,\text{LSE}} < 1-c^{*}$.

The following theorem states the asymptotic distribution of $\tilde{h}_{n,\text{LSE}}$. In \textbf{Regime (a)}: $||\mu_1 - \mu_2||_2 \to \infty$, we need the same assumptions as in Theorem \ref{thm: lse2}(a). In \textbf{Regime (b)},  $||\mu_1 - \mu_2||_2 \to 0$  additional assumptions are required, beyond those posited in  %and a stronger signal-to-noise condition 
%are required, in addition to the assumptions posited in 
Theorem \ref{thm: lse2}(b), as well as a stronger signal-to-noise condition.  These are:

\noindent \textbf{(A8)} $\{X_{kt}\}$ are Sub-Gaussian, and

\noindent \textbf{(SNR3)} $\frac{1}{\sqrt{\log m}}\frac{\sqrt{n}}{m} ||\mu_1 - \mu_2||_2^{2} \to \infty$.

\noindent To prove our result, at a certain point, we need to establish that 
\begin{eqnarray}
\frac{\sum_{k=1}^{m}(\hat{\mu}_{1k}-\hat{\mu}_{2k})^2\hat{\sigma}_{1k}^2}{\sum_{k=1}^{m}(\hat{\mu}_{1k}-\hat{\mu}_{2k})^2} \stackrel{P}{\to} \gamma_{_\text{L,LSE}}^2\ \ \text{and}\ \ \  \frac{\sum_{k=1}^{m}(\hat{\mu}_{1k}-\hat{\mu}_{2k})^2\hat{\sigma}_{2k}^2}{\sum_{k=1}^{m}(\hat{\mu}_{1k}-\hat{\mu}_{2k})^2} \stackrel{P}{\to} \gamma_{_\text{R,LSE}}^2. \label{eqn: adapgamma}
\end{eqnarray} 
(A8) and (SNR3) are needed to show the convergences in (\ref{eqn: adapgamma}).

%$\frac{m}{n} ||\mu_1 - \mu_2||_2^{-2} = o((n\log m)^{-1/2})$.

Next, we consider \textbf{Regime (c)}: $||\mu_1 - \mu_2||_2 \to c>0$.  Recall the partition of the index set $\{1,2,\ldots, m(n)\}$ into 
$\mathcal{K}_0$ and $\mathcal{K}_n$. Further, recall assumptions (A4) and (A5) on the set $\mathcal{K}_0$,  where we assume that $\mathcal{K}_0$ does not vary with $n$ and, for all $k \in \mathcal{K}_0$ and $0<f<1$, $X_{k\lfloor nf \rfloor} \stackrel{\mathcal{D}}{\to} X_{1k}^{*}I(f \leq \tau^{*}) + X_{2k}^{*}I(f > \tau^{*})$  and $\mu_{ik}(n) \to \mu_{ik}^{*},\ i=1,2$. By assumptions (A8) and (SNR3), $\hat{\mu}_{ik}(n) -\mu_{ik}(n) \stackrel{\text{P}}{\to} 0,\ \ k \in \mathcal{K}_0,  i=1,2$.  To ensure $\tilde{X}_{k\lfloor nf \rfloor,\text{LSE}} \stackrel{\mathcal{D}}{\to} X_{1k}^{*}I(f \leq \tau^{*}) + X_{2k}^{*}I(f > \tau^{*})$  and $\hat{\mu}_{ik} \stackrel{\text{P}}{\to} \mu_{ik}^{*}$, we need the following assumptions. 

Let $X_n \sim \mathbb{P}_{\mu_n,\sigma_n^2,\theta_n}$ and $X \sim \mathbb{P}_{\mu,\sigma^2,\theta}$. \\
\noindent \textbf{(A9)}   $X_n \stackrel{\mathcal{D}}{\to} X$, if and only if $(\mu_n,\sigma_n^2,\theta_n) \to (\mu,\sigma^2,\theta)$.

\noindent \textbf{(A10)} $\tau_n \to \tau^{*}$ and for all $k \in \mathcal{K}_0$, i=1,2,
$({\mu}_{ik}(n),\sigma_{ik}^2(n),\theta_{ik}(n)) \to (\mu_{ik}^{*},\sigma_{ik}^{*2},\theta_{ik}^{*})$.

\noindent Note that (A9) and (A10) together imply (A5) and, $X_{1k}^{*}$ and $X_{2k}^{*}$ come from the same family of distributions as the data.  Since the convergence in Regime (b) is similar to that on $\mathcal{K}_n$, % and on the set $\mathcal{K}_n$ is similar,
we  require (A8) on $\mathcal{K}_n$ and (SNR3) in addition to the assumptions in Theorem \ref{thm: lse2}(c). 

We can then establish the following result, whose proof is delegated to Section \ref{subsec: adaplse}.

\begin{theorem}  \label{thm: adaplse}
Suppose (A1) holds. Then, the following statements hold.

\noindent {\bf (a)} Under (SNR1) and $||\mu_1 - \mu_2||_2 \to \infty$, $P( \tilde{h}_{n,\text{LSE}} =0) \to 1$. %$P(\tilde{\tau}_{n,\text{LSE}} = \tau_n) \to 1$.

\noindent {\bf (b)} Suppose (A2), (A3), (A8) and (SNR3) hold and $||\mu_1 - \mu_2||_2 \to 0$. Then,
\begin{eqnarray}
%n||\hat{\mu}_1 -\hat{\mu}_2||_2^2 (\tilde{\tau}_{n,\text{LSE}} - \hat{\tau}_{n,\text{LSE}}) 
||{\mu}_1 -{\mu}_2||_2^2 \tilde{h}_{n,\text{LSE}} \stackrel{\mathcal{D}}{\to} \arg \max_{h \in \mathbb{R}} (-0.5|h| + \gamma_{_{\text{L,LSE}}}B_hI(h\leq 0) + \gamma_{_{\text{R,LSE}}}B_h I(h>0)), \nonumber
\end{eqnarray}
where $B_h$ denotes the standard Brownian motion.

\noindent {\bf (c)} Suppose (A4), (A6), (A7)-(A10) and (SNR3) hold and $||\mu_1-\mu_2||_2 \to c>0$. Then, 
\begin{eqnarray}
%n||\hat{\mu}_1 -\hat{\mu}_2||_2^2 (\tilde{\tau}_{n,\text{LSE}} - \hat{\tau}_{n,\text{LSE}}) 
\tilde{h}_{n,\text{LSE}} \stackrel{\mathcal{D}}{\to}  \arg \max_{h \in \mathbb{Z}} (D_1 (h) + C_1 (h) + A_1 (h)), \nonumber 
%&=& c^2 \arg \max_{c^2h \in c^2\mathbb{Z}} (D(c^2 h) + C(c^2 h) + A(c^2 h)), \nonumber 
\end{eqnarray}
where  for each $h \in \mathbb{Z}$,
\begin{eqnarray}
D_1 (h+1) -D_1(h) &=& -0.5\text{Sign}(h)  c_1^2, \nonumber \\
C_1 (h+1) - C_1 (h) &=& (\gamma_{_{\text{L,LSE}}}^{*}I(h\leq 0) + \gamma_{_{\text{R,LSE}}}^{*}I(h> 0)) W_h,\ \ W_h \stackrel{\text{i.i.d.}}{\sim} \mathcal{N}(0,1), \ \ \ \ \ \ \  \nonumber \\
A_1(h+1) - A_1(h) &=& \sum_{k \in \mathcal{K}_0} \bigg[(Z_{kh} -\mu_{1k}^{*})^2 -(Z_{kh} -\mu_{2k}^{*})^2 \bigg],\  %\nonumber %\\
  %Z_{kh} &\stackrel{d}{=}& X_{1k}I(h \leq 0) + X_{nk}I(h > 0) 
  \nonumber 
\end{eqnarray}
 $\{Z_{kh}\}$ are independently distributed with $Z_{kh} \sim \mathbb{P}_{\mu_{1k}^{*},\sigma_{1k}^{*2},\theta_{1k}^{*}}I(h \leq 0) + \mathbb{P}_{\mu_{2k}^{*},\sigma_{2k}^{*2},\theta_{2k}^{*}}I(h > 0)$  for all $k \in \mathcal{K}_0$.  % the above increments are independently  distributed.
  Moreover, if $\sigma_{ik}^2 = \sigma_i^2$ for all $k \geq 1$ and $i=1,2$, then (SNR3) in (b) and (c) can be relaxed to (SNR1). 
\end{theorem}

The upshot of this Theorem is that the asymptotic distributions of $\tilde{h}_{n,\text{LSE}}$ and $n(\hat{\tau}_{n,\text{LSE}}-\tau_n)$ are 
{\em identical for all regimes}. Therefore, in practice we can simulate $\tilde{h}_{n,\text{LSE}}$ for a large number of replicates and its 
sample quantiles will be good estimators for the quantiles of the limiting distribution {\em under the true regime}. 
Although this is a computationally expensive procedure, it is nevertheless {\em trivially parallelizable}. 

However, adaptive inference comes at a certain cost, namely the requirement for assumption (SNR3).  The reason for assuming (SNR3) is explained after stating (A8) and (SNR3) and is difficult to relax.

\subsection{Adaptive inference for the maximum likelihood estimates of the change point} \label{sub: adaptivemle}
Consider the set of  probability mass/density functions $\{\mathbb{P}_\lambda: \lambda \in \Lambda\}$ which satisfy (B1)-(B3).   
The observed data $\{X_{kt}: 1 \leq t \leq n, 1 \leq k \leq m\}$ are independently generated according to  
\begin{align*}
X_{kt} \sim \mathbb{P}_{{\theta}_k(n)}I(t \leq n \tau_n) + \mathbb{P}_{{\eta}_k(n)}I(t > n \tau_n),\ \ k \geq 1.
\end{align*}
Let  $\hat{\tau}_{n,\text{MLE}}$ be the maximum likelihood estimator of the change point $\tau_n$ based on the data set $\{X_{kt}: 1 \leq t \leq n, 1 \leq k \leq m\}$. Further, let $\hat{\theta}_k(n)$ and $\hat{\eta}_k(n)$ be respectively the %consistent 
solutions of the log-likelihood equation
\begin{align*}
\sum_{t=1}^{n\hat{\tau}_{n,\text{MLE}}} \frac{\partial}{\partial \theta}\log \mathbb{P}_{\theta}(X_{kt})=0\ \ \text{and}\ \ \  \sum_{t=n\hat{\tau}_{n,\text{MLE}}+1}^{n}\frac{\partial}{\partial \eta} \log \mathbb{P}_{\eta}(X_{kt})=0.
\end{align*}
Existence of such solutions is guaranteed by (B1)-(B3). 

Generate independent random variables $\{\tilde{X}_{kt,\text{MLE}}: k,t \geq 1\}$ by 
\begin{align*}
\tilde{X}_{kt,\text{MLE}} \sim \mathbb{P}_{\hat{\theta}_k(n)}I(t \leq n \hat{\tau}_{n,\text{MLE}}) + \mathbb{P}_{\hat{\eta}_k(n)}I(t > n \hat{\tau}_{n,\text{MLE}}),\ \ \forall k \geq 1.
\end{align*}
For ease of exposition, we shall write $\theta_{k}$, $\eta_k$, $\hat{\theta}_k$ and $\hat{\eta}_k$ respectively, for $\theta_{k}(n)$, $\eta_k(n)$, $\hat{\theta}_k(n)$ and $\hat{\eta}_k(n)$. Consider the maximum likelihood criterion function
\begin{align*}
\tilde{L}_n(h) = \frac{1}{n} \sum_{k=1}^{m} \bigg[\sum_{t=1}^{n\hat{\tau}_{n,\text{MLE}}+h} \log \mathbb{P}_{\hat{\theta}_k}(\tilde{X}_{kt,\text{MLE}}) +  \sum_{t=n\hat{\tau}_{n,\text{MLE}}+h+1}^{n} \log \mathbb{P}_{\hat{\eta}_k}(\tilde{X}_{kt,\text{MLE}})\bigg],
\end{align*}
and let $\tilde{h}_{n,\text{MLE}} = \arg \max \{\tilde{L}_n(h): h \in [n(c^* - \hat{\tau}_{n,\text{MLE}}), n(1-c^* -\hat{\tau}_{n,\text{MLE}})]\}$. 

Consider the following assumptions for Regime (c): $||\theta - \eta||_2 \to c>0$.  Let $X_n \sim \mathbb{P}_{\lambda_n}$ and $X \sim \mathbb{P}_{\lambda}$.

\noindent \textbf{(B12)}   $X_n \stackrel{\mathcal{D}}{\to} X$, if and only if $\lambda_n \to \lambda$. $\mathbb{P}_{\lambda}(x)$ is a continuous function of $\lambda$ and $x$. 

\noindent \textbf{(B13)} $\tau_n \to \tau^{*}$ and for all $k \in \mathcal{K}_0$, 
$(\theta_{k}(n),\eta_k(n)) \to (\theta_k^{*},\eta_{k}^{*})$.

The following Theorem states the asymptotic distribution of $\tilde{h}_{n,\text{MLE}}$ and its proof is similar to the proof of Theorem \ref{thm: adaplse} and therefore ommitted.

\begin{theorem} \label{thm: adapmle}
Suppose (B1)-(B6), (B8) and (B9) hold. Then, the following statements hold.

\noindent {\bf $(a)$} If (SNR2) holds, $||\theta-\eta||_2 \to \infty$, $\log m(n) = o(n)$, then $\lim_{n \to \infty} P(\tilde{h}_{n,\text{MLE}}=0) =1$.

\noindent {\bf $(b)$} If $\gamma_{_{\text{MLE}}}$ exists, (B11) and (SNR3) hold and  $||\theta -\eta||_2 \to 0$, then
\begin{eqnarray}
&& ||\theta - \eta||_2^2 \tilde{h}_{n,\text{MLE}} %-\hat{\tau}_{n,\text{MLE}}) 
\stackrel{\mathcal{D}}{\to} %\nonumber \\
 \hspace{0 cm} \arg \max_{h \in \mathbb{R}} (-0.5\gamma_{_{\text{MLE}}}^{2}|h| + \gamma_{_{\text{MLE}}}B_h) = \gamma_{_{\text{MLE}}}^{-2}\arg \max_{h \in \mathbb{R}} (-0.5|h| + B_h), \ \ \ \ \nonumber
\end{eqnarray}
where $B_h$ denotes the standard Brownian motion.

\noindent {\bf $(c)$} Suppose $\gamma_{_{\text{MLE}}}^*$ exists, (A4), (B11)-(B13) and (SNR3) hold, $\sup_{k \in \mathcal{K}_n}|\theta_k -\eta_k | \to 0$ and $||\theta - \eta||_2 \to c >0$, then
\begin{eqnarray}
\tilde{h}_{n,\text{MLE}} %-\hat{\tau}_{n,\text{MLE}}) 
&\stackrel{\mathcal{D}}{\to}& \arg \max_{h \in \mathbb{Z}} (D_2(h) + C_2(h) + A_2(h)),\nonumber 
%&=& c^2 \arg \max_{c^2h \in c^2\mathbb{Z}} (D(c^2 h) + C(c^2 h) + A(c^2 h)) \nonumber 
\end{eqnarray}
where for each $h \in \mathbb{Z}$,
\begin{eqnarray}
D_2 ( h+1)-D_2 ( h) &=& -0.5 \text{Sign}(h)  \gamma_{_{\text{MLE}}}^{*2}, \nonumber\\
C_2 (h+1) - C_2 ( h) &=& \gamma_{_{\text{MLE}}}^{*} W_h,\ \ W_h \stackrel{\text{i.i.d.}}{\sim} \mathcal{N}(0,1), \nonumber\\
A_2 (h+1) - A_2 ( h) &=& \sum_{k \in \mathcal{K}_0}  \left(\log \mathbb{P}_{\eta_k^*}(Z_{kh}) - \log \mathbb{P}_{\theta_k^*}(Z_{kh}) \right),
  \nonumber
\end{eqnarray}
and $\{Z_{kh}\}$ are independently distributed with $Z_{kh} \sim \mathbb{P}_{\theta_k^*}I(h \leq 0) + \mathbb{P}_{\eta_k^*}I(h > 0)$, $k \in \mathcal{K}_0$. % the above increments are independently  distributed.

Further, if  $G_2(x)$ in (B5) does not depend on $x$, i.e. it is a constant function, then (B6) can be relaxed to (B7). 
\end{theorem}

\begin{remark}
Suppose  $\frac{\partial^2 }{\partial \lambda^2} \log \mathbb{P}_{\lambda} (x) = -\Sigma$ for all $\lambda$, $x$ and for some positive definite matrix $\Sigma \in \mathbb{R}^{d \times d}$ which does not depend on $\lambda$ and $x$. This is equivalent to saying that for each $k \geq 1$,
%\begin{eqnarray}
$X_{kt} \sim \mathcal{N}_d(\theta_k, \Gamma)I(t \leq n\tau_n) + \mathcal{N}_d(\eta_k, \Gamma)I(t > n\tau_n)$, $\theta_k \neq \eta_k$ for at least one $k$ and 
%\end{eqnarray}
for some known $d \times d$ positive definite matrix $\Gamma$. Then, the asymptotic distribution in (a) continues to hold under the weaker assumptions (B7) and (SNR1). \\
\indent Moreover, if $I(\theta_k) = I_1$ and $I(\eta_k) = I_2$ for all $k \geq 1$, then (SNR3) in (b) and (c) can be relaxed to (SNR2). 
\end{remark}

\section{Proofs}
%In this section we present detailed proofs of all theorems and associated propositions and lemmas  stated in Sections \ref{sec: lse} and \ref{sec: mle}. 

\subsection{Proof of Theorem \ref{thm: lse1}} \label{subsec: proofmlelserate} 
We use the following lemma to prove this theorem. This is quoted from \cite{Wellner1996empirical}.

\begin{lemma} \label{lem: wvan1}
For each $n$, let $\mathbb{M}_n$ and $\tilde{\mathbb{M}}_n$ be stochastic processes indexed by a set $\mathcal{T}$. Let $\tau_n\ \text{(possibly random)} \in \mathcal{T}_n \subset \mathcal{T}$ and 
$d_n(b,\tau_n)$ be a map (possibly random) from $\mathcal{T}$ to $[0,\infty)$. Suppose that for every large $n$ and $\delta \in (0,\infty)$
\begin{eqnarray}
&& \sup_{\delta/2 < d_n(b,\tau_n) < \delta,\  b \in \mathcal{T}} (\tilde{\mathbb{M}}_n(b) - \tilde{\mathbb{M}}_n(\tau_n)) \leq -C\delta^2, \label{eqn: lemcon1} \\
&& E\sup_{\delta/2 < d_n(b,\tau_n) < \delta,\  b \in \mathcal{T}}  \sqrt{n} |\mathbb{M}_n(b) - \mathbb{M}_n(\tau_n) - (\tilde{\mathbb{M}}_n(b) - \tilde{\mathbb{M}}_n(\tau_n))| \leq  C\phi_{n}(\delta), \label{eqn: lemcon2}
\end{eqnarray}
for some $C>0$ and for function $\phi_n$ such that $\delta^{-\alpha}\phi_n(\delta)$ is decreasing in $\delta$ on $(0,\infty)$ for some $\alpha <2$. Let $r_n$ satisfy
\begin{eqnarray} \label{eqn: lemrn}
r_n^2  \phi(r_n^{-1}) \leq \sqrt{n}\ \ \text{for every  $n$}.
\end{eqnarray}
Further, suppose that the sequence $\{\hat{\tau}_n\}$ takes its values in $\mathcal{T}_n$ and 
satisfies $\mathbb{M}_n(\hat{\tau}_n) \geq \mathbb{M}_n(\tau_n) - O_P (r_n^{-2})$ for large enough $n$. Then,
%\begin{eqnarray}
$r_n d_{n}(\hat{\tau}_n,\tau_n) = O_P (1)$.
%\end{eqnarray} 
\end{lemma}

Recall that the least squares estimator $\hat{\tau}_{n,\text{LSE}}$ of $\tau_n$ from (\ref{eqn: optilse}). %can be obtained as 
%\begin{eqnarray}
% \hat{\tau}_{n,\text{lse}} &=& \arg \max_{b} M_{n}(b)\ \ \text{where,} \nonumber \\
% M_{n}(b) &=& \sum_{k=1}^{m} M_{k,n}(b),\ \ 
%M_{k,n}(b) = -\frac{1}{n}\bigg[ \sum_{t=1}^{nb}(X_{kt}-\hat{\mu}_{1k}(b))^2 + \sum_{t=nb+1}^{n} (X_{kt}-\hat{\mu}_{2k}(b))^2 \bigg],\ \ \ \ \  \  \label{eqn: lsedefn1pf}\\
%\hat{\mu}_{1k}(b) &=& \frac{1}{nb} \sum_{t=1}^{nb} X_{kt}\ \ \text{and}\ \ \hat{\mu}_{2k}(b) = \frac{1}{n(1-b)} \sum_{t=nb+1}^{n} X_{kt}. \label{eqn: lsedefn2pf}
%\end{eqnarray}
%Also recall $\{\mu_{1k}\}$ and $\{\mu_{2k}\}$ from Section \ref{sec: lse} and
%\begin{eqnarray}
%||\mu_1 - \mu_2||_2^2 = \sum_{k=1}^{m} (\mu_{1k} - \mu_{2k})^2. 
%\end{eqnarray}
For our purpose, we make use of the above lemma with $\mathbb{M}_n = M_n$, $\tilde{\mathbb{M}}_n = EM_n$,
$\mathcal{T} = [0,1]$, $\mathcal{T}_n = \{1/n,2/n,\ldots, (n-1)/n,1\} \cap [c^{*},1-c^{*}]$,  $d_n(b,\tau_n) = ||\mu_{1k} - \mu_{2k}||_2 \sqrt{|b - \tau_n|}$, $\phi_n(\delta) = \delta$, $\alpha = 1.5$, $r_n = \sqrt{n}$ and $\hat{\tau}_n = \hat{\tau}_{n,\text{LSE}}$.  Thus, to prove Theorem \ref{thm: lse1}, it is enough to establish that for some $C>0$,
\begin{eqnarray}
&&  E(M_n(b) - M_n(\tau_n)) \leq -C||\mu_{1} - \mu_{2}||_2^2 |b - \tau_n|\ \ \text{and} \label{eqn: lsecon1} \\
&& E\sup_{\delta/2 < d_n(b,\tau_n) < \delta,\  b \in \mathcal{T}}   |M_n(b) -M_n(\tau_n) - E(M_n(b) - M_n(\tau_n))| \leq  C\frac{\delta}{\sqrt{n}}. \label{eqn: lsecon2}
\end{eqnarray}

Note that the left hand side of (\ref{eqn: lsecon2}) is dominated by
\begin{eqnarray}
\left(E\sup_{\delta/2 < d_n(b,\tau_n) < \delta,\  b \in \mathcal{T}}   (M_n(b) -M_n(\tau_n) - E(M_n(b) - M_n(\tau_n)))^2\right)^{1/2}. \label{eqn: doobprelse}
\end{eqnarray}
By Doob's martingale inequality, (\ref{eqn: doobprelse}) is further dominated by
\begin{eqnarray}
(\text{Var}(M_n(b) -M_n(\tau_n)))^{1/2}\ \ \text{where $d_n(b,\tau_n) = \delta$}.
\end{eqnarray}
Thus, to prove (\ref{eqn: lsecon2}), it is enough to show that for some $C>0$,
\begin{eqnarray}
\text{Var}(M_n(b) -M_n(\tau_n)) \leq Cn^{-1} d_n^2(b,\tau_n). \label{eqn: lsecon2final}
\end{eqnarray}
Hence, it is enough to prove (\ref{eqn: lsecon1}) and (\ref{eqn: lsecon2final}) to establish Theorem \ref{thm: lse1}. 

Next, we write $(M_n(b) - M_n(\tau_n))$ as a sum of some processes, so that computing means and variances becomes easier. We first introduce 
some notation that facilitates the presentation. We also write $\tau$ for $\tau_n$.  

\noindent \textbf{Additional notation}.
\begin{eqnarray}
\hat{\mu}_{k}(a,b) &=& \frac{1}{n|b-a|} \sum_{t=n(a\wedge b)+1}^{n(a \vee b)} X_{kt}, \nonumber \\
T_{1k}(b) &=& b(\hat{\mu}_{1k}(b) - \hat{\mu}_{1k}(\tau))^2,\ \ T_{2k}(b) = - (1-\tau) (\hat{\mu}_{2k}(b) - \hat{\mu}_{2k}(\tau))^2,\nonumber \\
T_{3k}(b) &=& (\tau - b) \bigg [ (\hat{\mu}_{1k}(\tau) - \hat{\mu}_{k}(b,\tau))^2 - (\hat{\mu}_{2k}(b) - \hat{\mu}_{k}(b,\tau))^2\bigg ], \nonumber \\
N_{1k}(b) &=& (\hat{\mu}_{1k}(b)-E(\hat{\mu}_{1k}(b))),\ \ N_{2k}(b) = (\hat{\mu}_{2k}(b) - E(\hat{\mu}_{2k}(b))),\nonumber \\
 N_{3k}(a,b) &=& (\hat{\mu}_{k}(a,b) - E(\hat{\mu}_{k}(a,b)),\ \ N_{4k} = (E(\hat{\mu}_{1k}(\tau)) - E(\hat{\mu}_{2k}(\tau))).\nonumber  
 % &=& \frac{n-\tau}{n-t_1},\ \ c_1 = 1-c_2 \nonumber 
\end{eqnarray}
It is easy to see that
\begin{eqnarray}
 E(\hat{\mu}_{k}(b,\tau)) &=& E(\hat{\mu}_{1k}(\tau))I(b < \tau) + E(\hat{\mu}_{2k}(\tau))I(b > \tau) , \label{eqn: expectationa} \\
 E(\hat{\mu}_{2k}(b)) -E(\hat{\mu}_{1k}(\tau))  &=& - \left(\frac{1-\tau}{1-b}I(b<\tau) + I(b>\tau) \right) N_{4k}, \nonumber \\
 E(\hat{\mu}_{2k}(b)) -E(\hat{\mu}_{2k}(\tau))  &=&  \frac{\tau-b}{1-b}I(b<\tau)  N_{4k}, \nonumber \\
 E(\hat{\mu}_{1k}(b)) -E(\hat{\mu}_{1k}(\tau))  &=& - \frac{b-\tau}{b}I(b>\tau) N_{4k}, \nonumber \\
 E(\hat{\mu}_{1k}(b)) -E(\hat{\mu}_{2k}(\tau))  &=& \left(\frac{\tau}{b}I(b>\tau) + I(b < \tau)\right)N_{4k}. \label{eqn: expectationb} 
\end{eqnarray}
Assume $b < \tau$. Using above notations and (\ref{eqn: expectationa})-(\ref{eqn: expectationb}),  the following relations follows. % immediately. 
\begin{eqnarray}
T_{1k} (b) &=& b \tau^{-2} (\tau-b)^2 (N_{1k}^2 (b) + N_{3k}^2(b,\tau) + 2N_{1k}(b)N_{3k}(b,\tau)), \nonumber \\
T_{2k}(b) &=& -(1-\tau)(1-b)^{-2} (\tau-b)^2 (N_{2k}^2 (\tau) + N_{3k}^2(b,\tau) + N_{4k}^2 - 2N_{2k}(\tau)N_{3k}(b,\tau) \nonumber \\
&& \hspace{5 cm}- 2N_{3k}(b,\tau)N_{4k} + 2N_{2k}(\tau)N_{4k} ), \nonumber \\
T_{3k}(b) &=& -(\tau -b) (N_{2k}^2(b) - N_{1k}^2(\tau) + (1-\tau)^{-2}(1-b)^2 N_{4k}^2 -2N_{2k}(b)N_{3k}(b,\tau) \nonumber \\
&& \hspace{1 cm}-2(1-\tau)^{-1}(1-b) N_{3k}(b,\tau)N_{4k} + 2(1-\tau)^{-1}(1-b) N_{2k}(b) N_{4k} \nonumber \\
&& \hspace{7 cm} + 2N_{1k}(\tau)N_{3k}(b,\tau)). \label{eqn: t123ex}
\end{eqnarray}

\noindent Further, 
\begin{eqnarray} \label{eqn: t123}
&& M_{n}(b) - M_n(\tau)  = \sum_{k=1}^{n} (M_{k,n}(b) - M_{k,n}(\tau))  
= \sum_{k=1}^{m}(T_{1k}(b) + T_{2k}(b) + T_{3k}(b)).\  \  \ 
\end{eqnarray}

\noindent Thus, to prove (\ref{eqn: lsecon1}) and (\ref{eqn: lsecon2final}), we need to calculate the expectation and the variance of 
$N_{ik}N_{jk}$  $\forall i,j,k$. 

The following lemma proves useful to compute the expectation. Its proof is given in the Supplementary file. %Section \ref{subsec: expectation}. 

\begin{lemma} \label{lem: expectation}
Suppose $\sup_{k,t,n}{\rm{Var}}(X_{kt}) < \infty$. Then, for some $C>0$, 
\begin{eqnarray}
&& \sup_{k,b} E(N_{1k}^2(b)) \leq Cn^{-1},\ \ \sup_{k,b} E(N_{2k}^2(b)) \leq Cn^{-1},\ \ \sup_{k} E(N_{3k}^2(b,\tau)) \leq  C(n(\tau -b))^{-1}, \nonumber \\
&&  E(N_{4k}^2) = N_{4k}^2,\ \ \sup_{k, b < \tau} E(N_{1k}(b)N_{3k}(b,\tau)) = Cn^{-1},\ \sup_{k, b < \tau} E(N_{2k}(\tau)N_{3k}(b,\tau)) = Cn^{-1}, \nonumber \\
&& \sup_{k, b < \tau} E(N_{1k}(\tau)N_{3k}(b,\tau)) \leq Cn^{-1},\ \sup_{k,b<\tau} E(N_{2k}(b)N_{3k}(b,\tau)) \leq Cn^{-1}, \nonumber \\
&& \sup_{k,b} E(N_{1k}(b)N_{4k}) = 0,\ \sup_{k,b} E(N_{2k}(b)N_{4k}) = 0, \ \sup_{k,b} E(N_{3k}(b,\tau)N_{4k}) = 0. \nonumber 
\end{eqnarray}
\end{lemma}
 
Using Lemma \ref{lem: expectation}, for some $C,C_1>0$ we obtain
\begin{eqnarray}
\sum_{k=1}^{m}E T_{1k}(b), \sum_{k=1}^m ET_{2k}(b) \leq  C\frac{(\tau-b)m}{n},\ %\nonumber \\ 
\ \sum_{k=1}^{m} ET_{3k}(b) \leq C\frac{(\tau-b)m}{n} - C_1 (\tau-b) \sum_{k=1}^{m}N_{4k}^{2}. \nonumber 
\end{eqnarray}
Note that $\sum_{k=1}^{m}N_{4k}^{2} = ||\mu_1 - \mu_2||_2^2$. Hence, by (SNR1), for some $C>0$,
%\begin{eqnarray}
$E(M_n(b) - M_n(\tau)) \leq -C (\tau-b) ||\mu_1 - \mu_2||_2^2$.
%\end{eqnarray}
Thus, (\ref{eqn: lsecon1}) is established for $b<\tau$. A similar argument works for $b>\tau$.  

Next, we compute the variance, for which the following lemma proves useful. Its proof is given in the Supplementary file. %Section \ref{subsec: variance}.

\begin{lemma} \label{lem: variancelse}
Suppose (A1) holds. Then, for some $C>0$, 
\begin{eqnarray}
\sup_{k,b} {\rm{Var}}(N_{1k}^2(b)) \leq Cn^{-2},\ \ \sup_{k,b} {\rm{Var}}(N_{2k}^2(b)) \leq Cn^{-2} ,\ \ \sup_{k} {\rm{Var}}(N_{3k}^2(b,\tau)) \leq Cn^{-2} (\tau -b)^{-2},
 \nonumber \\
 \sup_{k} {\rm{Var}}(N_{1k}(b)N_{3k}(b,\tau)) \leq Cn^{-2} (\tau -b)^{-2},\ \ \sup_{k} {\rm{Var}}(N_{2k}(b)N_{3k}(b,\tau)) \leq Cn^{-2} (\tau -b)^{-2} \nonumber \\
  {\rm{Var}}(N_{4k}^2) = 0,\ \  {\rm{Var}}(N_{4k}N_{3k}(b,\tau)) \leq CN_{4k}^2 n^{-1} (\tau -b)^{-1},\ \ \sup_{b} {\rm{Var}}(N_{4k}N_{2k}(b)) \leq CN_{4k}^2 n^{-1}. \nonumber 
\end{eqnarray}
\end{lemma}

Hence, by Lemma \ref{lem: variancelse} and (SNR1), we have for some $C>0$
\begin{eqnarray}
\sum_{k=1}^{m}{\rm{Var}}(T_{1k}(b))  &\leq & C b^2 \tau^{-4} (\tau-b)^4 \frac{m}{n^2} (1 +(\tau-b)^{-2}) \nonumber \\
&\leq & Cn^{-1} (mn^{-1} ||\mu_1 - \mu_2||_2^{-2}) ||\mu_1 - \mu_2||_2^2 (\tau -b) %\nonumber \\
 \leq  Cn^{-1}d_{n}^{2}(b,\tau), \nonumber \\
\sum_{k=1}^{m}{\rm{Var}}(T_{2k}(b)) &\leq & C (\tau-b)^4 \bigg( \frac{m}{n^2}(1+(\tau - b)^{-2}) %\nonumber \\
 \hspace{0 cm} + \frac{1}{n}(\tau-b)^{-1}\sum_{k=1}^{m} N_{4k}^2  + \frac{1}{n}\sum_{k=1}^{m} N_{4k}^2 \bigg), \nonumber \\
& \leq & Cn^{-1} d_{n}^{2}(b,\tau) \nonumber 
\end{eqnarray}
and similarly $\sum_{k=1}^{m}{\rm{Var}}(T_{3k}(b)) \leq Cn^{-1} d_{n}^{2}(b,\tau)$. 
\noindent Thus, (\ref{eqn: lsecon2final}) is established for $b<\tau$, and a similar argument works for the case $b>\tau$. 

This completes the proof of Theorem \ref{thm: lse1}.  \qed

\subsection{Proof of Theorem \ref{thm: lse2}} \label{subsec: lse2}
\noindent \textbf{Proof of (a)}. Note that $P(\hat{\tau}_{n,\text{LSE}} \neq \tau) = P(|\hat{\tau}_{n,\text{LSE}} -\tau| \geq n^{-1}) \to 0$ 
since $||\mu_1 - \mu_2||_2 \to \infty$ and by Theorem \ref{thm: lse1}, $n||\mu_1 - \mu_2||_2^2(\hat{\tau}_{n,\text{LSE}}-\tau) = O_{P}(1)$. 
%\vskip 10pt
\vskip 5pt
\noindent \textbf{Proof of (b)}.
The following lemma from \cite{Wellner1996empirical} proves useful in this proof.  
\begin{lemma}  \label{lem: wvandis1}
Let $\mathbb{M}_n$ and $\mathbb{M}$ be two stochastic processes indexed by a metric space $\mathcal{T}$,
such that $\mathbb{M}_n \Rightarrow \mathbb{M}$ in $l^{\infty}(K)$ for every compact set $K \subset \mathcal{T}$ i.e.,
\begin{align}
\sup_{h \in K} |\mathbb{M}_n(h) - \mathbb{M}(h)| \stackrel{P}{\to} 0.
\end{align} 
Suppose that almost all sample paths $h \to \mathbb{M}(h)$ are upper semi-continuous and possess a unique maximum at a (random) point $\hat{h}$, which as a random map in $\mathcal{T}$ is tight. If the sequence $\hat{h}_n$ is uniformly tight and satisfies $\mathbb{M}_n(\hat{h}_n) \geq \sup_{n} \mathbb{M}_n(h) - o_{P}(1)$, then $\hat{h}_n \stackrel{\mathcal{D}}{\to} \hat{h}$ in $\mathcal{T}$.
\end{lemma}

%Let $[\cdot]$ denote the floor function on $\mathbb{R}$. 
To employ Lemma \ref{lem: wvandis1}, we consider 
 $\mathbb{M}_n(h) = n(M_n(b) - M_{n}(\tau))$ where $b = \tau +n^{-1}||\mu_1 - \mu_2||_2^{-2} h$ and $h \in \mathbb{R}$.  To prove Theorem \ref{thm: lse2}(b),  by Lemma \ref{lem: wvandis1}, it is enough to establish
\begin{align} \label{eqn: enoughlse2b}
\sup_{h \in K} |\mathbb{M}_n(h) + |h| - 2\gamma_{_\text{L,LSE}} B_hI(h<0) - 2\gamma_{_\text{R,LSE}} B_hI(h>0))| \to 0,
\end{align} 
as $||\mu_1 - \mu_2||_2 \to 0$, and for all compact subsets $K$ of $\mathbb{R}$. 
 
\noindent Note that by (\ref{eqn: t123}), we have
%\begin{align*}
$\mathbb{M}_n(h) = n\sum_{k=1}^{m} (T_{1k}(b) + T_{2k}(b) + T_{3k}(b))$.

%\end{align*}
\noindent %Let $K$ be a compacts set. 
First, we shall show that for any compact subset $K$ of $\mathbb{R}$ and as $||\mu_1 - \mu_2||_2 \to 0$,  % when $||\mu_1 - \mu_2||_2 \to 0$ or $c>0$
\begin{align} \label{eqn: l1dis}
& \sup_{h \in K}\bigg|n\sum_{k=1}^{m} T_{1k} (b)\bigg|, \ \ \  \sup_{h \in K}\bigg|n\sum_{k=1}^{m} T_{2k} (b)\bigg| \stackrel{P}{\to}  0\ \ \mbox{and} \\ & \sup_{h \in K}\bigg|n\sum_{k=1}^{m} \left(T_{3k} (b) + \frac{|\tau -b|(1-b)^2}{(1-\tau)^{2}} N_{4k}^2 - 2\frac{|\tau-b|(1-b)}{(1-\tau)}N_{3k}(b, \tau)N_{4k} \right)\bigg| \stackrel{P}{\to} 0. \label{eqn: l2dis}
\end{align}

\noindent \textbf{Proof of (\ref{eqn: l1dis}) and (\ref{eqn: l2dis})}. Note that by (\ref{eqn: t123ex}),  Lemmas \ref{lem: expectation}, \ref{lem: variancelse} and (SNR1), we have
\begin{align*}
&  E\sup_{h \in K}|n\sum_{k=1}^{m} T_{1k} (b)|  \\
 \leq  & C\sup_{h \in K} \bigg[ n (\tau-b)^2 \sum_{k=1}^{m}\left[ E(N_{1k}^2 (b)) + E(N_{3k}^2(b,\tau))  + 2E|N_{1k}(b)N_{3k}(b,\tau)|\right] \bigg]\\
 \leq  & C\sup_{h \in K} \bigg[ n (\tau-b)^2 \sum_{k=1}^{m}\left[ E(N_{1k}^2 (b)) + E(N_{3k}^2(b,\tau))  + 2(\mbox{Var}(N_{1k}(b)N_{3k}(b,\tau)))^{1/2}\right] \bigg]\\
 \leq & C\frac{m }{n ||\mu_1 - \mu_2||_2^2} \to 0\ \ \text{and} \\
%\end{align*}
%and  %$E|n\sum_{k=1}^{m} T_{2k} (b)| \to 0$ uniformly in $h \in K$.
%\begin{align*}
& E\sup_{h \in K} |n\sum_{k=1}^{m} T_{2k}(b)| \\
\leq  & C \sup_{h \in K} \bigg[ n(\tau-b)^2 \sum_{k=1}^{m} \bigg(E(N_{2k}^2 (\tau)) + E(N_{3k}^2(b,\tau)) + E(N_{4k}^2) + 2 (\mbox{Var}(N_{2k}(\tau)N_{3k}(b,\tau)))^{1/2}  \\
& \hspace{2.5 cm} + 2|N_{4k}|(\mbox{Var}(N_{3k}(b,\tau))^{1/2} + 2|N_{4k}|(\mbox{Var}(N_{2k}(\tau))^{1/2} \bigg) \bigg] \\
& \leq  C\bigg[\frac{m}{n||\mu_1 - \mu_2||_2^2} + \frac{\sqrt{m}}{n^{3/2}||\mu_1 - \mu_2||_2^3} \bigg] \to 0.
\end{align*}
This completes the proof of (\ref{eqn: l1dis}). A similar argument works for (\ref{eqn: l2dis}).  

Moreover, it is easy to see that
\begin{align}
\sup_{h \in K} \bigg|-\frac{n|\tau -b|(1-b)^2}{(1-\tau)^{2}} \sum_{k=1}^{m}N_{4k}^2 + |h| \bigg| \to 0. \label{eqn: nonran}
\end{align}
Therefore, by (\ref{eqn: l1dis}), (\ref{eqn: l2dis}) and (\ref{eqn: nonran}), to prove (\ref{eqn: enoughlse2b}), it remains to establish
\begin{align}
\sup_{h \in K} \bigg| \frac{n|\tau-b|(1-b)}{(1-\tau)}\sum_{k=1}^{m} N_{3k}(b, \tau)N_{4k} -2(\gamma_{_\text{L,LSE}} I(h<0) +\gamma_{_\text{R,LSE}} I(h>0))B_h\bigg| \stackrel{P}{\to} 0.  \label{eqn: t1234}
\end{align}
Since $\{\frac{n|\tau-b|(1-b)}{(1-\tau)}\sum_{k=1}^{m} N_{3k}(b, \tau)N_{4k}\}$ is tight, to prove (\ref{eqn: t1234}), it is enough to establish finite dimensional weak convergence. We shall show one dimensional convergence 
%\begin{align}
%\bigg| \frac{n|\tau-b|(1-b)}{(1-\tau)} \sum_{k=1}^{m} N_{3k}(b, \tau)N_{4k} -h\gamma_{_\text{L,LSE}} \mathcal{N}(0,1) I(h<0) - h\gamma_{_\text{R,LSE}} \mathcal{N}(0,1)I(h>0))\bigg| \stackrel{P}{\to} 0. \label{eqn: t1233}
%\end{align}
\begin{align}
 \frac{n|\tau-b|(1-b)}{(1-\tau)} \sum_{k=1}^{m} N_{3k}(b, \tau)N_{4k} \stackrel{\mathcal{D}}{\to} 2(h\gamma_{_\text{L,LSE}}  I(h<0) + h\gamma_{_\text{R,LSE}}I(h>0)) \mathcal{N}(0,1). \label{eqn: t1233}
\end{align}
for $h \in K$, using a univariate central limit theorem. 

Let $t_1^{*} = n\tau + ||\mu_1 - \mu_2||_2^{-1} hI(h < 0)$ and $t_2^{*} = n\tau + ||\mu_1 - \mu_2||_2^{-2} hI(h > 0)$. Note that $|t_1-t_2| = ||\mu_1 - \mu_2||_2^{-2} h$. Also let,
$X_{t}^{*} = \sum_{k=1}^{m}(\mu_{1k}-\mu_{2k})(X_{kt} - E(X_{kt}))\ \forall t.$ Observe that $\{X_t^{*}\}$ are independent over $t$.  Next, 
\begin{align}
\frac{n|\tau-b|(1-b)}{(1-\tau)}\sum_{k=1}^{n} N_{3k}(b, \tau)N_{4k} =  \frac{1-b}{1-\tau} \sum_{t= t^*_1 +1}^{t^*_2} X^{*}_t. \label{eqn: xtstar}
\end{align}
By the Lyapunov central limit theorem, for a given $h \neq 0$,  the right hand side of (\ref{eqn: xtstar}) converges to a Normally distributed random 
variable if
\begin{align}
\bigg(\sum_{t= t^*_1 +1}^{t^*_2}E( X^{*2}_t) \bigg)^{-3/2} \bigg(\sum_{t= t^*_1 +1}^{t^*_2} E|X^{*}_t|^3 \bigg) \to 0.  \label{eqn: asp}
\end{align}
(\ref{eqn: asp}) follows since
\begin{eqnarray}
\sum_{t= t^*_1 +1}^{t^*_2} E|X^{*}_t|^3 \leq Ch \frac{\sum_{k=1}^{m} |\mu_{1k} - \mu_{2k}|^3}{||\mu_1-\mu_2||_2^2} \to 0\ \ \text{and}\ \ \ 
\sum_{t= t^*_1 +1}^{t^*_2}E( X^{*}_t)^2 \geq Ch||\mu_1-\mu_2||_2^{-3}.\ \ \ \  
\end{eqnarray}
Therefore, the sequence of random variables in (\ref{eqn: xtstar}) converges weakly to a normal random variable with variance $h \gamma_{_\text{L,LSE}}^2 I(h<0) + h \gamma_{_\text{R,LSE}}^2 I(h>0)$. This completes the proof of one dimensional convergence in (\ref{eqn: t1233}). Similarly one can show finite dimensional convergence using a multivariate central limit theorem. This completes the proof of (\ref{eqn: t1234}) and hence the proof of (\ref{eqn: enoughlse2b}). Thus, Theorem \ref{thm: lse2}(b) is established.
\vskip 5pt
\noindent \textbf{Proof of (c)}. It is easy to see that (\ref{eqn: l1dis}) and (\ref{eqn: l2dis})  hold if $||\mu_1 - \mu_2||_2 \to c>0$. Moreover, 
\begin{align*}
& -\frac{n|\tau -b|(1-b)^2}{(1-\tau)^{2}} \sum_{k=1}^{m} N_{4k}^2 + 2\frac{n|\tau-b|(1-b)}{(1-\tau)}\sum_{k=1}^{m}N_{3k}(b, \tau)N_{4k} \\
=& -\frac{h(1-b)^2}{(1-\tau)^{2}||\mu_1 - \mu_2||_2^2} \sum_{k=1}^{m} N_{4k}^2 + 2\frac{h(1-b)}{(1-\tau)||\mu_1 - \mu_2||_2^2}\sum_{k=1}^{m}N_{3k}(b, \tau)N_{4k} \\
= & A_n(h) + B_n(h),\ \ 
\end{align*}
where
\begin{align*}
A_n(h) &= -\frac{|h|(1-b)^2}{(1-\tau)^{2}||\mu_1 - \mu_2||_2^2} \sum_{k \in \mathcal{K}_0} N_{4k}^2 + 2\frac{|h|(1-b)}{(1-\tau)||\mu_1 - \mu_2||_2^2}\sum_{k \in \mathcal{K}_0}N_{3k}(b, \tau)N_{4k}, \\
B_n(h) &= -\frac{|h|(1-b)^2}{(1-\tau)^{2}||\mu_1 - \mu_2||_2^2} \sum_{k \in \mathcal{K}_n} N_{4k}^2 + 2\frac{|h|(1-b)}{(1-\tau)||\mu_1 - \mu_2||_2^2}\sum_{k \in \mathcal{K}_n}N_{3k}(b, \tau)N_{4k} \\
& = B_{1n}(h) + B_{2n}(h)\ \ \ \text{say}.
\end{align*}
As $\mathcal{K}_0$ is a finite set, by (A4) and (A5), one can easily see that for $h \in K$,
\begin{align*}
A_n(c^2(h+1)) -A_n(c^2h) \Rightarrow \sum_{k \in \mathcal{K}_n^{c}} \bigg[(Z_{kh} -\mu_{1k}^{*})^2 -(Z_{kh} -\mu_{2k}^{*})^2 \bigg],
\end{align*}
where $\{Z_{kh}\}$ are independently distributed with $Z_{kh} \stackrel{d}{=} X_{1k}^{*}I(h \leq 0) + X_{2k}^{*}I(h > 0)$.
\vskip 5pt
\noindent Recall  $c_1^2 = \lim \sum_{k \in \mathcal{K}_n}(\mu_{1k} - \mu_{2k})^2$, $h \in K$. Then clearly,  $B_{1n}(h) \Rightarrow - |h|c^{-2}c_1^2$.  Moreover,
\begin{align*}
B_{2n}(c^2 (h+1)) - B_{2n}(C^2 h) = 2\frac{1-b}{1-\tau} \sum_{k \in \mathcal{K}_n} (\mu_{1k} - \mu_{2k})(X_{kt} - E(X_{kt}))
\end{align*}
which weakly converges to a normal random variable if 
\begin{align}
\bigg(\sum_{k \in \mathcal{K}_n} (\mu_{1k} - \mu_{2k})^2 E(X_{kt} - E(X_{kt}))^2 \bigg)^{3/2} \bigg(\sum_{k \in \mathcal{K}_n} |\mu_{1k} - \mu_{2k}|^3 E|X_{kt} - E(X_{kt})|^3 \bigg) \to 0. \label{eqn: ct123}
\end{align}
By (A1), the above convergence in (\ref{eqn: ct123}) holds if
%\begin{align}
$\sum_{k \in \mathcal{K}_n} |\mu_{1k} - \mu_{2k}|^3 \to 0$
%\end{align}
and this is guaranteed by (A5). Hence,
\begin{align}
B_{2n}(c^2 (h+1)) - B_{2n}(C^2 h)  \stackrel{\mathcal{D}}{\to} 2 (\gamma_{_\text{L,LSE}}^{*2}I(h<0) + \gamma_{_\text{R,LSE}}^{*2}I(h<0) ) \mathcal{N}(0,1). %\label{eqn: ct1234}
\end{align}
Similarly one can establish finite dimensional weak convergence of $B_{2n}(c^2 (h+1)) - B_{2n}(C^2 h)$. Moreover,  $B_{2n}(c^2 (h+1)) - B_{2n}(C^2 h)$ is tight.  Hence,
\begin{align}
B_{2n}(c^2 (h+1)) - B_{2n}(C^2 h)  \Rightarrow 2 (\gamma_{_\text{L,LSE}}^{*2}I(h<0) + \gamma_{_\text{R,LSE}}^{*2}I(h<0) ) W_h \label{eqn: ct1234}
\end{align}
where $\{W_h\}$ are all independent standard normal random variables.  This completes the proof of Theorem \ref{thm: lse2}(c).  \qed

\subsection{Proof of Theorem \ref{thm: lse4}} \label{subsec: lse4}
Theorem \ref{thm: lse4}(a) follows from Theorem \ref{thm: lse3}. 
\vskip 5pt
\noindent \textbf{Proof of (b)}.  (\ref{eqn: l1dis}), (\ref{eqn: l2dis}) and (\ref{eqn: nonran}) are easy to establish under (D1), (D2) and (SNR1). Let $t_1^{*} = n\tau + ||\mu_1 - \mu_2||_2^{-1} hI(h < 0)$ and $t_2^{*} = n\tau + ||\mu_1 - \mu_2||_2^{-2} hI(h > 0)$. Note that $|t_1-t_2| = ||\mu_1 - \mu_2||_2^{-2} h$. Also let,
$X_{t}^{*} = \sum_{k=1}^{m}(\mu_{1k}-\mu_{2k})(X_{kt} - E(X_{kt}))\ \forall t.$ Note that, 
\begin{align}
\frac{n|\tau-b|(1-b)}{(1-\tau)}\sum_{k=1}^{n} N_{3k}(b, \tau)N_{4k} =  \frac{1-b}{1-\tau} \sum_{t= t^*_1 +1}^{t^*_2} X^{*}_t. \label{eqn: xtstar1}
\end{align}
Now we shall establish  weak convergence of the process in (\ref{eqn: xtstar1}).   Let $t_{1j}^{*} = n\tau + ||\mu_1 - \mu_2||_2^{-1} h_j I(h < 0)$, $t_{2j}^{*} = n\tau + ||\mu_1 - \mu_2||_2^{-2} h_j I(h > 0)$ and $h_j \in K$, a compact subset of $\mathbb{R}$. By (D3) and for some $C>0$, we have
\begin{align*}
\bigg|\text{Cum}\bigg( \sum_{t= t^*_{1j} +1}^{t^*_{2j}} X^{*}_t:\ 1 \leq j \leq r\bigg)\bigg| \leq C \frac{\sum_{k=1}^{m} |\mu_{1k}-\mu_{2k}|^r}{||\mu_1-\mu_2||_2^2} \to 0\ \ \forall r >2.
\end{align*}
Also, 
%\begin{eqnarray}
$\text{Cum}\bigg( \sum_{t_1= t^*_{1j_1} +1}^{t^*_{2j_1}} X^{*}_t, \sum_{t_1= t^*_{1j_2} +1}^{t^*_{2j_2}} X^{*}_t\bigg) \to \gamma_{(h_{j_1},h_{j_2}),\text{DEP,LSE}},\ \ \forall 1 \leq j_1, j_2 \leq r$.
%\end{eqnarray}
This proves finite dimensional weak convergence of the process (\ref{eqn: xtstar1}). Therefore, by the tightness of (\ref{eqn: xtstar1}), Theorem \ref{thm: lse4}(b) is established.
\\
\indent Theorem \ref{thm: lse4}(c) can be easily established if we use similar modifications on the proof of Theorem \ref{thm: lse2}(c) as we have done above on the proof of Theorem \ref{thm: lse2}(b) for proving Theorem \ref{thm: lse4}(b). 
%Proof of Theorem \ref{thm: lse4}(c) is along the same line of the proof of Theorem \ref{thm: lse4}(b). 
%\vskip 3pt
\\ \indent This completes the proof of Theorem \ref{thm: lse4}. \qed

\subsection{Justification of Example \ref{rem: lsema}} \label{subsec: lsema}
Example \ref{rem: lsema}(a) follows directly from Theorem \ref{thm: lse4}(a).  We use the following lemma to prove Example \ref{rem: lsema}(b). This is quoted from \cite{BD2009}.
\begin{lemma}  \label{lem: bdlem}
Let $\{X_n: n\geq 1\}$ and $\{Y_{nj}: n,j \geq 1\}$ be random variables such that 
\vskip 3pt
\noindent (i) $Y_{nj} \stackrel{\mathcal{D}}{\to} Y_j$ as $n \to \infty$ for each $j=1,2,\ldots$,
\vskip 3pt
\noindent (ii) $Y_j \stackrel{\mathcal{D}}{\to} Y$ as $j \to \infty$, and
\vskip 3pt
\noindent (iii) $\lim_{j \to \infty} \lim_{n \to \infty}  P(|X_n - Y_{nj}| > \epsilon) = 0$ for every $\epsilon >0$.
\vskip 3pt
\noindent Then $X_n \stackrel{\mathcal{D}}{\to} Y$ as $n \to \infty$. 
\end{lemma}
Let $\tilde{\varepsilon}_{k,t,C} = \varepsilon_{k,t}I(|\varepsilon_{k,t}|\leq C) - E(\varepsilon_{k,t}I(|\varepsilon_{k,t}|\leq C))$, $Y_{kt,C} = \sum_{j=0}^{\infty} a_{k,j}\varepsilon_{k,t-j,C}$ and $X_{kt,C} = \mu_{1k}I(t \leq n\tau_n) + \mu_{2k}I(t> n\tau_n)+Y_{kt,C}$ for all $k,t$ and $C>0$.  Note that $\{X_{kt,C}\}$ satisfies (D1), (D2) and (D3). Let $\text{Var}(\tilde{\varepsilon}_{k,t,C}) = \sigma_{k\epsilon,C}^2$. Therefore, conclusion of Example \ref{rem: lsema}(b) hold for the process $\{X_{kt,C}\}$ with
%\begin{align}
$\text{Cum}(X_{kt_1,C},X_{kt_2,C}) = \sigma_{k\epsilon,C}^2 \big(\sum_{j=0}^{\infty} a_{k,j} a_{k,j+|t_2-t_1|} \big)$, $t_1,t_2 \in \mathbb{Z}, k \geq 1, C>0$.
%\end{align}
This establishes Lemma \ref{lem: bdlem}(i) for all $C>0$. Moreover, Lemma \ref{lem: bdlem}(ii) holds as $\sigma_{k\epsilon,C}^2 \to \sigma_{k\epsilon}^2$ as $C \to \infty$ and for all $k \geq 1$. 
\\
\indent Let $t_1^{*} = n\tau + ||\mu_1 - \mu_2||_2^{-1} hI(h < 0)$, $t_2^{*} = n\tau + ||\mu_1 - \mu_2||_2^{-2} hI(h > 0)$ and $h \in {K}$, a compact subset of $\mathbb{R}$.  Therefore, 
\begin{align*}
& P \bigg(|\sum_{t = t_1^* +1}^{t_2^*} \sum_{k=1}^{m}(\mu_{1k} - \mu_{2k}) (Y_{tk} - Y_{tk,C})| > \epsilon \bigg) \\
= & P \bigg(|\sum_{t = t_1^* +1}^{t_2^*} \sum_{k=1}^{m}\sum_{j=0}^{\infty}(\mu_{1k} - \mu_{2k}) a_{k,j}(\varepsilon_{k,t-j,C} - \varepsilon_{k,t-j})| > \epsilon \bigg) \\
& \leq \epsilon^{-2}  E\bigg|\sum_{t = t_1^* +1}^{t_2^*} \sum_{k=1}^{m}\sum_{j=0}^{\infty}(\mu_{1k} - \mu_{2k}) a_{k,j}(\varepsilon_{k,t-j,C} - \varepsilon_{k,t-j}) \bigg|^2 \\
\leq & C \frac{\sum_{k=1}^{m}\sum_{j=0}^{\infty}(\mu_{1k} - \mu_{2k})^2 a_{k,j}^2|\sigma_{k\epsilon,C}^2- \sigma_{k\epsilon}^2|}{||\mu_1 - \mu_2||_2^2}.
\end{align*}
Therefore, as $\lim_{C \to \infty} |\sigma_{k\epsilon,C}^2- \sigma_{k\epsilon}^2| =0$,  we have
%\begin{align*}
$\lim_{C \to \infty} \lim_{n \to \infty} P \big(|\sum_{t = t_1^* +1}^{t_2^*} \sum_{k=1}^{m}(\mu_{1k} - \mu_{2k}) (Y_{tk} - Y_{tk,C})| > \epsilon \big) = 0$.
%\end{align*}
Hence, Lemma \ref{lem: bdlem}(iii) holds. 

Hence, by Lemma \ref{lem: bdlem},  Example \ref{rem: lsema}(b) follows. \qed

\subsection{Proof of Theorem \ref{thm: mle1g}} \label{subsec: mle1proof}
We prove Theorem \ref{thm: mle1g} for $d=1$. Similar arguments work for finite $d >1$. 
\noindent  We employ the following lemma which follows easily from Lemma \ref{lem: wvan1}. 
\begin{lemma} \label{lem: wvan2}
For each $n$,  let $\mathbb{M}_n$ and $\mathbb{N}_n$  be stochastic processes on $\mathcal{T}$. Suppose $\tau_n$, $d_{n}(\cdot,\tau_n)$ and $r_n$ are as described in Lemma \ref{lem: wvan1} and
\begin{eqnarray} \label{eqn: wvandom}
\lim P\left [\mathbb{N}_n(b) - \mathbb{N}_n(\tau_n) \leq C (\mathbb{M}_n(b) - \mathbb{M}_n(\tau_n))\ \ \forall b \in \mathcal{T} \right] = 1.
\end{eqnarray}
$\mathbb{M}_n$ satisfies (\ref{eqn: lemcon1}) and (\ref{eqn: lemcon2}).  Further suppose that the sequence $\{\hat{\tau}_n\}$ takes  its values in $\mathcal{T}_n$ and satisfies $\mathbb{N}_n(\hat{\tau}_n) \geq \mathbb{N}_n(\tau_n) - O_P (r_n^{-2})$ for large enough $n$. Then
%\begin{eqnarray}
$r_n d_{n}(\hat{\tau}_n,\tau_n)) = O_P (1)$.
%\end{eqnarray} 
\end{lemma}
\vskip 5pt

Recall  $\hat{\tau}_{n, \text{MLE}}$ from (\ref{eqn: taumle}). 
%\begin{eqnarray}
%&& \hat{\tau}_{n, \text{MLE}} = \arg \max_{b} \sum_{k=1}^{m} L_{k,n}(b)\ \ \text{where,} \nonumber \\
%&& L_{k,n}(b) = \frac{1}{n}\sum_{t=1}^{nb}\log \mathbb{P}_{\hat{\theta}_k(b)}(X_{kt}) + \frac{1}{n}\sum_{t=nb+1}^{n}\log \mathbb{P}_{\hat{\eta}_{k(b)}}(X_{kt}), \nonumber \\
%&& \sum_{t=1}^{nb} \frac{\partial }{\partial \theta} \log \mathbb{P}_{{\theta}}(X_{kt}) \bigg|_{\theta = \hat{\theta}_k (b)} = \sum_{t=nb+1}^{n} \frac{\partial }{\partial \theta}\log \mathbb{P}_{\hat{\eta}_{k(b)}}(X_{kt}) \bigg|_{\eta = \hat{\eta}_k(b)} = 0 \nonumber
%\end{eqnarray}
%such that $\hat{\theta}_{k}(b)$ and $\hat{\eta}_{k}(b)$ are respectively consistent for $\theta$ and $\eta$.  Existence of $\{\hat{\theta}_{k}(b), \hat{\eta}_{k}(b) \}$ is guaranteed by Assumptions (B1)-(B3). 
To employ Lemma \ref{lem: wvan2}, we take $\mathcal{T}$ and $\mathcal{T}_n$ as in the proof of Theorem \ref{thm: lse1} and we also have
\begin{eqnarray} \label{eqn: Ndefn}
\mathbb{N}_n(b) &=& \sum_{k=1}^{n} \mathbb{L}_{k,n}(b).
\end{eqnarray}
\noindent Next, we obtain the process $\mathbb{M}_n$. To this end we define the following processes. 
We assume $b< \tau$. Similar arguments work for $b>\tau$. 
\begin{eqnarray}
\mathbb{M}_{1k}(b)  &=& \bigg(\frac{1}{nb} \sum_{t=1}^{nb} \frac{\partial}{\partial \theta_k} \log \mathbb{P}_{\theta_k}(X_{kt}) -  \frac{1}{n\tau} \sum_{t=1}^{n\tau} \frac{\partial}{\partial \theta_k} \log \mathbb{P}_{\theta_k}(X_{kt})\bigg)^2, \label{eqn: m1}  \\
\mathbb{M}_{2k}(b)  &=& \bigg(\frac{1}{n(1-b)} \sum_{t=nb+1}^{n} \frac{\partial}{\partial \theta_k} \log \mathbb{P}_{\theta_k}(X_{kt}) -  \frac{1}{n(1-\tau)} \sum_{t=n\tau +1}^{n} \frac{\partial}{\partial \theta_k} \log \mathbb{P}_{\theta_k}(X_{kt})\bigg)^2, \nonumber \\
\mathbb{M}_{3k}(b)  &=& (\tau-b)\bigg|\frac{1}{nb}\sum_{t=1}^{nb} \frac{\partial}{\partial \theta_k} \log \mathbb{P}_{\theta_k}(X_{kt}) \bigg|\bigg|\frac{1}{n(\tau-b)}\sum_{t=nb+1}^{n\tau} \frac{\partial}{\partial \theta_k} \log \mathbb{P}_{\theta_k}(X_{kt}) \bigg|, \nonumber \\
\mathbb{M}_{4k}(b)  &=& (\tau-b)\bigg|\frac{1}{nb}\sum_{t=1}^{nb} \frac{\partial}{\partial \theta_k} \log \mathbb{P}_{\theta_k}(X_{kt}) \bigg|^2 \bigg(\frac{1}{n(\tau-b)}\sum_{t=nb+1}^{n\tau} G_2(X_{kt}) \bigg), \nonumber \\
\mathbb{M}_{5k}(b)  &=& (\tau-b)\bigg|\frac{1}{n(1-b)}\sum_{t=nb+1}^{n} \frac{\partial}{\partial \eta_k} \log \mathbb{P}_{\eta_k}(X_{kt}) \bigg|\bigg|\frac{1}{n(\tau-b)}\sum_{t=nb+1}^{n\tau} \frac{\partial}{\partial \eta_k} \log \mathbb{P}_{\eta_k}(X_{kt}) \bigg|, \nonumber \\
\mathbb{M}_{6k}(b)  &=& (\tau-b)\bigg|\frac{1}{n(1-b)}\sum_{t=nb+1}^{n} \frac{\partial}{\partial \eta_k} \log \mathbb{P}_{\eta_k}(X_{kt}) \bigg|^2 \bigg(\frac{1}{n(\tau-b)}\sum_{t=nb+1}^{n\tau} G_2(X_{kt}) \bigg), \nonumber \\
\mathbb{M}_{7k}(b) &=&  \frac{1}{n}\sum_{t=nb+1}^{n\tau} (\log \mathbb{P}_{\eta_k}(X_{kt}) - \log \mathbb{P}_{\theta_k}(X_{kt})). \label{eqn: m7}
\end{eqnarray}
Note that for $1 \leq i \leq 7$ and $k \geq 1$, $\mathbb{M}_{ik}(\tau) =0$.  Let 
\begin{eqnarray}
\mathbb{M}_{k,n}(b) = \sum_{i=1}^{7} \mathbb{M}_{ik}(b)\ \text{and}\ \mathbb{M}_{n}(b) = \sum_{k=1}^{m}\mathbb{M}_{k,n}(b).  \label{eqn: mndefn}
\end{eqnarray}
\begin{lemma} \label{lem: Mndom}
Suppose (B3)-(B5), (B8), (B9) hold and $\log m(n) = o(n)$. Then, (\ref{eqn: wvandom}) is satisfied for the processes $\mathbb{N}_n(b)$ and $\mathbb{M}_{n}(b)$ defined in (\ref{eqn: Ndefn}) and (\ref{eqn: mndefn}). 
\end{lemma}

\noindent Its proof is given in the Supplementary file. 

\noindent Hence, by Lemma \ref{lem: wvan2}, the proof of Theorem \ref{thm: mle1g} will be complete once we show that $\mathbb{M}_n$ and $\tilde{\mathbb{M}}_n = E\mathbb{M}_n$ satisfy (\ref{eqn: lemcon1}) and (\ref{eqn: lemcon2}) with $d_{n}(b,\tau) = \sqrt{|b-\tau|}||\theta - \eta||_2$, $r_n = \sqrt{n}$ and $\phi_n(\delta) = \delta$.
 
\noindent \textbf{Proof of (\ref{eqn: lemcon1})}. Note that $\{\frac{\partial}{\partial \theta_k} \log \mathbb{P}_{\theta_k}(X_{kt}): t \leq n\tau\}$ are i.i.d. random variables with mean $0$. Also by (B4) and (B5), we have
\begin{eqnarray}
\sup_{k} \text{Var}\left(\frac{\partial}{\partial \theta_k} \log \mathbb{P}_{\theta_k}(X_{k1})\right) = \sup_{k} E\left(-\frac{\partial^2}{\partial \theta_k^2} \log \mathbb{P}_{\theta_k}(X_{k1})\right) \leq C \sup_{k} E(G_2(X_{k1})) < \infty. \nonumber 
\end{eqnarray}
Therefore,
%\begin{eqnarray}
$E\mathbb{M}_{1k}(b) \leq C((nb)^{-1} - (n\tau)^{-1}) \leq \frac{C}{n}(\tau -b)$. 
%\end{eqnarray}
\noindent Similarly, it is easy to see that $E\mathbb{M}_{2k}(b) \leq \frac{C}{n}(\tau -b)$. 
\noindent Next,
\begin{eqnarray}
\mathbb{M}_{3k}(b)  &=& (\tau-b)E\bigg|\frac{1}{nb}\sum_{t=1}^{nb} \frac{\partial}{\partial \theta_k} \log \mathbb{P}_{\theta_k}(X_{kt}) \bigg|\bigg|\frac{1}{n(\tau-b)}\sum_{t=nb+1}^{n\tau} \frac{\partial}{\partial \theta_k} \log \mathbb{P}_{\theta_k}(X_{kt}) \bigg| \nonumber \\
& \leq & (\tau-b) \sqrt{E\left(\frac{1}{nb}\sum_{t=1}^{nb} \frac{\partial}{\partial \theta_k} \log \mathbb{P}_{\theta_k}(X_{kt})  \right)^2} \sqrt{E\left(\frac{1}{n(\tau-b)}\sum_{t=nb+1}^{n\tau} \frac{\partial}{\partial \theta_k} \log \mathbb{P}_{\theta_k}(X_{kt}) \right)^2}. \nonumber
\end{eqnarray}
Using similar arguments,  given for $\mathbb{M}_{1k}(b)$, we have some $C>0$ such that
\begin{eqnarray}
 E\left(\frac{1}{nb}\sum_{t=1}^{nb} \frac{\partial}{\partial \theta_k} \log \mathbb{P}_{\theta_k}(X_{kt})  \right)^2 \leq Cn^{-1},\ \  
 E\left(\frac{1}{n(\tau-b)}\sum_{t=nb+1}^{n\tau} \frac{\partial}{\partial \theta_k} \log \mathbb{P}_{\theta_k}(X_{kt}) \right)^2  \leq C. \nonumber
\end{eqnarray}
Therefore,
%\begin{eqnarray}
$E\mathbb{M}_{3k}(b) \leq \frac{C}{\sqrt{n}}(\tau -b)$. 
%\end{eqnarray}
Similarly, $E\mathbb{M}_{5k}(b) \leq \frac{C}{\sqrt{n}}(\tau -b)$.

\noindent Now,
\begin{eqnarray}
\mathbb{M}_{4k}(b)  &=& (\tau-b)\bigg|\frac{1}{nb}\sum_{t=1}^{nb} \frac{\partial}{\partial \theta_k} \log \mathbb{P}_{\theta_k}(X_{kt}) \bigg|^2 \bigg(\frac{1}{n(\tau-b)}\sum_{t=nb+1}^{n\tau} G_2(X_{kt}) \bigg) \nonumber \\
& \leq & (\tau -b) \sqrt{E\left(\frac{1}{nb}\sum_{t=1}^{nb} \frac{\partial}{\partial \theta_k} \log \mathbb{P}_{\theta_k}(X_{kt}) \right)^4} \sqrt{E\left( \frac{1}{n(\tau-b)}\sum_{t=nb+1}^{n\tau} G_2(X_{kt})\right)^2}. \nonumber 
\end{eqnarray}
Note that by (B4)-(B6), we get
\begin{eqnarray}
 && E\left(\frac{1}{nb}\sum_{t=1}^{nb} \frac{\partial}{\partial \theta_k} \log \mathbb{P}_{\theta_k}(X_{kt}) \right)^4 \nonumber \\
&=& \frac{1}{(nb)^4}\bigg[\sum_{t=1}^{nb} E\left(\frac{\partial}{\partial \theta_k} \log \mathbb{P}_{\theta_k}(X_{kt}) \right)^4 + 6\bigg(\sum_{t=1}^{nb}E\bigg(\frac{\partial}{\partial \theta_k} \log \mathbb{P}_{\theta_k}(X_{kt}) \bigg)^2  \bigg)^2\bigg] \nonumber \\
& \leq & Cn^{-2} \nonumber
\end{eqnarray}
and $E\left( \frac{1}{n(\tau-b)}\sum_{t=nb+1}^{n\tau} G_2(X_{kt})\right)^2 \leq C$. Therefore, $E\mathbb{M}_{4k}(b) \leq C(\tau-b)n^{-1}$. 
\noindent Similarly, $E\mathbb{M}_{6k}(b) \leq C(\tau-b)n^{-1}$.
\noindent Next, consider $\mathbb{M}_{7k}(b)$.  By (B4) and (B5), it is easy to see that there is $C>0$ (independent of $k$) such that
%\begin{eqnarray}
$E(\mathbb{M}_{7k}(b)) \leq -C(\tau-b)(\theta_k - \eta_k)^2\ \ \forall k\geq 1$.
%\end{eqnarray}
\\
\indent Thus, combining $E(\mathbb{M}_{ik}(b))$ for $1 \leq i \leq 7$, we have
%\begin{eqnarray}
$E(\mathbb{M}_n(b) - \mathbb{M}_n(\tau)) \leq C(\tau -b) (mn^{-1} + mn^{-1/2} - ||\theta - \eta||_2^2)$. 
%\end{eqnarray}
Hence by (SNR2), $E(\mathbb{M}_n(b) - \mathbb{M}_n(\tau)) \leq -Cd_{n}^2(b,\tau)$ for some $C>0$. This completes the proof of  (\ref{eqn: lemcon1}).

\noindent \textbf{Proof of  (\ref{eqn: lemcon2})}.  As we have argued in the proof of Theorem \ref{thm: lse1}, it is enough to show
%\begin{align*}
$\text{Var}(\mathbb{M}_n(b) -\mathbb{M}_n(\tau)) \leq Cn^{-1} d_n^2(b,\tau_n)$. 
%\end{align*}
This is equivalent to establishing
\begin{align}
\sum_{k=1}^{m}\text{Var}(\mathbb{M}_{ik}(b))  \leq Cn^{-1} d_n^2(b,\tau_n)\ \ \forall 1 \leq i \leq 7.  \label{eqn: lsefinalmle}
\end{align}
%\begin{equation}
%\sum_{k=1}^{m}\text{Var}(\mathbb{M}_{ik}(b))  \leq Cn^{-1} d_n^2(b,\tau_n)\ \ \forall 1 \leq i \leq 7.  \label{eqn: lsecon2finalmle}
%\end{equation}
\noindent The proof of (\ref{eqn: lsefinalmle}) follows along the same lines of the proof of Lemma \ref{lem: variancelse}. Hence, it is omitted.  

Thus, $\mathbb{M}_n$ satisfies (\ref{eqn: lemcon1}) and (\ref{eqn: lemcon2}) and, this completes the proof of Theorem \ref{thm: mle1g}. \qed

\subsection{Proof of Theorem \ref{thm: mle2g}} \label{subsec: mle2}
The proof of (a) is the same as that of Theorem \ref{thm: lse2}(a). 

\noindent  \textbf{Proof of (b)}.
Let $b = \tau + h(n||\theta - \eta||_2^2)^{-1}$. Let $h \in K$, a compact subset of $\mathbb{R}$. Recall $\{\mathbb{M}_{ik}(b)\}$ %and $\{A_{ik}(b)\}$ respectively, 
from (\ref{eqn: m1})-(\ref{eqn: m7}). % and (\ref{eqn: aseries})-(\ref{eqn: aseriesen}). 
Note that
%\begin{eqnarray}
$L_{k,n}(b) - L_{k,n}(\tau) = \sum_{i=1}^{4} A_{ik} (b) + \mathbb{M}_{7k}(b)$ 
%\end{eqnarray}
where
\begin{align}
\label{eqn: aseries}
A_{1k}(b) &= \frac{1}{n}\sum_{t=1}^{nb} \Big(\log \mathbb{P}_{\hat{\theta}_k(b)}(X_{kt}) - \log \mathbb{P}_{\hat{\theta}_k(\tau)}(X_{kt}) \Big), \\
A_{2k}(b)&= - \frac{1}{n}\sum_{t= n\tau+1}^{n} \Big( \log \mathbb{P}_{\hat{\eta}_k(\tau)}(X_{kt}) - \log \mathbb{P}_{\hat{\eta}_k(b)}(X_{kt}) \Big), \nonumber \\
A_{3k}(b) &= \frac{1}{n}\sum_{t=nb+1}^{n\tau} \Big(  \log \mathbb{P}_{\hat{\eta}_k(b)}(X_{kt}) - \log \mathbb{P}_{{\eta}_k(\tau)}(X_{kt}) \Big),  \nonumber \\
A_{4k}(b) &= \frac{1}{n}\sum_{t=nb+1}^{n\tau} \Big(  \log \mathbb{P}_{{\theta}_k(b)}(X_{kt}) - \log \mathbb{P}_{\hat{\theta}_k(\tau)}(X_{kt}) \Big).\label{eqn: aseriesen}
\end{align}
 In the proof of Theorem \ref{thm: mle1g}, we have  already established that
\begin{eqnarray}
 nE\bigg|\sum_{k=1}^{m} A_{1k}(b)\bigg| &\leq & Cn\sum_{k=1}^{m}E|\mathbb{M}_{1k}(b)| + o(1) \leq C m |\tau-b| + o(1) \nonumber \\
 & =&  Cmn^{-1}||\theta - \eta||_2^{-2}|h| + o(1) = o(1), \nonumber\\
 nE|\sum_{k=1}^{m} A_{3k}(b)| &\leq & Cn\sum_{k=1}^{m}(E|\mathbb{M}_{3k}(b)| + E|\mathbb{M}_{4k}(b)|) + o(1) \leq Cmn^{1/2}|\tau-b| + o(1) \nonumber \\
&=& C|h|mn^{-1/2}||\theta-\eta||_2^{-2} + o(1) = o(1). \nonumber
\end{eqnarray}
Thus, $n\sum_{k=1}^{m} A_{1k}(b), n\sum_{k=1}^{m}A_{3k}(b) \stackrel{{\rm{P}}}{\to} 0$. Similarly, $n\sum_{k=1}^{m}A_{2k}(b), n\sum_{k=1}^{m}A_{4k}(b) \stackrel{{\rm{P}}}{\to} 0$. 
%\vskip 5pt
\\
\indent Let $b_i = \tau + h_i (n||\theta - \eta||_2^2)^{-1}$ for $1 \leq i \leq r$ and $h_i \in K$. Using similar argument as above, one can establish  finite dimensional convergence
%\begin{align*}
$(n\sum_{k=1}^{m} A_{jk}(b_i): 1 \leq i \leq r) \stackrel{{\rm{P}}}{\to} (0,0,\ldots,0),\ \  j=1,2,3,4$. 
%\end{align*}
Therefore, by tightness of $n\sum_{k=1}^{m} A_{jk}(b)$, we have
%\begin{align*}
$\sup_{h \in K} \bigg|n\sum_{k=1}^{m} A_{jk}(b) \bigg| \stackrel{\text{P}}{\to} 0,\ \ j=1,2,3,4$. 
%\end{align*}
%Now we consider $\mathbb{M}_{7k}(b)$. 
\\
\indent Let $-C<h<0$ for some $C>0$. Now, for some $\theta_k^{*}$ between $\theta_k$ and $\eta_k$, we have
\begin{eqnarray}
n\sum_{k=1}^{m}\mathbb{M}_{7k}(b) &=& \sum_{t=nb+1}^{n\tau}\sum_{k=1}^{m} (\log \mathbb{P}_{\theta_k}(X_{kt}) - \log \mathbb{P}_{\eta_k}(X_{kt})) \nonumber \\
&=& \sum_{t=nb+1}^{n\tau}\sum_{k=1}^{m}(\eta_k - \theta_k) \frac{\partial}{\partial \theta_k} \log \mathbb{P}_{\theta_k}(X_{kt}) + \frac{1}{2}\sum_{t=nb+1}^{n\tau}\sum_{k=1}^{m}(\eta_k - \theta_k)^2 \frac{\partial^2}{\partial \theta_k^2} \log \mathbb{P}_{\theta_k}(X_{kt}) \nonumber \\
&& \hspace{1 cm} + \frac{1}{6}\sum_{t=nb+1}^{n\tau}\sum_{k=1}^{m}(\eta_k - \theta_k)^3 \frac{\partial^3}{\partial \theta_k^3} \log \mathbb{P}_{\theta_k}(X_{kt}) \bigg|_{\theta_k = \theta_k^{*}} \nonumber \\
&=& T_1 + T_2 + T_3\ \ \text{(say)}.  \nonumber
\end{eqnarray}
By (B11),
%\begin{eqnarray}
$E|T_3| \leq C ||\theta -\eta||_2^{-2} |h| \big[\sum_{k=1}^{m}(\theta_k - \eta_k)^3\big] \big(\sup_{k,t} EG_3(X_{kt})\big) \to 0$.
%\end{eqnarray}
Thus, $T_3 \stackrel{{\rm{P}}}{\to} 0$. 
\noindent Let $B_h$ be the standard Brownian motion on the real line. Clearly, 
\begin{eqnarray}
T_2 \stackrel{{\rm{P}}}{\to} - \lim \frac{1}{2} \frac{\sum_{k=1}^{m}(\theta_k-\eta_k)^2I(\theta_k)}{||\theta - \eta||_2^2} |h|\ \ \text{and}\ \ T_1 \stackrel{\mathcal{D}}{\to} \sqrt{\lim \frac{\sum_{k=1}^{m}(\theta_k-\eta_k)^2I(\theta_k)}{||\theta - \eta||_2^2} }\ \  B_h. \label{eqn: mlegb}
\end{eqnarray}
Similarly,  $T_3 \stackrel{{\rm{P}}}{\to} 0$ and (\ref{eqn: mlegb}) also hold for $0<h<C$. This establishes one dimensional weak convergence of $n\sum_{k=1}^{m}\mathbb{M}_{7k}(b)$. Similarly one can show finite dimensional weak convergence  and tightness of $n\sum_{k=1}^{m}\mathbb{M}_{7k}(b)$. This completes the proof of (b). 

\noindent \textbf{Proof of (c)}.  In this case, we take $b = \tau + h/n$ and still analogously to (b), $n\sum_{k=1}^{m} A_{ik}(b) \stackrel{{\rm{P}}}{\to} 0$ for all $i=1,2,3,4$.  As we have seen in the proof of Theorem \ref{thm: lse2}(c), for the term $\mathbb{M}_{7k}(b)$ we consider the partition of $\{1,2,\ldots,m(n)\}$ into the sets $\mathcal{K}_n$ and $\mathcal{K}_0$ (see (\ref{eqn: partition})).  As we have discussed in Section 
\ref{sec: lse}, $\mathcal{K}_0$ is a finite set. Therefore, by (A4) and (B10), we have
\begin{eqnarray}
\sum_{k \in \mathcal{K}_0} (\mathbb{M}_{7k}(\tau + (h+1)/n) -\mathbb{M}_{7k}(\tau + h/n)) \Rightarrow \sum_{k \in \mathcal{K}_0} (\log \mathbb{P}_{\eta_k^*}(Z_{kh}) - \log \mathbb{P}_{\theta_k^*}(Z_{kh})),  \nonumber
\end{eqnarray}
where $Z_{kh} \stackrel{d}{=} X_{1k}^* I(h\leq 0) + X_{2k}^* I(h >0)$. 
Let $-C<h<0$ for some $C>0$. Then for  $\mathcal{K}_n$, 
\begin{eqnarray}
&& \sum_{k \in \mathcal{K}_n} (\mathbb{M}_{7k}(\tau + (h+1)/n) -\mathbb{M}_{7k}(\tau + h/n)) \nonumber \\
&=& \sum_{k \in \mathcal{K}_n} (\eta_k - \theta_k) \frac{\partial}{\partial \theta_k} \log \mathbb{P}_{\theta_k}(X_{kt})  +  0.5\sum_{k \in \mathcal{K}_n}(\eta_k - \theta_k)^2 \frac{\partial^2}{\partial \theta_k^2} \log \mathbb{P}_{\theta_k}(X_{kt}). \label{eqn: mlelast}
\end{eqnarray}
\noindent Thus, if $\sup_{k \in \mathcal{K}_n} |\eta_k - \theta_k| \to 0$,  then (\ref{eqn: mlelast}) converges to 
%\begin{eqnarray} \nonumber 
$-0.5\big(\lim \sum_{k \in \mathcal{K}_n} (\theta_k -\eta_k)^2I(\theta_k) \big) + \sqrt{ \big(\lim \sum_{k \in \mathcal{K}_n} (\theta_k -\eta_k)^2I(\theta_k) \big) }\ \  \mathcal{N}(0,1)$.
%\end{eqnarray}
A similar convergence also holds for $h>0$. This completes the proof of (c).
\noindent Hence, Theorem \ref{thm: mle2g} is established.  \qed

\subsection{Proof of Theorem \ref{thm: adaplse}} \label{subsec: adaplse}
In this proof we shall use $\tilde{X}_{kt}$,  $\hat{\tau}$ and $\tilde{h}$ for $\tilde{X}_{kt,\text{LSE}}$, $\hat{\tau}_{n,\text{LSE}}$ and $\tilde{h}_{n,\text{LSE}}$  respectively. First, we shall establish the convergence rate 
\begin{align} \label{eqn: adap1}
||\hat{\mu}_{1} - \hat{\mu}_{2}||_2^2 \tilde{h}  = O_P(1). 
\end{align} 
To prove (\ref{eqn: adap1}), we use the following lemma. Its proof is given in the supplement.
\begin{lemma} \label{lem: anew}
Suppose (A1)  and (A8) hold. Then for some $C>0$, $P(\sup_{i,k} \hat{\sigma}_{ik}^2 < C) \to 1$. 
\end{lemma}
By Lemma \ref{lem: anew}, to prove (\ref{eqn: adap1}), it is enough to establish that
\begin{align} \label{eqn: aadap1}
||\hat{\mu}_{1} - \hat{\mu}_{2}||_2^2 \tilde{h}  = O_{P^{*}}(1)
\end{align}
where $P^{*}(\cdot) = P(\cdot|\sup_{i,k} \hat{\sigma}_{ik}^2 < C)$. 

To prove (\ref{eqn: aadap1}),  we use Lemma \ref{lem: wvan1}.  Recall  $\tilde{M}_n(h)$ from (\ref{eqn: tildetau}).  We shall prove that (\ref{eqn: lemcon1}) and (\ref{eqn: lemcon2}) of Lemma \ref{lem: wvan1} are satisfied 
for $E(\cdot) = E^{*}(\cdot) = E(\cdot|\sup_{i,k} \hat{\sigma}_{ik}^2 < C)$, $\mathbb{M}_n = \tilde{M}_n$,  $\tilde{\mathbb{M}}_n = E^{*}(\tilde{M}_n| \{X_{kt}: k,t \geq 1\})$, $d^2_n (a,b) = n^{-1}||\hat{\mu}_1 -\hat{\mu}_2||_2^2 |a-b|$, $\phi_n(\delta) = \delta$, $\alpha = 1.5$ and $r_n = \sqrt{n}$. 

Suppose $h<0$ and $nb= n\hat{\tau}+h$. Therefore,
\begin{eqnarray}
\tilde{M}_{n}(h) - \tilde{M}_{n}(0) &=& -\frac{1}{n} \sum_{k=1}^{m} \bigg[\sum_{t=1}^{nb} (\tilde{X}_{kt} - \hat{\mu}_{1k})^2 + \sum_{t=nb+1}^{n} (\tilde{X}_{kt} - \hat{\mu}_{2k})^2\bigg] \nonumber\\
&& \hspace{1 cm} + \frac{1}{n} \sum_{k=1}^{m} \bigg[\sum_{t=1}^{n\hat{\tau}} (\tilde{X}_{kt} - \hat{\mu}_{1k})^2 + \sum_{t=n\hat{\tau}+1}^{n} (\tilde{X}_{kt} - \hat{\mu}_{2k})^2\bigg]  \nonumber \\
&=& \frac{1}{n} \sum_{k=1}^{m} \bigg[ \sum_{t=nb+1}^{n\hat{\tau}} (\tilde{X}_{kt} - \hat{\mu}_{1k})^2 - \sum_{t=nb+1}^{n\hat{\tau}} (\tilde{X}_{kt} - \hat{\mu}_{2k})^2 \bigg] \nonumber \\
&=& \frac{1}{n} \sum_{k=1}^{m} \bigg[- n(\hat{\tau}-b) (\hat{\mu}_{1k} - \hat{\mu}_{2k})^2 -2(\hat{\mu}_{1k} - \hat{\mu}_{2k})\sum_{t=nb+1}^{n\hat{\tau}} (\tilde{X}_{kt}-\hat{\mu}_{1k})\bigg] \nonumber \\
&=& A_1 + A_2,\ \ \text{(say)}.
\end{eqnarray}
It is easy to see that 
%\begin{eqnarray}
\begin{eqnarray}
\tilde{\mathbb{M}}_n(h) - \tilde{\mathbb{M}}_n(0) 
&=& E^{*}(\tilde{M}_n (h) - \tilde{M}_n(0)| \{X_{kt}: k,t \geq 1\})  \nonumber \\
&=& A_1 =  - (\hat{\tau}-b)||\hat{\mu}_1 - \hat{\mu}_2||_2^2 = -d_n^2(h,0),
\end{eqnarray}
which implies (\ref{eqn: lemcon1}) for $h<0$. 

To establish (\ref{eqn: lemcon2}), note that 
\begin{eqnarray} \nonumber
&& \sum_{t=nb+1}^{n\hat{\tau}} E^{*}((\tilde{X}_{kt}-\hat{\mu}_{1k})^2\big| \{X_{kt}: k,t \geq 1\}) \\
&=&  \sum_{t=nb+1}^{n\hat{\tau}} E((\tilde{X}_{kt}-\hat{\mu}_{1k})^2\big| \{X_{kt}: k,t \geq 1\}, \sup_{i,k}\hat{\sigma}_{ik}^2<C) \nonumber \\
&=& \sum_{t=nb+1}^{n\hat{\tau}} \hat{\sigma}_{ik}^2I(\sup_{i,k}\hat{\sigma}_{ik}^2<C) \nonumber \\
& \leq & C n(\hat{\tau}-b)
\end{eqnarray}
and for some $C_1,C_2,C_3,C_4>0$,
\begin{eqnarray}
&& E^{*} \sup_{d_n^2(h,0) \leq \delta^2} |\mathbb{M}_n(h) - \mathbb{M}_n(0) - \tilde{\mathbb{M}}_n(h) - \tilde{\mathbb{M}}_n(0) |  %\bigg| \sup_{i,k} \hat{\sigma}_{ik}^2 < C) \bigg) 
\nonumber \\
&=&E^{*}\bigg[ E^{*} \sup_{d_n^2(h,0) \leq \delta^2} \bigg(|\mathbb{M}_n(h) - \mathbb{M}_n(0) - \tilde{\mathbb{M}}_n(h) - \tilde{\mathbb{M}}_n(0) | \big| \{X_{kt}: k,t \geq 1\}\bigg)  \bigg] \nonumber \\
&= &  E^{*}\bigg[ E^{*} \sup_{d_n^2(h,0) \leq \delta^2} \bigg( |A_1 + A_2 - A_1 | \big| \{X_{kt}: k,t \geq 1\}) \bigg) \bigg] \nonumber \\
& \leq & C_1 E^{*} \sup_{d_n^2(h,0) = \delta^2} E^{*}(|A_2 | \big| \{X_{kt}: k,t \geq 1\})\nonumber \\
&\leq & C_1 E^{*} \sup_{d_n^2(h,0) = \delta^2} \left(E^{*}(|A_2|^2 \big| \{X_{kt}: k,t \geq 1\}) \right)^{1/2} \nonumber \\
&\leq & C_2  E^{*}\bigg[ \sup_{d_n^2(h,0) = \delta^2} \bigg(\frac{1}{n^2} \sum_{k=1}^{m} \bigg[(\hat{\mu}_{1k} - \hat{\mu}_{2k})^2 \sum_{t=nb+1}^{n\hat{\tau}} E^{*}((\tilde{X}_{kt}-\hat{\mu}_{1k})^2\big| \{X_{kt}: k,t \geq 1\}) \bigg)^{1/2} \bigg] \nonumber \\
& \leq & C_3 n^{-1/2} E^{*} \sup_{d_n^2(h,0) = \delta^2} \bigg[(\hat{\tau}-b)||\hat{\mu}_1 - \hat{\mu}_2||_2^2 \bigg]^{1/2} \leq C_4  n^{-1/2}\delta.
\end{eqnarray}
This proves (\ref{eqn: lemcon2}) for $h<0$. 
\noindent Similar argument proves (\ref{eqn: lemcon1}) and (\ref{eqn: lemcon2}) for $h>0$. 

\noindent Thus (\ref{eqn: aadap1}) and consequently (\ref{eqn: adap1}) are proved.

\noindent Note that (\ref{eqn: adap1}) implies Theorem \ref{thm: adaplse}(a). 

To prove Theorem \ref{thm: adaplse}(b), we use %Lemma \ref{lem: wvandis1} and 
the following lemma.  Its proof is given in the supplement.

\begin{lemma} \label{lem: adap}
Suppose (A1), (A8) and (SNR3) hold and  $||\mu_1 - \mu_2||_2 \to 0$. Then  for some $C_1,C_2>0$, the following statements hold.
\vskip 3pt
\noindent (a) $P( \sup_{k,t} E((\tilde{X}_{kt} - E(\tilde{X}_{kt}|\{X_{kt}\}))^4|\{X_{kt}\}) < C_1) \to 1$, $P(\inf_{k,i} \hat{\sigma}_{ik}^2 > C_2) \to 1$. 
%$\sup_{k,i} |\hat{\sigma}_{ik}^2 - \sigma_{ik}^2| \stackrel{P}{\to} 0$.
\vskip 3pt
\noindent (b)  $||\hat{\mu}_1 - \hat{\mu}_2||_2^{-2} \sum_{k=1}^{m} (\hat{\mu}_{1k} - \hat{\mu}_{2k})^2  \hat{\sigma}_{1k}^2 \stackrel{P}{\to} \gamma_{_{L,LSE}}^2$.
\vskip 3pt
\noindent (c)  $||\hat{\mu}_1 - \hat{\mu}_2||_2^{-2} \sum_{k=1}^{m} (\hat{\mu}_{1k} - \hat{\mu}_{2k})^2  \hat{\sigma}_{2k}^2 \stackrel{P}{\to} \gamma_{_{R,LSE}}^2$.
\end{lemma} 
\noindent Now, let $P^{**}(\cdot) = P(\cdot|\sup_{k,t} E((\tilde{X}_{kt} - E(\tilde{X}_{kt}|\{X_{kt}\}))^4|\{X_{kt}\}) < C_1, \inf_{k,i} \hat{\sigma}_{ik}^2 > C_2)$. Similarly define $E^{**}$ and $\text{Var}^{**}$.     By Lemma \ref{lem: adap},   it is easy to see that the convergences in Lemma \ref{lem: adap}(b) and (c)  hold for $P^{**}$ also. When $\sup_{k,t} E(\tilde{X}_{kt} - E(\tilde{X}_{kt}|\{X_{kt}\})|\{X_{kt}\})^4 < C_1$, $||\hat{\mu}_1 - \hat{\mu}_2||_2^{-2} \sum_{k=1}^{m} (\hat{\mu}_{1k} - \hat{\mu}_{2k})^2  \hat{\sigma}_{ik}^2$ is bounded for each $i=1,2$. Therefore, the convergences in Lemma \ref{lem: adap}(b) and (c)  also hold in $E^{**}$.

\noindent To prove Theorem \ref{thm: adaplse}(b), by Lemma \ref{lem: adap}(a), it is enough to establish that for all $x \in \mathbb{R}$,
\begin{eqnarray} \label{eqn: adapbb}
&& P^{**}(n||\hat{\mu}_1-\hat{\mu}_2||_2^2 \tilde{h} \leq x) \\
&\to & P(\arg \max_{h \in \mathbb{R}} (-0.5|h| + \gamma_{_{\text{L,LSE}}}B_hI(h\leq 0) + \gamma_{_{\text{R,LSE}}}B_h I(h>0)) \leq x). \nonumber
\end{eqnarray}

\noindent Let $b = \hat{\tau}+ h/n||\hat{\mu}_1-\hat{\mu}_2||_2^2$.  Note that,
\begin{eqnarray}
 n(\tilde{M}_{n}(||\hat{\mu}_1-\hat{\mu}_2||_2^{-2}h) - \tilde{M}_{n}(0))  
&=& \begin{cases} -|h| -2 \sum_{t=nb+1}^{n\hat{\tau}} \sum_{k=1}^{m} (\hat{\mu}_{1k} - \hat{\mu}_{2k}) (\tilde{X}_{kt}-\hat{\mu}_{1k}),\ \ \ \text{if $h<0$} \\ 
-|h| -2 \sum_{t=n\hat{\tau}+1}^{nb} \sum_{k=1}^{m} (\hat{\mu}_{1k} - \hat{\mu}_{2k}) (\tilde{X}_{kt}-\hat{\mu}_{2k}),\ \ \ \text{if $h>0$}. \nonumber
\end{cases}
\end{eqnarray}
First consider the case $h<0$. 
Note that $\{\sum_{k=1}^{m} (\hat{\mu}_{1k} - \hat{\mu}_{2k})(\tilde{X}_{kt} - \hat{\mu}_{1k})\}$ is a collection of independent random variables.  By Lemma \ref{lem: adap}(b), we have 
\begin{align*}
\hspace{-2 cm}E^{**}\sum_{t=nb+1}^{n\hat{\tau}} \sum_{k=1}^{m} (\hat{\mu}_{1k} - \hat{\mu}_{2k}) (\tilde{X}_{kt}-\hat{\mu}_{1k}) = E^{**}\bigg[\sum_{t=nb+1}^{n\hat{\tau}} \sum_{k=1}^{m} (\hat{\mu}_{1k} - \hat{\mu}_{2k}) E^{**}((\tilde{X}_{kt}-\hat{\mu}_{1k})|\{X_{kt}\})\bigg] = 0, \\
 \text{Var}^{**}\bigg(\sum_{t=nb+1}^{n\hat{\tau}} \sum_{k=1}^{m} (\hat{\mu}_{1k} - \hat{\mu}_{2k}) (\tilde{X}_{kt}-\hat{\mu}_{1k})\bigg) 
= hE^{**}\bigg(||\hat{\mu}_1 - \hat{\mu}_2||_2^{-2} \sum_{k=1}^{m} (\hat{\mu}_{1k} - \hat{\mu}_{2k})^2  \hat{\sigma}_{1k}^2 \bigg) \to h \gamma_{_\text{L,LSE}}^2, 
\end{align*}
%Moreover, 
\begin{eqnarray}
&& \frac{E^{**}\bigg[\sum_{t=nb+1}^{n\hat{\tau}} \sum_{k=1}^{m} (\hat{\mu}_{1k} - \hat{\mu}_{2k}) (\tilde{X}_{kt}-\hat{\mu}_{1k})\bigg]^3}{\bigg[\text{Var}^{**}\bigg(\sum_{t=nb+1}^{n\hat{\tau}} \sum_{k=1}^{m} (\hat{\mu}_{1k} - \hat{\mu}_{2k}) (\tilde{X}_{kt}-\hat{\mu}_{1k})\bigg)\bigg]^{3/2}} %\nonumber \\
 \leq  C E\left(\frac{\sum_{k=1}^{m}|\hat{\mu}_{1k} - \hat{\mu}_{2k}|^3}{\sum_{k=1}^{m}|\hat{\mu}_{1k} - \hat{\mu}_{2k}|^2}\right)  \nonumber \\
&\leq  & C (E||\hat{\mu}_{1} - \hat{\mu}_2||_2^2)^{1/2} %\nonumber \\
\leq  C\left( E(||\hat{\mu}_{1} - {\mu}_1||_2^2) + E(||\hat{\mu}_{2} - {\mu}_2||_2^2)+ ||{\mu}_{1} - {\mu}_1||_2^2 \right) \nonumber \\
& \leq & C(mn^{-1} + ||{\mu}_{1} - {\mu}_1||_2^2) \to 0. 
%C E\sup_{k} |\hat{\mu}_{1k} - \hat{\mu}_{2k}| 
\end{eqnarray}
Hence by Lyapunov's central limit theorem, under (A1), (A2), (A3), (A8), (SNR3), 
\begin{eqnarray}
n(\tilde{M}_n(||\hat{\mu}_1-\hat{\mu}_2||_2^{-2} h) - \tilde{M}(0)) \Rightarrow -|h| + \gamma_{_\text{L,LSE}} B_h \label{eqn: adaplseb1}
\end{eqnarray}
for $-C<h<0$ and where  $B_h$ is the standard Brownian motion. \\
\indent Similarly, when $0<h<C$, by (A1), (A2), (A3), (A8), (SNR3) and Lemma \ref{lem: adap}(b), 
\begin{eqnarray}
n(\tilde{M}_n(||\hat{\mu}_1-\hat{\mu}_2||_2^{-2}h) - \tilde{M}(0)) \Rightarrow -|h| + \gamma_{_\text{R,LSE}} B_h. \label{eqn: adaplseb2}
\end{eqnarray}
(\ref{eqn: adaplseb1}) and (\ref{eqn: adaplseb2}) in conjunction with Lemma \ref{lem: wvandis1} establish Theorem \ref{thm: adaplse}(b). 

\noindent A similar argument as in the proof of Theorem \ref{thm: lse2}(c) and similar approximations as in the proof of Theorem \ref{thm: adaplse}(b)  also work for Theorem \ref{thm: adaplse}(c) and hence we omit them.
\noindent Hence, Theorem \ref{thm: adaplse} is established.  \qed

%\bibliographystyle{abbrvnat}
%\bibliography{refbib}

%\newpage
  
%\noindent \textbf{Acknowledgments}.

\section{Supplementary Material} 
\label{sec:supp}
\subsection{Supplement to Section 2.1} 
 The following examples provide some situations where the limits in $\gamma_{_\text{L,LSE}}$, $\gamma_{_\text{R,LSE}}$, $c_1$,  $\gamma_{_\text{L,LSE}}^{*}$ and $\gamma_{_\text{R,LSE}}^{*}$ do not exist. 

 Example \ref{counter1} shows that the limits can not exists if $\sigma_{ik}^2 (n)$ oscillates over $n$. 
\begin{Example}
\label{counter1} 
Suppose $\sigma_{ik}^2 (n) = 2 + (-1)^n$ for all $k \geq 1$ and $i=1,2$. Further suppose $c_1^2 = \lim \sum_{k \in \mathcal{K}_n} (\mu_{1k} - \mu_{2k})^2$ exists. Then
\begin{align*}
& \frac{\sum_{k=1}^{m} (\mu_{1k} - \mu_{2k})^2 \sigma_{1k}^2}{||\mu_1-\mu_2||_2^2} = \frac{\sum_{k=1}^{m} (\mu_{1k} - \mu_{2k})^2 \sigma_{2k}^2 }{||\mu_1-\mu_2||_2^{2}} 
= 2 + (-1)^n \ \text{and}\ \\
&\sum_{k \in \mathcal{K}_n} (\mu_{1k} - \mu_{2k})^2 \sigma_{1k}^2 = \sum_{k \in \mathcal{K}_n} (\mu_{1k} - \mu_{2k})^2 \sigma_{2k}^2 = (2 + (-1)^n) \sum_{k \in \mathcal{K}_n} (\mu_{1k} - \mu_{2k})^2,
\end{align*}
which do not have a limit as $n \to \infty$. 
\end{Example}

\noindent Moreover, by (A1) and (A3), $\{\sigma_{ik}^2(n)\}$ is a bounded sequence for each $k$ and $i$. Thus for each $k \geq 1$ and $i=1,2$, $\{\sigma_{ik}^2(n)\}$ needs to converge to a limit $\sigma_{ik}^{*2}$ (say). This leads to \vskip 1pt
(a) $\sup_{k \in \mathcal{K}_n}|\sigma_{ik}^2(n) - \sigma_{ik}^{*2}| \to 0$ for all $i=1,2$ and for some $\sigma_{ik}^{*} >0$, 
\vskip 1pt
\noindent in Propositions 2.3 and 2.4. %This is guaranteed by the condition (a) in the following propositions.
\vskip 5pt

In the next two examples, we  deal with Regime (c): $||\mu_1 - \mu_2|| \to c^2>0$ and the existence of the limits in $c_1$, $\gamma_{_{\text{L,LSE}}}^{*}$ and $\gamma_{_{\text{R,LSE}}}^{*}$. Consider the following conditions which are considered in Proposition 2.4. For some $\mu_{ik}^{*} \in \mathbb{R}$, %and $\sigma_{ik}^{*}>0$, 
which are free of $n$, 
%Sufficient conditions for the existence of the above limits are (a)-(d) as described below.
\vskip 3pt
\noindent (f) $\sup_{k \in \mathcal{K}_n} |\mu_{ik}(n) - \mu_{ik}^{*}| = O(m(n)^{-1/2})$ for all $i=1,2$ and
\vskip 3pt
\noindent (g) $\sum_{k \in \mathcal{K}_n} | \mu_{1k}(n) - \mu_{2k}(n)| = o(\sqrt{m(n)})$.
\vskip 1pt
\noindent Example \ref{counter2} provides a situation where (a) and (f) are satisfied but (g) does not hold %. Also in Example \ref{counter3},  (a) and (g) hold but (f) is not satisfied. We 
and shows that the limits in $c_1$, $\gamma_{_{\text{L,LSE}}}^{*}$ and $\gamma_{_{\text{R,LSE}}}^{*}$ do not exists. 
%The following example illustrates that $\{\sqrt{m(n)}(\mu_{1k}(n) - \mu_{2k}(n))\}$ in the third regime and $\{\sqrt{m(n)}||\mu_1-\mu_2||_2^{-1}(\mu_{1k}(n) - \mu_{2k}(n))\}$ in the second regime,  cannot fluctuate over $n$. This leads us to the conditions (b) and (f) in the following propositions. Also observe that (f) implies (A7).

\begin{Example}
\label{counter2} 
Suppose  $\sigma_{ik}^2 (n) = \sigma^2$ for all $k,n \geq 1$ and $i=1,2$. Let for all $k \in \mathcal{K}_n$ and $i$,
\begin{eqnarray} \label{eqn: exosc2}
%\sqrt{m(n)}(\mu_{1k}(n) - \mu_{2k}(n))
\mu_{ik}(n) = \begin{cases}
\frac{1}{\sqrt{m(n)}}(2i + \frac{1}{k}),\ \ \text{if $n$ is odd} \\
\frac{1}{\sqrt{m(n)}}(3i + \frac{1}{k}),\ \ \text{if $n$ is even}.
\end{cases}
\end{eqnarray}
Note that this example satisfies (a) and (f) but not (g). \\
Now, 
\begin{eqnarray}
\sum_{k \in \mathcal{K}_n} (\mu_{1k}(n) - \mu_{2k}(n))^2 \sigma_{ik}^2 (n) = \begin{cases}
2\sigma^2 \left(\frac{\text{cardinality of $\mathcal{K}_n$}}{m(n)} \right),\ \ \text{if $n$ is odd} \\
3\sigma^2 \left(\frac{\text{cardinality of $\mathcal{K}_n$}}{m(n)} \right),\ \ \text{if $n$ is even}.
\end{cases}
\end{eqnarray}
It is easy to see that $\left(\frac{\text{cardinality of $\mathcal{K}_n$}}{m(n)} \right) \to 1$. 
Hence  $c_1^2$, $\gamma_{_\text{L,LSE}}^{*2}$ and $\gamma_{_\text{R,LSE}}^{*2}$ do not exist. 
\end{Example}

Similar phenomenon happens for regime (b): $||\mu_1 - \mu_2||_2 \to 0$ once we replace $m(n)$ in (f), (g) and Example \ref{counter2} by $m(n)||\mu_1 - \mu_2||_2^{-2}$ and then $\gamma_{_\text{L,LSE}}^2$ and $\gamma_{_\text{R,LSE}}^2$ do not exist. 

\subsubsection{Proof of Proposition 2.3} \label{subsec: prop1}
We have  that $||\mu_1 (n) - \mu_2 (n)||_2^2 \to 0$ and $nm^{-1}||\mu_1(n) - \mu_2(n)||_2^2 \to \infty$.
Recall
\begin{align*}
\gamma_{_\text{L,LSE}}^{2} &= \lim \frac{\sum_{k =1}^{m(n)} (\mu_{1k}(n) - \mu_{2k}(n))^2 \sigma_{1k}^2(n)}{||\mu_1(n) - \mu_2(n)||_2^2},\ \ 
\gamma_{_\text{R,LSE}}^{2}  = \lim \frac{\sum_{k =1}^{m(n)} (\mu_{1k}(n) - \mu_{2k}(n))^2 \sigma_{2k}^2(n)}{||\mu_1(n) - \mu_2(n)||_2^2}.
\end{align*}
 \noindent By (a),  it is easy to see that
 \begin{align*}
\gamma_{_\text{L,LSE}}^{2} &= \lim \frac{\sum_{k =1}^{m(n)} (\mu_{1k}(n) - \mu_{2k}(n))^2 \sigma_{1k}^2}{||\mu_1 (n) - \mu_2 (n)||_2^2},\ \ 
\gamma_{_\text{R,LSE}}^{2}  = \lim \frac{\sum_{k =1}^{m(n)} (\mu_{1k}(n) - \mu_{2k}(n))^2 \sigma_{2k}^2}{||\mu_1 (n) - \mu_2 (n)||_2^2}.
 \end{align*}
 \noindent Moreover,  by (A1) and (A3),   $||\mu_1 (n) - \mu_2 (n)||_2^{-2}\sum_{k =1}^{m(n)} (\mu_{1k}(n) - \mu_{2k}(n))^2 \sigma_{1k}^2$ and $||\mu_1 (n) - \mu_2 (n)||_2^{-2}\sum_{k=1}^{m(n)} (\mu_{1k}(n) - \mu_{2k}(n))^2 \sigma_{2k}^2$ are all bounded quantities.  Therefore, their limit exists if they are Cauchy sequences. We shall show that $\{||\mu_1 (n) - \mu_2 (n)||_2^{-2}\sum_{k =1}^{m(n)} (\mu_{1k}(n) - \mu_{2k}(n))^2 \sigma_{1k}^2\}$ is indeed a
Cauchy sequence.  \\ %Similarly one can see for other sequences. 
\indent Note that all the inequalities below hold for large enough $n$. The difference between two consecutive terms in $\{||\mu_1 (n) - \mu_2 (n)||_2^{-2}\sum_{k =1}^{m(n)} (\mu_{1k}(n) - \mu_{2k}(n))^2 \sigma_{1k}^2\}$ is given by 
\begin{align*}
& \bigg | \frac{\sum_{k =1}^{m(n+1)} (\mu_{1k}(n+1) - \mu_{2k}(n+1))^2 \sigma_{1k}^2}{||\mu_1(n+1) - \mu_2(n+1)||_2^2} - \frac{\sum_{k =1}^{m(n)} (\mu_{1k}(n) - \mu_{2k}(n))^2 \sigma_{1k}^2}{||\mu_1(n) - \mu_2(n)||_2^2} \bigg| \\
\leq\ & C||\mu_1(n+1) - \mu_2 (n+1)||_2^{-2} \bigg|\sum_{k =1}^{m(n+1)} (\mu_{1k}(n+1) - \mu_{2k}(n+1))^2 \sigma_{1k}^2 \\
& \hspace{8 cm}-\sum_{k =1}^{m(n)} (\mu_{1k}(n) - \mu_{2k}(n))^2 \sigma_{1k}^2 \bigg|\\
=\  & C||\mu_1(n+1) - \mu_2 (n+1)||_2^{-2} \bigg | \sum_{k =1}^{m(n)} (\mu_{1k}(n+1) - \mu_{2k}(n+1))^2 \sigma_{1k}^2 \\
& \hspace{8 cm}- \sum_{k =1}^{m(n)} (\mu_{1k}(n) - \mu_{2k}(n))^2 \sigma_{1k}^2 \bigg|  \\
& \hspace{1.5 cm} + C||\mu_1(n+1) - \mu_2 (n+1)||_2^{-2}\bigg|\sum_{k = m(n)+1}^{m(n+1)} (\mu_{1k}(n+1) - \mu_{2k}(n+1))^2 \sigma_{1k}^2 \bigg|  \\
 \leq \  & C ||\mu_1(n+1) - \mu_2 (n+1)||_2^{-2}\sum_{k =1}^{m(n)} \bigg[ |\mu_{1k}(n) - \mu_{2k}(n) + \mu_{1k}(n+1) -\mu_{2k}(n+1)|\\
 &\hspace{6 cm}|\mu_{1k}(n) - \mu_{2k}(n) - \mu_{1k}(n+1) + \mu_{2k}(n+1)| \bigg] \\
 & \hspace{0.5 cm}+ C \bigg|\frac{m(n+1)-m(n)}{m(n+1)}\bigg| \frac{m(n+1)}{ ||\mu_1(n+1) - \mu_2 (n+1)||_2^{2}}\sup_{k} |\mu_{1k}(n+1) - \mu_{2k}(n+1)|^2 \\
 \leq\ & C\sup_{k} (|\mu_{1k}(n+1) -\mu_{1k}(n)| + |\mu_{2k}(n+1) -\mu_{2k}(n)|)\frac{\sum_{k=1}^{m(n)}|\mu_{1k}(n) - \mu_{2k}(n)|}{ ||\mu_1(n+1) - \mu_2 (n+1)||_2^{2}} \\
 & + C\sup_{k} (|\mu_{1k}(n+1) -\mu_{1k}(n)| + |\mu_{2k}(n+1) -\mu_{2k}(n)|)\frac{\sum_{k =1}^{m(n)}|\mu_{1k}(n) - \mu_{2k}(n)|}{ ||\mu_1(n+1) - \mu_2 (n+1)||_2^{2}} \\
 &  \hspace{0.5 cm}+ C \bigg|\frac{m(n+1)-m(n)}{m(n+1)}\bigg| \frac{m(n+1)}{ ||\mu_1(n+1) - \mu_2 (n+1)||_2^{2}}\sup_{k} |\mu_{1k}(n+1) - \mu_{2k}(n+1)|^2 \\
 =\ & o(1).
\end{align*} 
This shows $\{||\mu_1 (n) - \mu_2 (n)||_2^{-2}\sum_{k =1}^{m(n)} (\mu_{1k}(n) - \mu_{2k}(n))^2 \sigma_{1k}^2\}$ is Cauchy. Similar arguments
establish the result for the other sequences.  \qed
 
\subsubsection{Proof of Proposition 2.4} \label{subsec: prop2}
We have  that $\sum_{k=1}^{m(n)} (\mu_{1k}(n) - \mu_{2k}(n))^2 \to c^2>0$ and $m(n) = o(n)$.
Recall
\begin{align*}
c_1^2 &= \lim \sum_{k \in \mathcal{K}_n} (\mu_{1k}(n) - \mu_{2k}(n))^2,\ \ 
\gamma_{_\text{L,LSE}}^{*2} = \lim \sum_{k \in \mathcal{K}_n} (\mu_{1k}(n) - \mu_{2k}(n))^2 \sigma_{1k}^2(n), \\
\gamma_{_\text{R,LSE}}^{*2}  &= \lim \sum_{k \in \mathcal{K}_n} (\mu_{1k}(n) - \mu_{2k}(n))^2 \sigma_{2k}^2(n).
\end{align*}
 \noindent By (a), it is easy to see that
 \begin{align*}
\gamma_{_\text{L,LSE}}^{*2} &= \lim \sum_{k \in \mathcal{K}_n} (\mu_{1k}(n) - \mu_{2k}(n))^2 \sigma_{1k}^2,\ \ 
\gamma_{_\text{R,LSE}}^{*2}  = \lim \sum_{k \in \mathcal{K}_n} (\mu_{1k}(n) - \mu_{2k}(n))^2 \sigma_{2k}^2.
 \end{align*}
 \noindent Moreover,  by (A1) and (A3),  $\sum_{k \in \mathcal{K}_n} (\mu_{1k}(n) - \mu_{2k}(n))^2$, $\sum_{k \in \mathcal{K}_n} (\mu_{1k}(n) - \mu_{2k}(n))^2 \sigma_{1k}^2$ and $\sum_{k \in \mathcal{K}_n} (\mu_{1k}(n) - \mu_{2k}(n))^2 \sigma_{2k}^2$ are all bounded quantities.  Therefore, their limit exist if they are Cauchy sequences. \\ %We shall show that $\{\sum_{k \in \mathcal{K}_n^{c}} (\mu_{1k}(n) - \mu_{2k}(n))^2 \sigma_{1k}^2\}$ is Cauchy. Similarly one can see for other sequences. 
\indent The difference between two consecutive terms in $\{\sum_{k \in \mathcal{K}_n} (\mu_{1k}(n) - \mu_{2k}(n))^2 \sigma_{1k}^2\}$ is given by
\begin{align*}
& \bigg | \sum_{k \in \mathcal{K}_{n+1}} (\mu_{1k}(n+1) - \mu_{2k}(n+1))^2 \sigma_{1k}^2 - \sum_{k \in \mathcal{K}_n} (\mu_{1k}(n) - \mu_{2k}(n))^2 \sigma_{1k}^2 \bigg| \\
=\  & \bigg | \sum_{k \in \mathcal{K}_{n+1} \cap \mathcal{K}_n} (\mu_{1k}(n+1) - \mu_{2k}(n+1))^2 \sigma_{1k}^2 - \sum_{k \in \mathcal{K}_{n+1} \cap \mathcal{K}_n} (\mu_{1k}(n) - \mu_{2k}(n))^2 \sigma_{1k}^2 \bigg|  \\
& \hspace{1.5 cm} + \bigg|\sum_{k \in \mathcal{K}_{n+1} - \mathcal{K}_n} (\mu_{1k}(n+1) - \mu_{2k}(n+1))^2 \sigma_{1k}^2 \bigg| 
 + \bigg|\sum_{k \in \mathcal{K}_{n} - \mathcal{K}_{n+1}} (\mu_{1k}(n) - \mu_{2k}(n))^2 \sigma_{1k}^2 \bigg| \\
 \leq \  & C \sum_{k \in \mathcal{K}_{n+1} \cap \mathcal{K}_n} \bigg[ |\mu_{1k}(n) - \mu_{2k}(n) + \mu_{1k}(n+1) -\mu_{2k}(n+1)|\\
 &\hspace{4 cm}|\mu_{1k}(n) - \mu_{2k}(n) - \mu_{1k}(n+1) + \mu_{2k}(n+1)| \bigg] \\
 & \hspace{2 cm}+ C \bigg|\frac{m(n+1)-m(n)}{m(n+1)}\bigg| m(n+1)\sup_{k \in \mathcal{K}_{n+1}} |\mu_{1k}(n+1) - \mu_{2k}(n+1)|^2 \\
 & \hspace{2 cm}+ C \bigg|\frac{m(n+1)-m(n)}{m(n)}\bigg| m(n)\sup_{k \in \mathcal{K}_{n}} |\mu_{1k}(n) - \mu_{2k}(n)|^2 \\
 \leq\ & \sup_{k \in \mathcal{K}_n \cap \mathcal{K}_{n+1}} (|\mu_{1k}(n+1) -\mu_{1k}(n)| + |\mu_{2k}(n+1) -\mu_{2k}(n)|)\sum_{k \in\mathcal{K}_n}|\mu_{1k}(n) - \mu_{2k}(n)| \\
 & + \sup_{k \in \mathcal{K}_n \cap \mathcal{K}_{n+1}} (|\mu_{1k}(n+1) -\mu_{1k}(n)| + |\mu_{2k}(n+1) -\mu_{2k}(n)|)\sum_{k \in\mathcal{K}_{n+1}}|\mu_{1k}(n) - \mu_{2k}(n)| \\
 & \hspace{2 cm}+ C \bigg|\frac{m(n+1)-m(n)}{m(n+1)}\bigg| m(n+1)\sup_{k \in \mathcal{K}_{n+1}} |\mu_{1k}(n+1) - \mu_{2k}(n+1)|^2 \\
 & \hspace{2 cm}+ C \bigg|\frac{m(n+1)-m(n)}{m(n)}\bigg| m(n)\sup_{k \in \mathcal{K}_{n}} |\mu_{1k}(n) - \mu_{2k}(n)|^2 \\
 =\ & o(1).
\end{align*} 
This establishes that $\{\sum_{k \in \mathcal{K}_n} (\mu_{1k}(n) - \mu_{2k}(n))^2 \sigma_{1k}^2\}$ is Cauchy. Similarly one can establish the result for the other sequences.  \qed

\subsection{Supplement to Section 2.3} \label{sec:supp2}
The following  two examples discuss asymptotic distribution of $\hat{\tau}_{n,\text{LSE}}$ for two $m$-dependent white noise processes. Suppose (SNR1) holds for these examples. 
\begin{Example} \textbf{IID process}. \label{example: dep1}
 Suppose  $\{\varepsilon_{kt}\}$ are independent and identically distributed over $t$ and independent over $k$ with mean $0$, variance $\sigma_{k\varepsilon}^2$, $\sup_{k} E\varepsilon_{kt}^r < \infty$ for all $r \geq 1$ and $\inf_{k} \sigma_{k\varepsilon} >C$ for some $C>0$. Note that $\{\varepsilon_{kt}\}$ is a $0$-dependent white noise process. Suppose for each $k, t\geq 1$,
\begin{eqnarray} \label{eqn: iidwh}
 X_{kt} = \mu_{1k}I(t \leq n\tau_n) + \mu_{2k}I(t>n\tau_n) + \varepsilon_{kt}.
 \end{eqnarray}
Then (D1), (D2), (D3)  are satisfied and hence the conclusions of Theorem 2.5 % \ref{thm: lse3} 
and 2.6(a) %\ref{thm: lse4}(a) 
hold for (\ref{eqn: iidwh}).  Moreover (D4) holds if (b)-(e) in Proposition 2.3 % \ref{prop: 1} 
are satisfied. Therefore, conclusion of Theorem 2.6(b) %\ref{thm: lse4}(b) 
holds for (\ref{eqn: iidwh}) under (b)-(e) in Proposition 2.3.  % \ref{prop: 1}.  
Also (D6) holds if (e)-(g) in Proposition 2.4 hold. % \ref{prop: 2} hold.  
Thus, under (A4), (D5) and (e)-(g) in Proposition 2.4, % \ref{prop: 2}, 
the conclusion of Theorem 2.6(c) % \ref{thm: lse4}(c) 
hold for (\ref{eqn: iidwh}).
\end{Example}

\begin{Example} \textbf{Uncorrelated non-linear moving average process}. \label{example: dep2}
Suppose $\{\varepsilon_{kt}\}$ is as in Example \ref{example: dep1}. Suppose for all $k,t \geq 1$,
\begin{eqnarray} \label{eqn: nonma}
 X_{kt} &=& \mu_{1k}I(t \leq n\tau_n) + \mu_{2k}I(t>n\tau_n) + Y_{kt},\ \ \text{where} \\
 Y_{kt} &=& \varepsilon_{k(t-1)}\varepsilon_{k(t-2)} + \varepsilon_{kt}. \nonumber 
 \end{eqnarray}
Note that for each $k \geq 1$,  $\{Y_{kt}: t \in \mathbb{Z}\}$, in this example, is a $3$-dependent white noise process. Here (D1), (D2) and (D3) are satisfied. % and hence the for (\ref{eqn: nonma}).  
Note that $\text{Var}(X_{kt}) = \sigma_{k\varepsilon}^4$. Therefore, (D4) holds if (b)-(e) in Proposition 2.3 % \ref{prop: 1}  
are satisfied. Moreover  (e)-(g) in Proposition 2.4 %\ref{prop: 2} 
imply (A7) and (D6).   %Suppose (SNR1) holds. 
Then  the results given in  Example \ref{example: dep1} also hold for (\ref{eqn: nonma}).
\end{Example}
As discussed after Example $2.4$, $\{Y_{kt}: t \in \mathbb{Z}\}$ in (D1) may not be always  an $m$-dependent or a Gaussian process.  Then (D3) may not be satisfied always even if $\{Y_{kt}: t \in \mathbb{Z}\}$ is moment stationary process.  The following remark provides a wide class of processes for which (D3) holds. 

 Often dependence in time series is captured by its mixing properties. There are several notions of mixing in the literature such as $\alpha$, $\beta$, $\phi$ and $\rho$-mixing, with $\alpha$-mixing is strongest among them. 

\noindent {\bf Definition:} A  process $\{Y_t: t \in \mathbb{Z}\}$ is called an $\alpha$-mixing process if as $n \to \infty$, 
\begin{eqnarray}
\alpha(n) &=& \sup_{m \in \mathbb{N}} \sup_{\stackrel{A \in \mathcal{A}_1^m}{B \in \mathcal{A}_{m+n}^{\infty}}} | P(A \cap B) - P(A)P(B)|\to 0\ \ \ \text{where} \label{eqn: amix} \\
\mathcal{A}_1^m &=& \text{$\sigma$-algebra generated by $Y_1,Y_2,\ldots,Y_m$}, \\
\mathcal{A}_{m+n}^{\infty} &=& \text{$\sigma$-algebra generated by $Y_{m+n},Y_{m+n+1},\ldots$}.
\end{eqnarray}
$\alpha(n)$ in (\ref{eqn: amix}) is called mixing coefficient of $\{Y_t\}$. 
\vskip 5pt
\noindent The following lemma from \cite{bbs2017mixing} provides sufficient condition on a process $\{Y_{t}: t \in \mathbb{Z}\}$  so that its cumulants are summable. 
\begin{lemma} \label{lem: bbs}
Suppose $\{Y_t: t \in \mathbb{Z}\}$ is an uniformly bounded  centered moment stationary $\alpha$-mixing process with mixing coefficient $\alpha(n) = \alpha^n$ for some $0<\alpha<1$. Then for all $t, t_1, t_2,\ldots, t_r \in \mathbb{Z}$ and $r \geq 1$, 
$$ |\text{Cum}(Y_{t}, Y_{t_1+t},Y_{t_2+t},\ldots, Y_{t_r+t})| < C\alpha^{(|t_1|+|t_2|+\cdots + |t_r|)/r}.$$
\end{lemma}
\noindent The following remark is immediate from Lemma \ref{lem: bbs}.
\begin{remark} \label{rem: mixing}
Consider $\{X_{kt}\}$ as in (D1).  Suppose $\{Y_{kt}: t \in \mathbb{Z}, k \geq 1\}$ are uniformly bounded and for each $k$, $\{Y_{kt}: t \in \mathbb{Z}\}$ is  centered moment stationary  $\alpha$-mixing process with mixing-coefficient $\alpha_k (n) = \alpha_k^n$  and there is $c>0$ such that $0<\sup_{k} \alpha_k < 1-c$. Then $\{X_{kt}\}$ satisfies Assumption (D3).
\end{remark}

\subsubsection{Proof of Theorem 2.5} \label{subsec: lse3}
This proof will be complete using similar techniques as in Section 5.1 once we establish Lemmas 5.2 and 5.3 % \ref{lem: expectation} and \ref{lem: variancelse} 
under (D1), (D2) and (SNR1). 

\noindent Note that there is $C>0$ (independent of $k$) such that
\begin{align*}
& E(N_{1k}^2 (b)) = E \left(\frac{1}{nb}\sum_{t=1}^{nb}(X_{kt} - E(X_{kt})) \right)^2 = \frac{1}{(nb)^2} \sum_{t,t^\prime=1}^{nb} \mbox{Cum}(X_{kt},X_{kt^\prime}) \leq \frac{C}{n}, \\
& E(N_{2k}^2 (b)) = E \left(\frac{1}{n(1-b)}\sum_{t=nb+1}^{n}(X_{kt} - E(X_{kt})) \right)^2\\
& \hspace{6 cm}= \frac{1}{(n(1-b))^2} \sum_{t,t^\prime=nb+1}^{n} \mbox{Cum}(X_{kt},X_{kt^\prime}) 
 \hspace{0 cm}\leq \frac{C}{n}, \\
& E(N_{3k}^2 (b,\tau)) = E \left(\frac{1}{n(\tau-b)}\sum_{t=nb+1}^{n\tau}(X_{kt} - E(X_{kt})) \right)^2 \\
&\hspace{4 cm }= \frac{1}{(n(\tau-b))^2} \sum_{t,t^\prime=nb+1}^{n\tau} \mbox{Cum}(X_{kt},X_{kt^\prime}) \leq \frac{C}{n(\tau-b)}, \\
& E(N_{1k}(\tau) N_{3k}(b,\tau)) = \frac{1}{n\tau} \frac{1}{n(\tau-b)} E \left[\left(\sum_{t=1}^{n\tau} (X_{kt} - E(X_{kt})) \right) \left(\sum_{t=nb+1}^{n\tau} (X_{kt} - E(X_{kt})) \right) \right] \\
& \hspace{3.5 cm}= \frac{1}{n\tau} \frac{1}{n(\tau-b)} \sum_{t=1}^{n\tau} \sum_{t^\prime=nb+1}^{n\tau} \mbox{Cum}(X_{kt},X_{kt^\prime}) \leq \frac{C}{n}, \\
& E(N_{2k}(b) N_{3k}(b,\tau)) = \frac{1}{n(1-b)} \frac{1}{n(\tau-b)} E \left[\left(\sum_{t=nb+1}^{n} (X_{kt} - E(X_{kt})) \right) \left(\sum_{t=nb+1}^{n\tau} (X_{kt} - E(X_{kt})) \right) \right] \\
& \hspace{3.5 cm}= \frac{1}{n(1-b)} \frac{1}{n(\tau-b)} \sum_{t=nb+1}^{n} \sum_{t^\prime=nb+1}^{n\tau} \mbox{Cum}(X_{kt},X_{kt^\prime}) \leq \frac{C}{n} \\
& E(N_{1k}(b) N_{3k}(b,\tau)) = \frac{1}{nb} \frac{1}{n(\tau-b)} E \left[\left(\sum_{t=1}^{nb} (X_{kt} - E(X_{kt})) \right) \left(\sum_{t=nb+1}^{n\tau} (X_{kt} - E(X_{kt})) \right) \right] \\
& \hspace{3.5 cm}= \frac{1}{nb} \frac{1}{n(\tau-b)} \sum_{t=1}^{nb} \sum_{t^\prime=nb+1}^{n\tau} \mbox{Cum}(X_{kt},X_{kt^\prime}) \leq \frac{C}{n}, \\
& E(N_{2k}(\tau) N_{3k}(b,\tau)) = \frac{1}{n(1-\tau)} \frac{1}{n(\tau-b)} E \left[\left(\sum_{t=n\tau+1}^{n} (X_{kt} - E(X_{kt})) \right) \left(\sum_{t=nb+1}^{n\tau} (X_{kt} - E(X_{kt})) \right) \right] \\
& \hspace{3.5 cm}= \frac{1}{n(1-\tau)} \frac{1}{n(\tau-b)} \sum_{t=n\tau+1}^{n} \sum_{t^\prime=nb+1}^{n\tau} \mbox{Cum}(X_{kt},X_{kt^\prime}) \leq \frac{C}{n}.
\end{align*}
Therefore,
\begin{align*}
& \sup_{k,b} E(N_{1k}^2 (b)) \leq Cn^{-1},\ \ \sup_{k,b} E(N_{2k}^2(b)) \leq Cn^{-1},\ \ \sup_{k} E(N_{3k}(b,\tau)) \leq C(n(\tau-b))^{-1}, \\
 & \sup_{k,b<\tau} E(N_{1k}(\tau) N_{3k}(b,\tau)) \leq Cn^{-1},\ \ \sup_{k,b<\tau} E(N_{2k}(b) N_{3k}(b,\tau)) \leq Cn^{-1}, \\
 & \sup_{k,b<\tau} E(N_{1k}(b) N_{3k}(b,\tau)) \leq Cn^{-1},\ \ \sup_{k,b<\tau} E(N_{2k}(\tau) N_{3k}(b,\tau)) \leq Cn^{-1}.
\end{align*}
Since $N_{4k}$ is non-random, $E(N_{4k}^2) = N_{4k}^2$. 
Finally, since $E(N_{1k}(b)) = E(N_{2k}(b)) = E(N_{3k}(b,\tau)) = 0$ for all $b$, we have
\begin{align*}
\sup_{k,b} E(N_{1k}(b)N_{4k}) = 0,\ \ \sup_{k,b}E(N_{2k}(b)N_{4k}) = 0\ \ \text{and}\ \ \sup_{k,b}E(N_{3k}(b,\tau)N_{4k}) = 0.
\end{align*}
Hence, Lemma 5.2 % \ref{lem: expectation}
 is established under (D1), (D2) and (SNR1).

Note that by (A1) and for some $C>0$, 
\begin{align*}
\mbox{Var}(N_{1k}^2(b)) &= E(N_{1k}^4 (b)) - (E(N_{1k}^2 (b)))^2 \\
& = E \bigg(\frac{1}{nb}\sum_{t=1}^{nb}(X_{kt}-E(X_{kt})) \bigg)^4 - \bigg(E \bigg(\frac{1}{nb}\sum_{t=1}^{nb}(X_{kt}-E(X_{kt})) \bigg)^2 \bigg)^2 \\
& = \frac{1}{(nb)^4}  \sum_{\stackrel{t_1,t_2,}{t_3,t_4=1}}^{nb}\bigg[\mbox{Cum}(X_{kt_1},X_{kt_2},X_{kt_3},X_{kt_4}) + \mbox{Cum}(X_{kt_1},X_{kt_2}) \mbox{Cum}(X_{kt_3},X_{kt_4}) \bigg] \\
& \hspace{5 cm} + \frac{1}{(nb)^4} \bigg( \sum_{t=1}^{nb}\sum_{t^\prime = 1}^{nb} \mbox{Cum}(X_{kt},X_{kt^\prime}) \bigg)^2\\
& \leq  Cn^{-3} + Cn^{-2}  \leq Cn^{-2}.
\end{align*}
Therefore, $\sup_{k,b} \mbox{Var}(N_{1k}^2 (b)) \leq Cn^{-2}$. Similarly, $\sup_{k,b} \mbox{Var}(N_{2k}^2 (b)) \leq Cn^{-2}$.

\noindent Next,
\begin{align*}
& \mbox{Var}(N_{3k}^2(b,\tau)) = E(N_{3k}^4 (b,\tau)) - (E(N_{3k}^2 (b,\tau)))^2 \\
 =\ & E \bigg(\frac{1}{n(\tau-b)}\sum_{t=nb+1}^{n\tau}(X_{kt}-E(X_{kt})) \bigg)^4 - \bigg(\frac{1}{(n(\tau-b))^2} \sum_{t=nb+1}^{n\tau} \mbox{Var}(X_{kt}) \bigg)^2 \\
 =\ & \frac{1}{(n(\tau-b))^4}  \sum_{\stackrel{t_1,t_2,}{t_3,t_4=nb+1}}^{n\tau}\bigg[\mbox{Cum}(X_{kt_1},X_{kt_2},X_{kt_3},X_{kt_4}) + \mbox{Cum}(X_{kt_1},X_{kt_2}) \mbox{Cum}(X_{kt_3},X_{kt_4}) \bigg] \\
& \hspace{5 cm} + \frac{1}{(n(\tau-b))^4} \bigg( \sum_{t=nb+1}^{n\tau}\sum_{t^\prime = nb+1}^{n\tau} \mbox{Cum}(X_{kt},X_{kt^\prime}) \bigg)^2\\
\leq\ &  C (n(\tau-b))^{-3} + C (n(\tau-b))^{-2}  \leq C(n(\tau-b))^{-2}.
\end{align*}
Therefore,  $\sup_{k} \mbox{Var}(N_{3k}^2(b,\tau)) \leq C(n(\tau-b))^{-2}$.  Similarly, $\sup_{k} \mbox{Var}(N_{1k}(b)N_{3k}(b,\tau)) \leq C(n(\tau-b))^{-2}$ and $\sup_{k} \mbox{Var}(N_{2k}(b)N_{3k}(b,\tau)) \leq C(n(\tau-b))^{-2}$. Moreover, since $N_{4k}$ is non-random, $\mbox{Var}(N_{4k}^2) = 0$. Also by Lemma  5.2, %\ref{lem: expectation}, 
we get
\begin{align*}
\mbox{Var}(N_{4k} N_{3k}(b,\tau)) = N_{4k}^2 E(N_{3k}(b,\tau))^2 \leq CN_{4k}^2 (n(\tau-b))^{-1}.
\end{align*}
Therefore,   $\sup_ {k} \mbox{Var}(N_{4k} N_{3k}(b,\tau)) \leq CN_{4k}^2 (n(\tau-b))^{-1}$.
Similarly,  $\sup_ {k} \mbox{Var}(N_{4k} N_{2k}(b,\tau)) \leq CN_{4k}^2 n^{-1}$. 

\noindent This completes the of proof Lemma 5.3 % \ref{lem: variancelse} 
under (D1), (D2) and (SNR1), and hence the results in Theorem 2.5
%\ref{thm: lse3}
 are established. \qed

\subsection{Supplement to Section 3}
\label{supp:sec3}
In this section, we present in detail certain examples that were omitted for space considerations from Section $3$. 

\begin{Example} \label{example: expfam}
\textit{\textbf{Exponential family}}. 
Consider the model in (\ref{eqn: exppdfch}) as presented in Example \ref{example: expfam1}.  As mentioned in Section \ref{sec: lse}, (A1) implies $\sup_{k,n,t} \text{Var}(X_{kt}) < \infty$. Here, it reduces to %the following assumption.
\newline
$(\tilde{B}_1)$: $\sup_{\lambda \in \Lambda} \beta^{\prime \prime}(\lambda) < \infty$.
%\begin{eqnarray} \label{eqn: a2mle}
%\sup_{\lambda \in \Lambda} \beta^{\prime \prime}(\lambda) < \infty.
%\end{eqnarray} 
\\ \noindent  Further, Assumption (A3) is equivalent to 
\newline
$(\tilde{B}_2)$: $\inf_{\lambda \in \Lambda} \beta^{\prime \prime}(\lambda) > \epsilon\ \ \text{ for some $\epsilon >0$}$.
% \begin{eqnarray} \label{eqn: a4mle}
% \inf_{\lambda \in \Lambda} \beta^{\prime \prime}(\lambda) > \epsilon\ \ \text{ for some $\epsilon >0$}.
% \end{eqnarray}

Therefore, all the results stated in Section \ref{sec: lse} continue to hold if we assume  (A1), (A2) and $\tilde{B}_2$ (instead of (A3)), (A4)-(A7)  and (SNR1) with $\mu_{1k} = \beta^\prime (\theta_k),\  \sigma_{1k}^2 = \beta^{\prime \prime}(\theta_k)$ and $\mu_{2k} = \beta^\prime (\eta_k),\ \sigma_{2k}^2 = \beta^{\prime \prime}(\eta_k)$ for all $k\geq 1$.
\\ \indent As has been previously mentioned, stronger assumptions are needed for the maximum likelihood estimator $\hat{\tau}_{n,\text{MLE}}$ compared to its
least squares counterpart. Note that Assumptions (B1)-(B5), (B8) and (B9) are automatically satisfied by the one parameter natural exponential family %$\{f_{\lambda}: \lambda \in \Lambda \subset \mathbb{R}\}$ 
defined in (\ref{eqn: defnexp}). Moreover, by $(\tilde{B}_1)$ and $(\tilde{B}_2)$, we have $G_2(x) = C$ for some $C>0$. Therefore, we only need to assume (B7), which is equivalent to (A1). %We also consider (B3) and (B10), 
As we have seen in Theorems \ref{thm: mle1g} and \ref{thm: mle2g}, some times we further require $\log m = o(n)$, which is a stronger assumption compared to those posited in Section \ref{sec: lse}.

\begin{remark}
Suppose for each $k,t \geq 1$, $X_{kt}$ are generated by a Gaussian distribution
with unknown mean $\theta_k I(t \leq \tau_n) + \eta_k I(t > \tau_n)$ and known variance 
$\sigma_{1k}^2 I(t \leq \tau_n) + \sigma_{2k}^2 I(t > \tau_n)$. Then, by Theorems \ref{thm: mle1g} and \ref{thm: mle2g},  
we do not need the stronger assumption $\log m(n) = o(n)$. 
\end{remark}

Next, we present asymptotic properties of the maximum likelihood estimator of the change point $\tau_n$. Recall %For a fixed $b \in \{1/n,2/n,\ldots, (n-1)/n,1\}$ and $k \geq 1$,  let 
$\hat{\mu}_{1k}(b)$ and $\hat{\mu}_{2k}(b)$ from (\ref{eqn: optilse}). 
%\begin{eqnarray}
%\hat{p}_k(b) &=& \frac{1}{nb} \sum_{t=1}^{nb} X_{kt}\ \ \text{and}\ \ \hat{q}_k(b) = \frac{1}{n(1-b)} \sum_{t=nb+1}^{n} X_{kt}.
%\end{eqnarray}
We can estimate $\theta_k$ and $\eta_k$ respectively by $\hat{\theta}_{k}(b)$ and $\hat{\eta}_{k}(b)$, where
\begin{eqnarray}
\hat{\theta}_{k}(b) &=& \arg \max_{\lambda} \frac{1}{n}\sum_{t=1}^{nb} \log f_{\lambda}(x_{kt}) = \arg \max_{\lambda}   b(\hat{\mu}_{1k}(b) \lambda + \beta(\lambda))\ \ \text{and} \\
\hat{\eta}_{k}(b) &=& \arg \max_{\lambda} \frac{1}{n}\sum_{t=nb+1}^{n} \log f_{\lambda}(x_{kt}) = \arg \max_{\lambda}   (1-b)(\hat{\mu}_{2k}(b) \lambda + \beta(\lambda)).
\end{eqnarray}
Hence, it is easy to obtain $\hat{\theta}_{k}(b) = \beta^{\prime -1}(\hat{\mu}_{1k} (b) ) = \alpha(\hat{\mu}_{1k} (b))$ and $\hat{\eta}_{k}(b) = \alpha(\hat{\mu}_{2k} (b))$.

The maximum likelihood estimator  $\hat{\tau}_{n,\text{MLE}}$ of $\tau_n$ is obtained as:
\begin{eqnarray}
\hat{\tau}_{n,\text{MLE}} &=& \arg \max_{b \in (c^{*},1-c^{*})} \sum_{k=1}^{m} L_{k,n}(b)\ \ \text{where,} \\
L_{k,n}(b) &=& \frac{1}{n}\bigg [\sum_{t=1}^{nb}\log f_{\hat{\theta}_k(b)}(x_{kt}) + \sum_{t=nb+1}^{n}\log f_{\hat{\eta}_k(b)}(x_{kt}) \bigg ] \nonumber \\
&=&  b [\hat{\mu}_{1k} (b) \alpha(\hat{\mu}_{1k} (b)) + \beta(\alpha(\hat{\mu}_{1k} (b)))] + (1-b)[\hat{\mu}_{2k}(b) \alpha(\hat{\mu}_{2k} (b)) + \beta(\alpha(\hat{\mu}_{2k} (b)))]. \nonumber
\end{eqnarray}

The following result provides the corresponding rate of convergence and its proof follows as a direct consequence of Theorem \ref{thm: mle1g}.

\begin{theorem} \label{thm: mle1}
Suppose that $(\tilde{B}_1)$, $(\tilde{B}_2)$, (B7),  (SNR2) hold and $\log m(n) = o(n)$. Then, 
\begin{eqnarray} \label{eqn: msethm1}
n||\theta - \eta||_2^2 (\hat{\tau}_{n,\text{MLE}}-\tau_n) = O_{P}(1).
\end{eqnarray}
Further, suppose $\beta(x) = Cx^2$ for some $C>0$; i.e., $f_{\lambda}(x) = (\sqrt{2\pi} \sigma)^{-1}e^{-0.5\sigma^{-2}(x-\lambda)^2}$ for some known $\sigma >0$, then (\ref{eqn: msethm1}) continues to hold under the weaker assumption  (SNR1). 
\end{theorem}

Note that Assumptions (A1) and (B7) are equivalent. Therefore, for the Gaussian likelihood function, we do not require any stronger assumptions than those used in Theorem \ref{thm: lse1}.  Moreover, as previously discussed, (A1) or (B7) implies $(\tilde{B}_1)$.  However,  assumptions
$(\tilde{B}_2)$, (SNR2) and $\log m(n) = o(n)$ are additionally required for other distributions members of the one parameter exponential family
to establish results for $\hat{\tau}_{n,\text{MLE}}$ vis-a-vis those for $\hat{\tau}_{n,\text{LSE}}$.

For establishing the asymptotic distribution of $n||\theta - \eta||_2^2 (\hat{\tau}_{n,\text{MLE}}-\tau_n)$, note that under the assumptions 
given in Theorem \ref{thm: mle1}, it becomes degenerate at $0$, if $||\theta - \eta||_2 \to \infty$. 
Analogously to the results given in Section \ref{sec: lse}, additional assumptions are needed for the cases when $||\theta - \eta||_2 \to 0$ or $c>0$.
\\ \indent Recall the set $\mathcal{K}_n$ in (\ref{eqn: partition}). Let,
\begin{eqnarray} \label{eqn: gammamle}
\gamma_{_{\text{MLE, EXP}}}^2 = \lim \frac{\sum_{k=1}^{m}(\theta_k - \eta_k)^2 (\beta^{\prime \prime}(\theta_k))}{||\theta-\eta||_2^2}\ \ \text{and}\ \ \
\gamma_{_{\text{MLE, EXP}}}^{*2} = \lim \sum_{k \in \mathcal{K}_n}(\theta_k - \eta_k)^2 (\beta^{\prime \prime}(\theta_k)).\ \ \nonumber 
\end{eqnarray}
Note that  $(\tilde{B}_2)$  implies $\gamma_{_{\text{MLE, EXP}}} >0$. Moreover, $\gamma_{_{\text{MLE, EXP}}}^* >0$ if and only if $\lim \sum_{k \in \mathcal{K}_n}(\theta_k - \eta_k)^2 >0$.  Existence of $\gamma_{_{\text{MLE, EXP}}}$ and $\gamma_{_{\text{MLE, EXP}}}^*$ are required respectively for $||\theta - \eta||_2 \to 0$ and $c>0$.  Moreover, this is guaranteed by the conditions in Propositions \ref{prop: 1} and \ref{prop: 2} when $\mu_{1k}$, $\mu_{2k}$, $\sigma_{1k}^2$ and $\sigma_{2k}^2$ are respectively replaced by $\theta_k$, $\eta_k$, $\beta^{\prime \prime}(\theta_k)$ and $\beta^{\prime \prime}(\eta_k)$. 
\\ \indent Further, (B11) reduces to the following assumption. \\
\noindent $(\tilde{B}_3)$ $\beta(\cdot)$ has bounded  third derivative. 

\noindent For the Gaussian likelihood, $(\tilde{B}_3)$ is always true, since $\beta(\cdot)$ is a quadratic function.

The following Theorem describes the asymptotic distribution of $\hat{\tau}_{n,\text{MLE}}$, which follows from Theorem \ref{thm: mle2g}.
\begin{theorem} \label{thm: mle2}
Suppose $(\tilde{B}_1)$, $(\tilde{B}_2)$, (B7)  and (SNR2) hold. Then, the following statements are true.

\noindent $(a)$ If $||\theta-\eta||_2 \to \infty$ and $\log m(n) = o(n)$, then $\lim_{n \to \infty} P(\hat{\tau}_{n,\text{MLE}}=\tau_n) =1$.

\noindent $(b)$ If $\gamma_{_{\text{MLE, EXP}}}$ exists, $(\tilde{B}_3)$ holds and  $||\theta -\eta||_2 \to 0$, then
\begin{eqnarray}
n||\theta - \eta||_2^2 (\hat{\tau}_{n,\text{MLE}}-\tau_n) &\stackrel{\mathcal{D}}{\to} &  \arg \max_{h \in \mathbb{R}} (-0.5\gamma_{_{\text{MLE, EXP}}}^{2}|h| + \gamma_{_{\text{MLE, EXP}}}B_h) \nonumber \\
&=& \gamma_{_{\text{MLE, EXP}}}^{-2}\arg \max_{h \in \mathbb{R}} (-0.5|h| + B_h),
\end{eqnarray}
where $B_h$ denotes the standard Brownian motion.

\noindent $(c)$ If $\gamma_{_{\text{MLE, EXP}}}^*$ exists, $(\tilde{B}_3)$, (A4) and (B10) hold, $\sup_{k \in \mathcal{K}_n}|\theta_k -\eta_k | \to 0$ and $||\theta - \eta||_2 \to c >0$, then
\begin{eqnarray}
n (\hat{\tau}_{n,\text{MLE}} -\tau_n) &\stackrel{\mathcal{D}}{\to}& \arg \max_{h \in \mathbb{Z}} (D_2(h) + C_2(h) + A_2(h)) \nonumber 
%&=& c^2 \arg \max_{c^2h \in c^2\mathbb{Z}} (D(c^2 h) + C(c^2 h) + A(c^2 h)) \nonumber 
\end{eqnarray}
where  for each $h \in \mathbb{Z}$,
\begin{eqnarray}
D_2 ( h+1)-D_2 ( h) &=& -0.5\text{Sign}(h)  \gamma_{_{\text{MLE, EXP}}}^{*2}, \\
C_2 (h+1) - C_2 ( h) &=& \gamma_{_{\text{MLE, EXP}}}^{*} W_h,\ \ W_h \stackrel{\text{i.i.d.}}{\sim} \mathcal{N}(0,1), \label{eqn: msethm2ca}\\
A_2 (h+1) - A_2 ( h) &=& \sum_{k \in \mathcal{K}_0} \bigg[Z_{kh}(\eta_k^* - \theta_k^*) -(\beta(\eta_k^*) - \beta(\theta_k^*)) \bigg],\  %\nonumber %\\
  %Z_{kh} &\stackrel{d}{=}& X_{1k}I(h \leq 0) + X_{nk}I(h > 0) 
  \label{eqn: msethm2aa}
\end{eqnarray}
and $\{Z_{kh}\}$ are independently distributed with $Z_{kh} \stackrel{d}{=} X_{1k}^{*}I(h \leq 0) + X_{2k}^{*}I(h > 0)$. % the above increments are independently  distributed.

Further, suppose $\beta(x) = Cx^2$ for some $C>0$ i.e., $f_{\lambda}(x) = (\sqrt{2\pi} \sigma)^{-1}e^{-0.5\sigma^{-2}(x-\lambda)^2}$ for some known $\sigma >0$, then the asymptotic distributions in (a)-(c) continue to hold under the weaker assumption (SNR1). 
\end{theorem}

In the ensuing discussion from Theorem \ref{thm: mle1}, for the Gaussian likelihood there is no requirement for stronger assumptions compared to those in Theorem \ref{thm: lse2}. For other likelihoods, we additionally require assumptions $(\tilde{B}_2)$, (SNR2) and $\log m(n) = o(n)$ for the cases of $||\theta - \eta||_2 \to \infty$ and $(\tilde{B}_3)$, and (SNR2) when $||\theta - \eta||_2 \to 0$ or $c>0$. 

The following comments provide additional insights.\\
\noindent (I) Suppose for each $k,t \geq 1$, $X_{kt}$ are generated from a Gaussian distribution with unknown mean $\theta_k I(t \leq \tau_n) + \eta_k I(t > \tau_n)$ and known variance $\sigma^2$. Then, 
$\hat{\tau}_{n,\text{LSE}} = \hat{\tau}_{n,\text{MLE}}$.

\noindent (II) Suppose for each $k \geq 1$,  $E(X_{kt}) = \mu_{1k}I(t \leq \tau_n) + \mu_{2k}I(t > \tau_n)$. If $||\theta -\eta||_2 \to 0$, then
\begin{eqnarray} \label{eqn: gammamlebb}
\gamma_{_{\text{MLE, EXP}}}^2 &=& \lim \frac{\sum_{k=1}^{m}(\mu_{1k} - \mu_{2k})^2 (\beta^{\prime \prime}(\theta_k))^{-1}}{||\theta-\eta||_2^2}. 
\end{eqnarray}
Also if $||\theta - \eta||_2  \to c>0$ and $\sup_{k \in \mathcal{K}_n}|\theta_k -\eta_k | \to 0$, then
\begin{eqnarray} \label{eqn: gammamle2bb}
\gamma_{_{\text{MLE, EXP}}}^{*2} &=& \lim \sum_{k \in \mathcal{K}_n}(\mu_{1k} - \mu_{2k})^2 (\beta^{\prime \prime}(\theta_k))^{-1}.
\end{eqnarray}

\begin{proof}
%We shall prove (\ref{eqn: gammamlebb}). 
Note that for all $k \geq 1$, we have $\mu_{1k} = \beta^\prime (\theta_k)$, $\mu_{2k} = \beta^\prime (\eta_k)$.  Therefore  $\theta_k = \alpha(\mu_{1k})$ and $\eta_k = \alpha(\mu_{2k})$ where $\alpha = \beta^{\prime -1}$.  By $(\tilde{B}_1)$ and applying the mean value theorem, 
\begin{eqnarray}
|\mu_{1k} - \mu_{2k}| = |\beta^\prime (\theta_k)-\beta^\prime (\eta_k) | \leq C|\theta_k - \eta_k|.
\end{eqnarray}
Thus $||\theta -\eta||_2 \to 0$ implies $||\mu_1 - \mu_2||_2 \to 0$ and hence
\begin{eqnarray}
\gamma_{_{\text{MLE, EXP}}}^2 &=& \lim \frac{\sum_{k=1}^{m}(\theta_{k} - \eta_{k})^2 (\beta^{\prime \prime}(\theta_k))}{||\theta-\eta||_2^2} \nonumber \\
&=& \lim \frac{\sum_{k=1}^{m}(\alpha(\mu_{1k}) - \alpha(\mu_{2k}))^2 (\beta^{\prime \prime}(\theta_k))}{||\theta-\eta||_2^2} \nonumber \\
&=& \lim \frac{\sum_{k=1}^{m}(\mu_{1k} - \mu_{2k})^2 (\beta^{\prime \prime}(\theta_k))^{-2} (\beta^{\prime \prime}(\theta_k))}{||\theta-\eta||_2^2} \nonumber \\
&=& \lim \frac{\sum_{k=1}^{m}(\mu_{1k} - \mu_{2k})^2 (\beta^{\prime \prime}(\theta_k))^{-1}}{||\theta-\eta||_2^2}.
\end{eqnarray}
The completes the proof of (\ref{eqn: gammamlebb}). Similar arguments work for (\ref{eqn: gammamle2bb}). 
\end{proof}

\noindent (III) It is immediate from (II) that Theorem \ref{thm: mle2} continues to hold if we replace $||\theta -\eta||_2$, $|\theta_k - \eta_k|$ and $\gamma_{_{MLE, EXP}}^2$ respectively by by $||\mu_{1} - \mu_{2}||_2$, $|\mu_{1k} - \mu_{2k}|$ and  
\begin{eqnarray} \label{eqn: gammatilde}
\tilde{\gamma}_{_{MLE, EXP}}^2 = \lim \frac{ \sum_{k=1}^{m}(\mu_{1k} - \mu_{2k})^2 (\beta^{\prime \prime}(\theta_k))^{-1}}{||\mu_1 - \mu_2||^{2}}.
\end{eqnarray}

\noindent (IV) Recall $\{\mu_{ik}\}$, $\gamma_{_\text{LSE}}$ and $\tilde{\gamma}_{_\text{MLE, EXP}}$ respectively from (II),  (\ref{eqn: gammalse}) and (\ref{eqn: gammatilde}). Suppose (A1) (equivalently (B7)), $(\tilde{B}_1)$, $(\tilde{B}_2)$ (equivalently (A3)),  $(\tilde{B}_3)$ and (SNR2) hold, $\gamma_{_{LSE}}$ and $\tilde{\gamma}_{_\text{MLE, EXP}}$ exist and $||\mu_1 - \mu_2||_2 \to 0$. Then, by Remark \ref{rem: lse1}, Theorem \ref{thm: mle2}(b)  and (III), 
\begin{eqnarray}
\text{Var}(n||\mu_1 -\mu_2||_2^2  (\hat{\tau}_{n,\text{LSE}} -\tau_n)) &=& \gamma_{_\text{LSE}}^4 \text{Var}(\arg \max_{h \in \mathbb{R}} (-0.5|h| + B_h)) \nonumber \\
\text{Var}(n ||\mu_1 -\mu_2||_2^2 (\hat{\tau}_{n,\text{MLE}} -\tau_n)) &=& \tilde{\gamma}_{_\text{MLE, EXP}}^{-4} \text{Var}(\arg \max_{h \in \mathbb{R}} (-0.5|h| + B_h)). \nonumber
\end{eqnarray}
Note that $\gamma_{_\text{LSE}}^2$ and $\tilde{\gamma}_{_\text{MLE, EXP}}^{-2}$ are respectively weighted arithmetic and harmonic means of $\{\sigma_{1k}^2\}$.  Therefore, the variance of $n ||\mu_1 -\mu_2||_2^2 (\hat{\tau}_{n,\text{MLE}} -\tau_n)$ is smaller than that of 
$n ||\mu_1 -\mu_2||_2^2 (\hat{\tau}_{n,\text{LSE}} -\tau_n)$.

\noindent (V) Suppose $||\theta -\eta||_2 \to c>0$. Recall the partition of $\{1,2,\ldots,m(n)\}$ into $\mathcal{K}_n$ and $\mathcal{K}_n^c$ given in (\ref{eqn: partition}).  Suppose $\mathcal{K}_n^c$ is the empty set. Then, under the assumptions of Theorem \ref{thm: mle2}(c), 
\begin{eqnarray}
n(\hat{\tau}_{n,\text{MLE}} -\tau_n) \stackrel{\mathcal{D}}{\to}  \gamma_{_\text{MLE, EXP}}^{*-2} \arg \max_{h \in \mathbb{R}} (-0.5|h| + B_h). \label{eqn: mle2c}
\end{eqnarray}
%Suppose $\gamma_{_{MLE}}^{*}$,  defined in (\ref{eqn: gammalse}), exists.  
Then, by (\ref{eqn: mle2c}), Remark  \ref{rem: lserem2} and using similar arguments given in (IV), we obtain that the variance of $n  (\hat{\tau}_{n,\text{MLE}} -\tau_n)$ is smaller than that of $n  (\hat{\tau}_{n,\text{LSE}} -\tau_n)$.
\end{Example}

\begin{Example} \label{example: bermle} \textbf{Bernoulli data}, \textit{continuation of Example \ref{example: berlse}}.
Suppose the data $\{X_{kt}\}$ are generated from model (\ref{eqn: berlse1}). As discussed in Example \ref{example: berlse}, %(A1) or equivalently 
(B6) is satisfied for this model. It is easy to see that (\ref{eqn: berlse2}) implies $(\tilde{B}_1)$ and $(\tilde{B}_2)$. Therefore, if (\ref{eqn: berlse2}) and 
(SNR2) hold and $\log m(n) = o(n)$, then the conclusions of Theorems \ref{thm: mle1} and \ref{thm: mle2}(a) continue to hold. Observe that in this case, we require strong assumptions (\ref{eqn: berlse2}),  (SNR2) and $\log m(n) = o(n)$ compared to those used in Example \ref{example: berlse} when $||\mu_1 - \mu_2||_2 \to \infty$. 
\\ \indent Further, (\ref{eqn: berlse2}) implies $(\tilde{B}_3)$.  Suppose (b)-(e) in Proposition \ref{prop: 1} hold after replacing  $\mu$ by $p$. %I(i=1) + p_{2k}(n)I(i=2)$. 
Then, under (\ref{eqn: berlse2}) and (SNR2), the conclusions of Theorem \ref{thm: mle2}(b)  continue to hold for the model  (\ref{eqn: berlse1}). 
\\ \indent Suppose (e)-(g) in Propositions \ref{prop: 2} hold after replacing  $\mu$ by $p$. %(n)I(i=1) + q_k(n)I(i=2)$. 
Then, under (\ref{eqn: berlse2}), (\ref{eqn: extraber}), (A4) and (SNR2), the conclusions of Theorem \ref{thm: mle2}(c)  holds for the model (\ref{eqn: berlse1}).
%Therefore, under (\ref{eqn: berlse2}), (\ref{eqn: berlse3}) and (SNR2),  the conclusion of Theorem \ref{thm: mle2}(b) and (c) holds. 
\\ \indent Note that (SNR2) is the only  assumption that we additionally need compared to those used in Example \ref{example: berlse} when $||\mu_1 - \mu_2||_2 \to 0$ or $c^2>0$.
\end{Example}

\begin{Example} \label{example: poimle} \textbf{Poisson data}, \textit{continuation of Example \ref{example: poilse}}.
Suppose  $\{X_{kt}\}$ are generated from model (\ref{eqn: poilse1}).   (B6) is satisfied for this model if last inequality of (\ref{eqn: poilse2}) holds. Moreover,  (\ref{eqn: poilse2}) implies $(\tilde{B}_1)$ and $(\tilde{B}_2)$. Therefore, if  (\ref{eqn: poilse2}) and (SNR2) hold and $\log m(n) = o(n)$, then the conclusions of Theorems \ref{thm: mle1} and \ref{thm: mle2}(a) continue to hold.  \\
\indent Also (\ref{eqn: poilse2}) implies $(\beta 3)$. Suppose (b)-(e) in Proposition \ref{prop: 1} hold with $\mu$ replaced by $\lambda$. Then, under  (\ref{eqn: poilse2}) and (SNR2), the conclusions of Theorem \ref{thm: mle2}(b)  hold for model (\ref{eqn: poilse1}). \\
\indent Suppose (e)-(g) in Proposition \ref{prop: 2} hold with $\mu$ replaced by $\lambda$. Then, under  (\ref{eqn: poilse2}), (\ref{eqn: extraber}) with $p$ replaced by $\lambda$, (A4) and (SNR2), the conclusions of Theorem \ref{thm: lse2}(c)  hold for the model in (\ref{eqn: poilse1}). 
\\ \indent Again we observe the need for stronger assumptions than the least sqaures counterpart. 
%\\ \indent As in Example \ref{example: bermle}, here also (SNR2) is the only  assumption that we additionally need compared to those used in
%Example \ref{example: poilse}.
\end{Example}

\begin{Example} \textbf{Normal data}, \textit{continuation of Example \ref{example: normallse}}. 
Suppose $\{X_{kt}\}$ are generated from model (\ref{eqn: normallse1}). Then, by Theorems \ref{thm: mle1} and \ref{thm: mle2}, all results in Example \ref{example: normallse} continue to hold for the estimator $\hat{\tau}_{n,\text{MLE}}$. 
\end{Example}

\begin{Example} \label{example: normaltheta} \textbf{(A curved exponential distribution.)} Let $\Lambda$ be a bounded open subset of $\mathbb{R}$, such that $\inf_{x \in \Lambda} |x| > C >0$.
% Let $\bar{\Lambda} \subset \mathbb{R}$ is a compact set which does not contain $0$ and $\Lambda$ be the interior set of $\Lambda$.  
 Consider the family of  $\mathcal{N}(\lambda,\lambda^2)$ distributions, where $\lambda \in \Lambda$. Note that this family satisfies (B1)-(B3). 
Further, define $\delta = \lambda^{-1}$. As $\Lambda$ is bounded away from $0$ and $\infty$, it is equivalent to working with $\delta$, instead of 
$\lambda$.  For a given observation $X=x$, the log-likelihood of $\delta$ is given by
% \begin{eqnarray}
 $L(\delta) =  \log \delta - 0.5(x^2\delta^{2} + 1 -2x\delta)$. 
 %\end{eqnarray}
 Thus (B4) holds and 
 \begin{eqnarray}
 \frac{\partial^2}{\partial \delta^2} L(\delta) = - \delta^{-2} - x^2. \nonumber
 \end{eqnarray}
 Therefore, (B5) holds with $G_2(x) = 1+x^2$.   Moreover, as we are studying Gaussian distributions, clearly (B6)-(B9) hold. 
 \\ \indent Suppose our data $\{X_{kt}\}$  are independently generated from $\mathcal{N}(\theta_k,\theta_k^2)$ and $\mathcal{N}(\eta_k,\eta_k^2)$, respectively, when $t \leq n\tau_n$ and $t>n\tau_n$. Moreover, $\theta_k \neq \eta_k$ for at least one $k$ and hence $\tau_n$ is the common change point.  Let $\theta^{-1} = (\theta_1^{-1},\theta_2^{-1},\ldots,\theta_m^{-1})$ and   $\theta^{-1} = (\eta_1^{-1},\eta_2^{-1},\ldots,\eta_m^{-1})$. By Theorem \ref{thm: mle1g}, if  (SNR1) holds and $\log m(n) = o(n)$,  then
%\begin{eqnarray}
$n||\theta^{-1} - \eta^{-1}||_2^2 (\hat{\tau}_{n,\text{MLE}} -\tau_n) = O_{P}(1)$. 
%\end{eqnarray}
By Theorem \ref{thm: mle2g}(a), under  (SNR1), $\log m(n) = o(n)$ and $||\theta^{-1} - \eta^{-1}||_2 \to \infty$, 
$P(\hat{\tau}_{n,\text{MLE}} =\tau_n) \to 1.$ 

Note that $\frac{\partial^3}{\partial \delta^3} L(\delta) = 2\delta^{-3}$ and thus (B11) holds. Moreover, $I(\lambda)= -E\frac{\partial^2}{\partial \delta^2} L(\delta) = 3\delta^{-2}$. Suppose the following limit exists:
\begin{eqnarray}
\sigma^2 = \lim \frac{3\sum_{k=1}^{m}(\theta_k^{-1} - \eta_k^{-1})^{2} \theta_k^{2}}{||\theta^{-1}-\eta^{-1}||_2^2}.
\end{eqnarray}
\noindent  Hence, by Theorem \ref{thm: mle2g}(b), if  (SNR1) holds  and $||\theta^{-1} - \eta^{-1}||_2 \to 0$, then
\begin{eqnarray}
n||\theta^{-1} - \eta^{-1}||_2^2 (\hat{\tau}_{n,\text{MLE}} -\tau_n) \stackrel{}{\to} \frac{1}{\sigma^2} \arg \max_{h \in \mathbb{R}} (-0.5|h| + B_h). \nonumber 
\end{eqnarray}
Recall the set $\mathcal{K}_0$ and $\mathcal{K}_n=\mathcal{K}_0^c$ in (\ref{eqn: partition}).  Suppose the following limit exists:
\begin{eqnarray}
\sigma_1^2 = \lim \sum_{k \in \mathcal{K}_n} 3(\theta_k^{-1} - \eta_k^{-1})^2 \theta_{k}^2. \nonumber
\end{eqnarray}  
Suppose (A4), (B10) and (SNR1) hold, $\sup_{k \in \mathcal{K}_n}|\theta_k^{-1} -\eta_k^{-1} |_2 \to 0$ and $||\theta^{-1} - \eta^{-1}||_2 \to c >0$, then
\begin{eqnarray}
n (\hat{\tau}_{n,\text{MLE}} -\tau) &\stackrel{\mathcal{D}}{\to}& \arg \max_{h \in \mathbb{Z}} (D_2(h) + C_2(h) + A_2(h)) \nonumber 
%&=& c^2 \arg \max_{c^2h \in c^2\mathbb{Z}} (D(c^2 h) + C(c^2 h) + A(c^2 h)) \nonumber 
\end{eqnarray}
where  for each $h \in \mathbb{Z}$,
\begin{eqnarray}
D_2 ( h+1)-D_2 ( h) &=& -0.5  \sigma_1^2,\ \  %\nonumber \\
C_2 (h+1) - C_2 ( h) = \sigma_1^2W_h,\ \ W_h \stackrel{\text{i.i.d.}}{\sim} \mathcal{N}(0,1), \nonumber \\
A_2 (h+1) - A_2 ( h) &=& \sum_{k \in \mathcal{K}_0}  \left(\log (\theta_k^*/\eta_k^*) + 0.5Z_{kh}^2(\theta_k^{*-2}- \eta_k^{*-2}) - Z_{kh}(\theta_k^{*-1}-\eta_k^{*-1}) \right) \nonumber
\end{eqnarray}
and $\{Z_{kh}\}$ are independent with $Z_{kh} \stackrel{d}{=} \mathcal{N}(\theta_k^*,\theta_k^{*2})I(h \leq 0) + \mathcal{N}(\eta_k^*,\eta_k^{*2})I(h > 0).$
\end{Example}

\subsection{Supplement to Section 4.1: Adaptive inference for Gaussian time dependent data}
We observe the data $\{X_{kt}: 1 \leq k \leq m, 1 \leq t \leq n\}$ which are independent over $k$ but dependent over $t$. For all $t,t_1,t_2,\ldots, t_l, l,k \geq 1$,  suppose 
\begin{eqnarray}
&& {X}_{kt} = \tilde{Y}_{kt} + \mu_{1k}I(t \leq n\tau_n) + \mu_{2k}I(t > n\tau_n), \label{eqn: normal1} \\
&& (Y_{kt_1},Y_{kt_2},\ldots, Y_{kt_l}) \sim \mathcal{N}_{l} (0, ((\text{Cum}(Y_{kt_{l_i}},Y_{kt_{l_j}}) ))_{l \times l})\ \ \text{and}  \label{eqn: normal2} \\
&& \sum_{l=-\infty}^{\infty}  \sup_{k} |\text{Cum}(Y_{k1},Y_{k(l+1)})| < \infty \label{eqn: add}
\end{eqnarray}
where $\{Y_{kt}\}$ are not observable.  Let,
\begin{eqnarray}
\hat{\mu}_{1k} = \frac{1}{n\hat{\tau}_{n,\text{LSE}}} \sum_{t=1}^{n\hat{\tau}_{n,\text{LSE}}} X_{kt},\ \ \hat{\mu}_{2k} = \frac{1}{n(1-\hat{\tau}_{n,\text{LSE}})} \sum_{t=n\hat{\tau}_{n,\text{LSE}}+1}^{n} X_{kt}, \nonumber \\
\hat{{C}}_{k,l}  = \bigg(\frac{1}{n\hat{\tau}_{n,\text{LSE}}-l}\sum_{t=1}^{n\hat{\tau}_{n,\text{LSE}}-l} X_{kt} X_{k(t+l)}\bigg) - \hat{\mu}_{1k}^2.
\end{eqnarray}
We generate $\{\tilde{Y}_{kt,\text{LSE}}: k,t \geq 1\}$ independently over $k \geq 1$ and for all $t_1,t_2,\ldots, t_l, l \geq 1$, 
\begin{eqnarray}
(\tilde{Y}_{kt_1,\text{LSE}},\tilde{Y}_{kt_2,\text{LSE}},\ldots,\tilde{Y}_{kt_l,\text{LSE}}) \sim \mathcal{N}_{l} (0, ((\hat{C}_{k,|t_{i} - t_{j}| }))_{l \times l}).
\end{eqnarray}
Define $\tilde{X}_{kt,\text{LSE}} = \hat{\mu}_{1k}I(t \leq n \hat{\tau}_{n,\text{LSE}}) + \hat{\mu}_{2k}I(t >  n\hat{\tau}_{n,\text{LSE}}) + \tilde{Y}_{kt,\text{LSE}}\ \ \forall k,t \geq 1$.  Recall $\tilde{M}_n(h)$ from (4.2). % (\ref{eqn: tildetau}).
 Let $\tilde{h}_{n,\text{LSE}} = \arg \max_{h} \tilde{M}_n(h)$. The following theorem states asymptotic distribution of $\tilde{h}_{n,\text{LSE}}$.  %Its proof is given in Section \ref{sec:  deptildetau}.

\begin{theorem} \label{thm: deptildetau}
Suppose (\ref{eqn: normal1})-(\ref{eqn: add}) hold. Then the following statements are true.
\vskip 3pt
\noindent (a) Under (SNR1) and $||\mu_1 - \mu_2||_2 \to \infty$, $P(\tilde{h}_{n,\text{LSE}} = 0) \to 1$.
\vskip 3pt
\noindent (b) Suppose (D4) and (SNR3) hold and $||\mu_1 - \mu_2||_2 \to 0$. Then
\begin{eqnarray}
||{\mu}_1 -{\mu}_2||_2^2 \tilde{h}_{n,\text{LSE}}  \stackrel{\mathcal{D}}{\to} \arg \max_{h \in \mathbb{R}} (-0.5|h| + B^{*}_h)
\end{eqnarray}
where for all $h_1,h_2,\ldots, h_r \in \mathbb{R}$ and $r \geq 1$, 
\begin{eqnarray}
(B_{h_1}^{*},B_{h_2}^{*},\ldots, B_{h_r}^{*}) \sim \mathcal{N}_{r}(0,\Sigma)\ \ \text{where}\ \  
\Sigma = ((\gamma_{(h_1,h_2),{\rm{DEP,MSE}}} ))_{1 \leq h_1,h_2 \leq r}.
\end{eqnarray}.
\noindent (c) Suppose (A4), (A7), (D5), (D6), (SNR3) hold and  $||\mu_1-\mu_2||_2 \to c>0$. Then 
\begin{eqnarray}
\tilde{h}_{n,\text{LSE}} \stackrel{\mathcal{D}}{\to}  \arg \max_{h \in \mathbb{Z}} (D^{*}(h) + C^{*}(h) + A^{*}(h))
\end{eqnarray}
where for each $h, t_1,t_2,\ldots,t_r \in \mathbb{Z}$ and $r \geq 1$,
\begin{eqnarray}
&& D^{*}(h) = -0.5  c_1^2 |h|,\ \ \ C^{*}(h) = \sum_{t=0 \wedge h}^{0 \vee h} W^{*}_t,\ \  %\nonumber \\ &&  
(W_{t_1}^{*}, W_{t_2}^{*},\ldots, W_{t_r}^{*}) \sim \mathcal{N}_{r}(0,\Sigma^{*}),\nonumber \\
&& \Sigma^{*} = ((\gamma^{*}_{(t_1,t_1),\text{DEP,LSE}})),\  %_{1 \leq t_1, t_2 \leq r}, %\nonumber \\ && 
A^*(h) = \sum_{t=0 \wedge h}^{0 \vee h} \sum_{k \in \mathcal{K}_0} \bigg[(Y_{kt} + (\mu_{2k}^{*}-\mu_{1k}^{*}){\rm{sign}}(h))^2 -Y_{kt}^2 \bigg]{\rm{sign}}(h).\  \nonumber 
%  y_k &\stackrel{d}{=}& X_{1k}I(h \leq 0) + X_{nk}I(h > 0).\ \ \ \ 
\end{eqnarray}
\end{theorem}
%\subsection{Proof of Theorem \ref{thm: deptildetau}} \label{sec: deptildetau}
\begin{proof}
This  will go through the same arguments as given in Section 5.7,  % the proof of Theorem \ref{thm: adaplse}, 
once we establish
\vskip 3pt
\noindent (a) %$\gamma_{(h_1,h_2),\text{DEP,LSE}}$ equals
%\begin{eqnarray}
   $\displaystyle \sum_{t_1=0 \wedge [h_1 ||\hat{\mu}_1-\hat{\mu}_2||_2^{-2}] }^{0 \vee [h_1||\hat{\mu}_1-\hat{\mu}_2||_2^{-2}] }\ \  \sum_{t_2=0 \wedge [h_2 ||\hat{\mu}_1-\hat{\mu}_2||_2^{-2}] }^{0 \vee [h_2||\hat{\mu}_1-\hat{\mu}_2||_2^{-2}] } \left(\sum_{k=1}^{m}(\hat{\mu}_{1k}-\hat{\mu}_{2k})^2 \hat{C}_{k,t_1-t_2} \right) \stackrel{\text{P}}{\to} \gamma_{(h_1,h_2),\text{DEP,LSE}}$,
%\end{eqnarray}
\vskip 3pt
\noindent (b)  $\lim \sum_{k \in \mathcal{K}_n} (\hat{\mu}_{1k} -\hat{\mu}_{2k})^2 \hat{C}_{k,t_1-t_2} \stackrel{\text{P}}{\to} \gamma^{*}_{(t_1,t_2),\text{DEP,LSE}}$.
\vskip 3pt
\noindent for $h_1, h_2$ in a compact subset of $\mathbb{R}$. Here we shall show (a) only.  (b) can be proved similarly. 
\vskip 3pt
\noindent \textbf{Proof of (a)}. Consider $h_1, h_2 >0$. 
\begin{eqnarray}
&& \sum_{t_1 = 1}^{ h_1 ||\hat{\mu}_1 - \hat{\mu}_2||_2^{-2}} \sum_{t_2 =1}^{ h_2 ||\hat{\mu}_1 - \hat{\mu}_2||_2^{-2}}  \sum_{k=1}^{m} (\hat{\mu}_{1k} - \hat{\mu}_{2k})^2 \hat{C}_{k,t_1-t_2} \nonumber \\
&& \hspace{3 cm}- \sum_{t_1 = 1}^{ h_1 ||{\mu}_1 - {\mu}_2||_2^{-2}} \sum_{t_2 = 1}^{ h_2 ||{\mu}_1 - {\mu}_2||_2^{-2}}  \sum_{k=1}^{m} ({\mu}_{1k} - {\mu}_{2k})^2 \text{Cum} (X_{kt_1},X_{kt_2}) \nonumber  \\
&=& \sum_{t_1 =1}^{ h_1 ||\hat{\mu}_1 - \hat{\mu}_2||_2^{-2}} \sum_{t_2 =1}^{ h_2 ||\hat{\mu}_1 - \hat{\mu}_2||_2^{-2}}  \sum_{k=1}^{m} (\hat{\mu}_{1k} - \hat{\mu}_{2k})^2 (\hat{C}_{k,t_1-t_2} - E\hat{C}_{k,t_1-t_2})\nonumber \\
&& \hspace{0.5 cm}+ \sum_{t_1 =1}^{ h_1 ||\hat{\mu}_1 - \hat{\mu}_2||_2^{-2}} \sum_{t_2 = 1}^{ h_2 ||\hat{\mu}_1 - \hat{\mu}_2||_2^{-2}}  \sum_{k=1}^{m} (\hat{\mu}_{1k} - \hat{\mu}_{2k})^2 (E\hat{C}_{k,t_1-t_2} - \text{Cum} (X_{kt_1},X_{kt_2})) \nonumber  \\
&& +  \bigg[ \sum_{k=1}^{m} (\hat{\mu}_{1k} - \hat{\mu}_{2k})^2 \sum_{t = - h_1 ||\hat{\mu}_1 - \hat{\mu}_2||_2^{-2}}^{h_1 ||\hat{\mu}_1 -\hat{\mu}_2||_2^{-2}} \bigg(\frac{h_1}{||\hat{\mu}_1 - \hat{\mu}_2||_2^{2}}- t \bigg) \text{Cum} (X_{k1},X_{k(t+1)}) \nonumber \\
&& \hspace{1 cm}- \sum_{k=1}^{m} (\hat{\mu}_{1k} - \hat{\mu}_{2k})^2 \sum_{t = - h_1 ||{\mu}_1 - {\mu}_2||_2^{-2}/2}^{h_1 ||{\mu}_1 -{\mu}_2||_2^{-2}/2} \bigg(\frac{h_1}{||\hat{\mu}_1 - \hat{\mu}_2||_2^{2}}- t \bigg) \text{Cum} (X_{k1},X_{k(t+1)}) \bigg] \nonumber \\
&& + \bigg[ \sum_{k=1}^{m} (\hat{\mu}_{1k} - \hat{\mu}_{2k})^2 \sum_{t = - h_1 ||{\mu}_1 - {\mu}_2||_2^{-2}/2}^{h_1 ||{\mu}_1 -{\mu}_2||_2^{-2}/2} \bigg(\frac{h_1}{||\hat{\mu}_1 - \hat{\mu}_2||_2^{2}}- t \bigg) \text{Cum} (X_{k1},X_{k(t+1)}) \nonumber \\
&& \hspace{1 cm}- \sum_{k=1}^{m} (\hat{\mu}_{1k} - \hat{\mu}_{2k})^2 \sum_{t = - h_1 ||{\mu}_1 - {\mu}_2||_2^{-2}/2}^{h_1 ||{\mu}_1 -{\mu}_2||_2^{-2}/2} \bigg(\frac{h_1}{||{\mu}_1 - {\mu}_2||_2^{2}}- t \bigg) \text{Cum} (X_{k1},X_{k(t+1)}) \bigg] \nonumber \\
&& + \bigg[ \sum_{k=1}^{m} (\hat{\mu}_{1k} - \hat{\mu}_{2k})^2 \sum_{t = - h_1 ||{\mu}_1 - {\mu}_2||_2^{-2}/2}^{h_1 ||{\mu}_1 -{\mu}_2||_2^{-2}/2} \bigg(\frac{h_1}{||{\mu}_1 - {\mu}_2||_2^{2}}- t \bigg) \text{Cum} (X_{k1},X_{k(t+1)}) \bigg] \nonumber \\
&& \hspace{1 cm} - \sum_{k=1}^{m} ({\mu}_{1k} - {\mu}_{2k})^2 \sum_{t = - h_1 ||{\mu}_1 - {\mu}_2||_2^{-2}/2}^{h_1 ||{\mu}_1 -{\mu}_2||_2^{-2}/2} \bigg(\frac{h_1}{||{\mu}_1 - {\mu}_2||_2^{2}}- t \bigg) \text{Cum} (X_{k1},X_{k(t+1)}) \bigg] \nonumber \\
&& + \sum_{k=1}^{m} ({\mu}_{1k} - {\mu}_{2k})^2 \sum_{|t| >  h_1 ||{\mu}_1 - {\mu}_2||_2^{-2}/2} \bigg(\frac{h_1}{||{\mu}_1 - {\mu}_2||_2^{2}}- t \bigg) \text{Cum} (X_{k1},X_{k(t+1)}). \nonumber \\
&=& T_1 + T_2 + T_3 + T_4 + T_5 + T_6,\ \ \ \text{(say)}.
\end{eqnarray}
In Lemma 5.8, % \ref{lem: adap}, 
we have seen that
%\begin{eqnarray}
$\frac{||\hat{\mu}_1 - \hat{\mu}_2||_2^2}{||\mu_1 - \mu_2||_2^2} - 1 = o_{P}(1)$. 
%\end{eqnarray}
Thus, by (\ref{eqn: add}) and for some $C>0$,
%\begin{eqnarray}
$|T_4| + |T_5| \leq C\bigg(\frac{||\hat{\mu}_1 - \hat{\mu}_2||_2^2}{||\mu_1 - \mu_2||_2^2} - 1 \bigg) = o_{P}(1)$.
%\end{eqnarray}
Also,  for some $C_1, C_2>0$
\begin{eqnarray}
|T_3| + |T_6| \leq C_1 \sum_{|t| > C_2||\mu_1 - \mu_2||_2^{-2}} \sup_{k} |\text{Cum} (X_{k1},X_{k(t+1)})| + o_{P}(1) = o_{P}(1).
\end{eqnarray}
It is easy to see that for some $C>0$, $\sup_{k} |E\hat{C}_{k,t_1-t_2} - \text{Cum}(X_{kt_1},X_{kt_2})| \leq Cn^{-1}$. Therefore, for some $C>0$,
%\begin{eqnarray}
$|T_2|  \leq C(n||\hat{\mu_1}-\hat{\mu}_2||_2^2)^{-1} = o_{P}(1)$.
%\end{eqnarray}
Also, for some $C_1,C_2>0$, $$|T_1| \leq C_2 \sum_{t=-C_1 ||\hat{\mu}_1 - \hat{\mu}_2||_2^{-2}}^{C_1 ||\hat{\mu}_1 - \hat{\mu}_2||_2^{-2}} |\hat{C}_{k,t} - E\hat{C}_{k,t}|.$$
Now, using the tail probability $P(|Z|>\epsilon) \leq C_1 e^{-C_2\epsilon^2}$ for the standard normal variable $Z$ and for some $C_1,C_2>0$, it is easy to see that $T_1 = o_{P}(1)$.  Similarly idea works for $h_1,h_2<0$;  $h_1<0,h_2>0$ and $h_1>0,h_2<0$. Hence (a) is established. 
%\vskip 5pt
%\noindent 

This completes the proof of Theorem \ref{thm: deptildetau}.
\end{proof}

\subsection{Proof of auxiliary lemmas}
\subsubsection{Proof of Lemma 5.2} % \ref{lem: expectation}} 
\label{subsec: expectation}
Note that
\begin{align*}
& E(N_{1k}^2 (b)) = E \left(\frac{1}{nb}\sum_{t=1}^{nb}(X_{kt} - E(X_{kt})) \right)^2 = \frac{1}{(nb)^2} \sum_{t=1}^{nb} \mbox{Var}(X_{kt}) \leq \frac{1}{nb} \sup_{k,t,n} \mbox{Var}(X_{kt}), \\
& E(N_{2k}^2 (b)) = E \left(\frac{1}{n(1-b)}\sum_{t=nb+1}^{n}(X_{kt} - E(X_{kt})) \right)^2 = \frac{1}{(n(1-b))^2} \sum_{t=nb+1}^{n} \mbox{Var}(X_{kt}) \\ 
& \hspace{8.5 cm}\leq \frac{1}{n(1-b)} \sup_{k,t,n} \mbox{Var}(X_{kt}), \\
& E(N_{3k}^2 (b,\tau)) = E \left(\frac{1}{n(\tau-b)}\sum_{t=nb+1}^{n\tau}(X_{kt} - E(X_{kt})) \right)^2 = \frac{1}{(n(\tau-b))^2} \sum_{t=nb+1}^{n\tau} \mbox{Var}(X_{kt}) \\
& \hspace{9 cm}\leq \frac{1}{n(\tau-b)} \sup_{k,t,n} \mbox{Var}(X_{kt}), \\
& E(N_{1k}(\tau) N_{3k}(b,\tau)) = \frac{1}{n\tau} \frac{1}{n(\tau-b)} E \left[\left(\sum_{t=1}^{n\tau} (X_{kt} - E(X_{kt})) \right) \left(\sum_{t=nb+1}^{n\tau} (X_{kt} - E(X_{kt})) \right) \right] \\
& \hspace{3.5 cm}= \frac{1}{n\tau} \frac{1}{n(\tau-b)}  \sum_{t=nb+1}^{n\tau} \mbox{Var}(X_{kt}) \leq \frac{1}{n\tau} \sup_{k,t,n} \mbox{Var}(X_{kt}), \\
& E(N_{2k}(b) N_{3k}(b,\tau)) = \frac{1}{n(1-b)} \frac{1}{n(\tau-b)} E \left[\left(\sum_{t=nb+1}^{n} (X_{kt} - E(X_{kt})) \right) \left(\sum_{t=nb+1}^{n\tau} (X_{kt} - E(X_{kt})) \right) \right] \\
& \hspace{3.5 cm}= \frac{1}{n(1-b)} \frac{1}{n(\tau-b)}  \sum_{t=nb+1}^{n\tau} \mbox{Var}(X_{kt}) \leq \frac{1}{n(1-b)} \sup_{k,t,n} \mbox{Var}(X_{kt}).
\end{align*}
Therefore,
\begin{align*}
& \sup_{k,b} E(N_{1k}^2 (b)) \leq Cn^{-1},\ \ \sup_{k,b} E(N_{2k}^2(b)) \leq Cn^{-1},\ \ \sup_{k} E(N_{3k}(b,\tau)) \leq C(n(\tau-b))^{-1}, \\
 & \sup_{k,b<\tau} E(N_{1k}(\tau) N_{3k}(b,\tau)) \leq Cn^{-1},\ \ \sup_{k,b<\tau} E(N_{2k}(b) N_{3k}(b,\tau)) \leq Cn^{-1}.
\end{align*}
Since $N_{4k}$ is non-random, we get $E(N_{4k}^2) = N_{4k}^2$. Moreover, since $N_{1k}(b)$, $N_{2k}(\tau)$ and $N_{3k}(b,\tau)$ are independently distributed, we have
\begin{align*}
\sup_{k,b<\tau} E(N_{1k}(b) N_{3k}(b,\tau)) = 0,\ \ \sup_{k,b<\tau} E(N_{2k}(\tau) N_{3k}(b,\tau)) = 0.
\end{align*}
Finally, since $E(N_{1k}(b)) = E(N_{2k}(b)) = E(N_{3k}(b,\tau)) = 0$ for all $b$, we get
\begin{align*}
\sup_{k,b} E(N_{1k}(b)N_{4k}) = 0,\ \ \sup_{k,b}E(N_{2k}(b)N_{4k}) = 0\ \ \text{and}\ \ \sup_{k,b}E(N_{3k}(b,\tau)N_{4k}) = 0.
\end{align*}
Hence, Lemma 5.2 % \ref{lem: expectation}
 is established. \qed

\subsubsection{Proof of Lemma 5.3} % \ref{lem: variancelse}} 
\label{subsec: variance}
Note that by (A1) and for some $C>0$, 
\begin{align*}
\mbox{Var}(N_{1k}^2(b)) &= E(N_{1k}^4 (b)) - (E(N_{1k}^2 (b)))^2 \\
& = E \bigg(\frac{1}{nb}\sum_{t=1}^{nb}(X_{kt}-E(X_{kt})) \bigg)^4 - \bigg(\frac{1}{(nb)^2} \sum_{t=1}^{nb} \mbox{Var}(X_{kt}) \bigg)^2 \\
& = \frac{1}{(nb)^4} \bigg[ \sum_{t=1}^{nb}E(X_{kt}-E(X_{kt}))^4 + 2 \bigg( \sum_{t=1}^{nb} \mbox{Var}(X_{kt}) \bigg)^2\bigg] \\
& \leq  \sup_{k,t,n} E(X_{kt}-E(X_{kt}))^4 (nb)^{-3} + 2 \sup_{k,t,n} \mbox{Var}(X_{kt}) (nb)^{-2} \\
& \leq Cn^{-2}.
\end{align*}
Therefore, $\sup_{k,b} \mbox{Var}(N_{1k}^2 (b)) \leq Cn^{-2}$. Similarly, $\sup_{k,b} \mbox{Var}(N_{2k}^2 (b)) \leq Cn^{-2}$.

\noindent Next,
\begin{align*}
\mbox{Var}(N_{3k}^2(b,\tau)) &= E(N_{3k}^4 (b,\tau)) - (E(N_{3k}^2 (b,\tau)))^2 \\
& = E \bigg(\frac{1}{n(\tau-b)}\sum_{t=nb+1}^{n\tau}(X_{kt}-E(X_{kt})) \bigg)^4 - \bigg(\frac{1}{(n(\tau-b))^2} \sum_{t=nb+1}^{n\tau} \mbox{Var}(X_{kt}) \bigg)^2 \\
& = \frac{1}{(n(\tau-b))^4} \bigg[ \sum_{t=nb+1}^{n\tau}E(X_{kt}-E(X_{kt}))^4 + 2 \bigg( \sum_{t=nb+1}^{n\tau} \mbox{Var}(X_{kt}) \bigg)^2\bigg] \\
& \leq  \sup_{k,t,n} E(X_{kt}-E(X_{kt}))^4 (n(\tau-b))^{-3} + 2 \sup_{k,t,n} \mbox{Var}(X_{kt}) (n(\tau-b))^{-2} \\
& \leq C(n(\tau-b))^{-2}.
\end{align*}
Therefore,  $\sup_{k} \mbox{Var}(N_{3k}^2(b,\tau)) \leq C(n(\tau-b))^{-2}$.  Similarly, $\sup_{k} \mbox{Var}(N_{1k}(b)N_{3k}(b,\tau)) \leq C(n(\tau-b))^{-2}$ and $\sup_{k} \mbox{Var}(N_{2k}(b)N_{3k}(b,\tau)) \leq C(n(\tau-b))^{-2}$. Moreover, as $N_{4k}$ is non-random, $\mbox{Var}(N_{4k}^2) = 0$. Also by Lemma 5.2, % \ref{lem: expectation},
\begin{align*}
\mbox{Var}(N_{4k} N_{3k}(b,\tau)) = N_{4k}^2 E(N_{3k}(b,\tau))^2 \leq CN_{4k}^2 (n(\tau-b))^{-1}.
\end{align*}
Therefore,   $\sup_ {k} \mbox{Var}(N_{4k} N_{3k}(b,\tau)) \leq CN_{4k}^2 (n(\tau-b))^{-1}$.
Similarly,  $\sup_ {k} \mbox{Var}(N_{4k} N_{2k}(b,\tau)) \leq CN_{4k}^2 n^{-1}$. 
This completes the proof of Lemma 5.3. % \ref{lem: variancelse}. 
\qed

\subsubsection{Proof of Lemma 5.7} % \ref{lem: Mndom}} 
\label{subsec: Mndom}
Note that
$L_{k,n}(b) - L_{k,n}(\tau) = \sum_{i=1}^{4} A_{ik} (b) + \mathbb{M}_{7k}(b)$ %+ A_{3k} (b)$
where $\{\mathbb{M}_{7k}(b)\}$ and $\{A_{ik}(b)\}$ are taken respectively from  Sections 5.5 and 5.6. % ()  and (\ref{eqn: aseries})-(\ref{eqn: aseriesen}). 
\noindent To prove Lemma 5.7, % \ref{lem: Mndom},  
it is enough to establish (a)-(d) given below. 
\begin{eqnarray}
&& (a)\ \ P(A_{1k}(b) \leq C\mathbb{M}_{1k}(b)) \to 1,\ \ \ (b)\ \ P(A_{2k}(b) \leq C\mathbb{M}_{2k}(b)) \to 1,\nonumber \\
&& (c)\ \ P(A_{3k}(b) \leq C(\mathbb{M}_{3k}(b)+ \mathbb{M}_{4k}(b))) \to 1,\ \ \ (d)\ \ P(A_{4k}(b) \leq C(\mathbb{M}_{5k}(b)+\mathbb{M}_{6k})) \to 1. \nonumber 
\end{eqnarray}
Next, we prove (a) and (c) only. The proof of (b) and (d) is similar. 

\noindent We state two lemmas that are useful in proving (a) and (c). Their proofs are respectively given in Sections 
\ref{subsec: inlemma} and \ref{subsec: mlesdvbdd}. 

\begin{lemma} \label{lem: inlemma}
Suppose (B3)-(B5), (B8), (B9) hold and $\log m(n) = o(n)$. Then, for all $b \in (c,1-c)$, we have
%\begin{eqnarray}
$P(\hat{\theta}_k(b),  \hat{\eta}_k(b)  \in \Lambda \ \ \forall k) \to 1$. 
%\nonumber
%\end{eqnarray}
\end{lemma}

\begin{lemma} \label{lem: mlesdvbdd}
 Suppose $\{X_i:\ 1 \leq i \leq n\}$ and $\{Y_i:\ 1 \leq i \leq n\}$ are  two independent random samples from $\mathbb{P}_a$ and $\mathbb{P}_b$ for $a, b \in \Lambda$, respectively. %Let $s$ be either a finite non-negative integer  or $n^{-1}s$ converges to a finite positive constant.  
 Suppose (B4) and (B5) hold. Then, for some $0 < C_1 \leq C_2 < \infty$, $s_n = O(n)$ and large $n$,    
 \begin{eqnarray}
 0 < C_1 &\leq & \inf_{\lambda \in \Lambda} \bigg|\frac{1}{n+s_n}\sum_{i=1}^{n} \frac{\partial^2}{\partial \lambda^2} \log \mathbb{P}_{\lambda}(X_i) + \frac{1}{n+s_n}\sum_{i=1}^{s_n} \frac{\partial^2}{\partial \lambda^2} \log \mathbb{P}_{\lambda}(Y_i) \bigg|  \nonumber \\
 &\leq & \sup_{\lambda \in \Lambda} \bigg|\frac{1}{n+s_n}\sum_{i=1}^{n} \frac{\partial^2}{\partial \lambda^2} \log \mathbb{P}_{\lambda}(X_i) + \frac{1}{n+s_n}\sum_{i=1}^{s_n} \frac{\partial^2}{\partial \lambda^2} \log \mathbb{P}_{\lambda}(Y_i) \bigg| \leq  C_2 < \infty. \nonumber
\end{eqnarray}  
 \end{lemma}

\noindent Now we are ready to prove (a) and (c).
\vskip 5pt
\noindent \textbf{Proof of (a)}.
For some $\theta_k^*$ between  $\hat{\theta}_k(b)$ and  $\hat{\theta}_k(\tau)$ and some $C>0$, we get $A_{1k}(b)$ equals
\begin{eqnarray}
&& - \frac{1}{n}\sum_{t=1}^{nb} \bigg[(\hat{\theta}_k(\tau) - \hat{\theta}_k(b)) \frac{\partial }{\partial \theta} \log \mathbb{P}_\theta(X_{kt}) \bigg|_{\theta = \hat{\theta}_k (b)} + \Big( \hat{\theta}_k(\tau) - \hat{\theta}_k(b) \Big)^2 \frac{\partial^2}{\partial \theta^2} \log \mathbb{P}_\theta (X_{kt}) \bigg|_{\theta= \theta^*}  \bigg] \nonumber \\
& = &  b \Big( \hat{\theta}_k(\tau) - \hat{\theta}_k(b) \Big)^2 \bigg[ - \frac{1}{nb}\sum_{t=1}^{nb} \frac{\partial^2}{\partial \theta^2} \log \mathbb{P}_\theta (X_{kt}) \bigg|_{\theta= \theta^*_k} \bigg]. \nonumber  
\end{eqnarray}

\noindent By Lemma \ref{lem: inlemma}, $P(\theta_k^{*} \in \Lambda) \to 1$ and hence by Lemma \ref{lem: mlesdvbdd}, there is $C>0$ such that
\begin{eqnarray}
P(A_{1k}(b) \leq C ( \hat{\theta}_k(\tau) - \hat{\theta}_k(b) )^2) \to 1. \nonumber
\end{eqnarray}

\noindent Moreover, note that for some $\theta_k^*$ between $\hat{\theta}_k(b)$ and $\theta_k$,
\begin{eqnarray} \nonumber 
0 = \sum_{t=1}^{nb} \frac{\partial}{\partial \theta} \log \mathbb{P}_\theta(X_{kt}) \bigg|_{\theta = \hat{\theta}_k(b)} = \sum_{t=1}^{nb} \frac{\partial}{\partial \theta_k} \log \mathbb{P}_{\theta_k} (X_{kt}) + (\hat{\theta}_k(b) - \theta_k) \sum_{t=1}^{nb} \frac{\partial^2}{\partial \theta^2} \log \mathbb{P}_\theta(X_{kt}) \bigg|_{\theta = \theta_k^*}
\end{eqnarray}
which implies
\begin{eqnarray}
(\hat{\theta}_k(b) - \theta_k) = \frac{ \frac{1}{nb} \sum_{t=1}^{nb} \frac{\partial}{\partial \theta_k } \log \mathbb{P}_{\theta_k} (X_{kt})}{ - \frac{1}{nb} \sum_{t=1}^{nb} \frac{\partial^2}{\partial \theta^2} \log \mathbb{P}_\theta(X_{kt}) \bigg|_{\theta = \theta_k^*} }. \nonumber
\end{eqnarray}
Similarly, for some $\theta^{**}_k$ between $\hat{\theta}_k(\tau)$ and $\theta_k$, 
\begin{eqnarray}
(\hat{\theta}_k(\tau) - \theta_k) = \frac{ \frac{1}{n\tau} \sum_{t=1}^{n\tau} \frac{\partial}{\partial \theta_k} \log \mathbb{P}_{\theta_k} (X_{kt})}{ - \frac{1}{n\tau} \sum_{t=1}^{n\tau} \frac{\partial^2}{\partial \theta^2} \log \mathbb{P}_\theta(X_{kt}) \bigg|_{\theta = \theta_k^{**}} }. \nonumber
\end{eqnarray}
 Hence,
\begin{eqnarray}
(\hat{\theta}_k(b) - \hat{\theta}_k(\tau))^2 &=& \bigg[\frac{ \frac{1}{nb} \sum_{t=1}^{nb} \frac{\partial}{\partial \theta_k} \log \mathbb{P}_{\theta_k} (X_{kt})}{ - \frac{1}{nb} \sum_{t=1}^{nb} \frac{\partial^2 }{\partial \theta^2} \log \mathbb{P}_{\theta_k} (X_{kt}) \bigg|_{\theta = \theta^*_k} } - \frac{ \frac{1}{n\tau} \sum_{t=1}^{n\tau} \frac{\partial}{\partial \theta_k} \log \mathbb{P}_{\theta_k} (X_{kt})}{ - \frac{1}{n\tau} \sum_{t=1}^{n\tau} \frac{\partial^2}{\partial \theta^2} \log \mathbb{P}_\theta(X_{kt}) \bigg|_{\theta = \theta_k^{**}} } \bigg]^2 \nonumber \\
& \leq & \frac{2}{a_n^2 \wedge b_n^2} \bigg [\frac{1}{n\tau} \sum_{t=1}^{n\tau} \frac{\partial}{\partial \theta_k} \log \mathbb{P}_{\theta_k}(X_{kt}) -  \frac{1}{nb} \sum_{t=1}^{nb} \frac{\partial}{\partial \theta_k} \log \mathbb{P}_{\theta_k}(X_{kt})\bigg]^2 \nonumber 
\end{eqnarray}
where %for all large $n$ and some $C>0$,
\begin{eqnarray}
a_n^2 &=&  \left( - \frac{1}{nb} \sum_{t=1}^{nb} \frac{\partial^2 }{\partial \theta^2} \log \mathbb{P}_{\theta_k} (X_{kt}) \bigg|_{\theta = \theta^*_k}\right)^2\ %\geq C 
\ \ \text{and}\ \  %\frac{1}{nb} \sum_{t=1}^{nb} G_{2}(X_{kt}) \stackrel{P}{\to} CE(G_{2}(X_{k1})), 
%\nonumber \\
b_n^2 =  \left( - \frac{1}{n\tau} \sum_{t=1}^{n\tau} \frac{\partial^2 }{\partial \theta^2} \log \mathbb{P}_{\theta_k} (X_{kt}) \bigg|_{\theta = \theta^{**}_k}\right)^2. % \geq C. %\frac{1}{n\tau} \sum_{t=1}^{n\tau} G_{2}(X_{kt}) \stackrel{P}{\to} CE(G_{2}(X_{k1})). 
\nonumber
\end{eqnarray}
Now by Lemma \ref{lem: inlemma}, $P(\theta_k^{*}, \theta_k^{**} \in \Lambda) \to 1$ and thus by Lemma \ref{lem: mlesdvbdd},  we have $C>0$ such that %we have  lower bound for both $a_n^2 \wedge b_n^2$. 
%\begin{eqnarray}
$P(A_{1k}(b) \leq C \mathbb{M}_{1k}(b)) \to 1$. 
%\end{eqnarray}
Hence (a) is proved.

\noindent \textbf{Proof of (c)}. For some $\eta_k^{*}$ and $\eta_k^{**}$ between $\hat{\eta}_k(b)$ and $\eta_k$,  we have
\begin{eqnarray}
A_{3k} (b) & =& \frac{1}{n}\sum_{t = nb +1}^{n\tau} \bigg( \log \mathbb{P}_{\hat{\eta}_k(b)}(X_{kt}) - \log \mathbb{P}_{\eta_k} (X_{kt}) \bigg) \nonumber \\
& =&  \Big(\hat{\eta}_k(b) - \eta_k\Big) \sum_{t = nb +1}^{n\tau} \frac{\partial}{\partial \eta_k} \log \mathbb{P}_{\eta_k} (X_{kt}) + \Big(\hat{\eta}_k(b) - \eta_k\Big)^2 \sum_{t = nb +1}^{n\tau} \frac{\partial^2}{\partial^2 \eta_k} \log \mathbb{P}_{\eta_k} (X_{kt}) \bigg|_{\eta_k = \eta_k^*} \nonumber \\
&=& \frac{\bigg|\frac{1}{n(1-b)} \sum_{t=nb+1}^{n}\frac{\partial}{\partial \eta_k} \log \mathbb{P}_{\eta_k}(X_{kt})\bigg|}{\bigg| \frac{1}{n(1-b)} \sum_{t=nb+1}^{n}\frac{\partial^2}{\partial \eta_k^2} \log \mathbb{P}_{\eta_k}(X_{kt})\bigg|_{\eta_k = \eta_k^{**}}\bigg|} \bigg|\sum_{t = nb +1}^{n\tau} \frac{\partial}{\partial \eta_k} \log \mathbb{P}_{\eta_k} (X_{kt}) \bigg| \nonumber \\
&& + \frac{\bigg|\frac{1}{n(1-b)} \sum_{t=nb+1}^{n}\frac{\partial}{\partial \eta_k} \log \mathbb{P}_{\eta_k}(X_{kt})\bigg|^2}{\bigg| \frac{1}{n(1-b)} \sum_{t=nb+1}^{n}\frac{\partial^2}{\partial \eta_k^2} \log \mathbb{P}_{\eta_k}(X_{kt})\bigg|_{\eta_k = \eta_k^{**}}\bigg|^2} \bigg|\sum_{t = nb +1}^{n\tau} \frac{\partial^2}{\partial \eta_k^2} \log \mathbb{P}_{\eta_k} (X_{kt}) \bigg|_{\eta_k = \eta_k^*} \bigg|. \nonumber
\end{eqnarray}

\noindent Thus, by Lemma \ref{lem: inlemma}, $\eta_k^{*}$ and $\eta_k^{**}$ are in $\Lambda$ with probability tending to $1$ and hence by Lemma \ref{lem: mlesdvbdd}, there is $C>0$ such that
%\begin{eqnarray}
$P(A_{3k}(b) \leq C(\mathbb{M}_{3k}(b) + \mathbb{M}_{4k}(b))) \to 1$.
%\end{eqnarray}
Hence (c) is proved.

\noindent This completes the proof of Lemma 5.7. \qed % \ref{lem: Mndom}.  \qed

\subsubsection{Proof of Lemma \ref{lem: inlemma}} \label{subsec: inlemma}
As $\theta$ and $\eta$ are interior points of $\Lambda$, there is $\epsilon >0$ such that
\begin{eqnarray}
(\theta - \epsilon, \theta + \epsilon), (\eta -\epsilon, \eta + \epsilon) \subset \Lambda.
\end{eqnarray}
To prove Lemma \ref{lem: inlemma}, it is enough to establish that for some $b \in (c,1-c)$,
\begin{eqnarray}
P( |\hat{\theta}_k(b) - \theta_k| > \epsilon\ \text{for some $k$}) \to 0,
P( |\hat{\eta}_k(b) - \eta_k| > \epsilon\ \text{for some $k$}) \to 0.
\end{eqnarray}
\noindent For some $\theta_k^*$ between  $\hat{\theta}_k(b)$ and $\theta_k $, we have
\begin{eqnarray}
0 &=& \sum_{t=1}^{nb} \frac{\partial}{\partial \theta} \log \mathbb{P}_\theta(X_{kt}) \bigg|_{\theta = \hat{\theta}_k(b)} \nonumber \\
& = & \sum_{t=1}^{nb} \frac{\partial}{\partial \theta_k} \log \mathbb{P}_{\theta_k} (X_{kt}) + (\hat{\theta}_k(b) - \theta_k) \sum_{t=1}^{nb} \frac{\partial^2}{\partial \theta^2} \log \mathbb{P}_\theta(X_{kt}) \bigg|_{\theta = \theta_k^*}
\end{eqnarray}
which implies
\begin{eqnarray}
|\hat{\theta}_k(b) - \theta_k| = \frac{ \bigg| \frac{1}{nb} \sum_{t=1}^{nb} \frac{\partial}{\partial \theta_k } \log \mathbb{P}_{\theta_k} (X_{kt}) \bigg|}{ \bigg|- \frac{1}{nb} \sum_{t=1}^{nb} \frac{\partial^2}{\partial \theta^2} \log \mathbb{P}_\theta(X_{kt}) \bigg|_{\theta = \theta_k^*} \bigg|}.
\end{eqnarray}
Therefore by (B5), for some $C>0$ %By Lemma \ref{lem: mlesdvbdd},   for large enough $n$ and some $C>0$,
\begin{eqnarray}
|\hat{\theta}_k(b) - \theta_k| \leq C \left( \frac{1}{nb} \sum_{t=1}^{nb} G_{2}(X_{kt}) \right)^{-2} \bigg| \frac{1}{nb} \sum_{t=1}^{nb} \frac{\partial}{\partial \theta_k } \log \mathbb{P}_{\theta_k} (X_{kt}) \bigg|.
\end{eqnarray}
%Hence, by (B8), (B9) and for some $C, C_1, C_2 >0$
Chose $\delta < \inf_{k,t} E(G_2(X_{kt}))$. Hence, for some $C>0$
\begin{eqnarray}
&&\bigg|\frac{1}{nb}\sum_{t=1}^{nb} G_{2}(X_{kt}) - EG_2(X_{k1}) \bigg| < \delta 
\implies  \bigg|\bigg|\frac{1}{nb}\sum_{t=1}^{nb} G_{2}(X_{kt})\bigg| - EG_2(X_{k1}) \bigg| < \delta \nonumber \\
\implies && \bigg|\frac{1}{nb}\sum_{t=1}^{nb} G_{2}(X_{kt})\bigg| \geq \inf_{k,t}EG_2(X_{k1}) - \delta > C >0.
\end{eqnarray}
Therefore by (B8), (B9) and $\log m(n) = o(n)$, and  for some $C_1,C_2, C_3>0$, we have
\begin{eqnarray}
 && P( |\hat{\theta}_k(b) - \theta_k| > \epsilon\ \text{for some $k$}) 
\leq \sum_{k=1}^{m} P( |\hat{\theta}_k(b) - \theta_k| > \epsilon) \nonumber \\
& \leq &  \sum_{k=1}^{m} P \left( \bigg| \frac{1}{nb} \sum_{t=1}^{nb} \frac{\partial}{\partial \theta_k } \log \mathbb{P}_{\theta_k} (X_{kt}) \bigg| > C\bigg|\frac{1}{nb}\sum_{t=1}^{nb} G_{2}(X_{kt})\bigg|,\ \ \bigg|\frac{1}{nb}\sum_{t=1}^{nb} G_{2}(X_{kt}) - EG_2(X_{k1}) \bigg| < \delta \right) \nonumber \\
&& \hspace{1 cm} + \sum_{k=1}^{m}P\left(\bigg|\frac{1}{nb}\sum_{t=1}^{nb} G_{2}(X_{kt}) - EG_2(X_{k1}) \bigg| > \delta \right) \nonumber \\
& \leq & \sum_{k=1}^{m} P\left(\bigg| \frac{1}{nb} \sum_{t=1}^{nb} \frac{\partial}{\partial \theta_k } \log \mathbb{P}_{\theta_k} (X_{kt}) \bigg| > C_1 \right) + \sum_{k=1}^{m}P\left(\bigg|\frac{1}{nb}\sum_{t=1}^{nb} G_{2}(X_{kt}) - EG_2(X_{k1}) \bigg| > \delta \right) \nonumber \\
& \leq & C_2 m e^{-C_3 n} \to 0. \nonumber
\end{eqnarray}
 This completes the proof of Lemma \ref{lem: inlemma}. \qed
 
 \subsubsection{Proof of Lemma \ref{lem: mlesdvbdd}} \label{subsec: mlesdvbdd}
\noindent  By (B4) and (3.3), we have % (\ref{eqn: sdv1}), we have
\begin{eqnarray}
  && C_1 \left(\frac{1}{n+s_n}\sum_{i=1}^{n} G_2(X_i) + \frac{1}{n+s_n}\sum_{i=1}^{s_n} G_2(Y_i) \right)  \nonumber \\
  &\leq & \inf_{\lambda \in \Lambda} \bigg|\frac{1}{n+s_n}\sum_{i=1}^{n} \frac{\partial^2}{\partial \lambda^2} \log \mathbb{P}_{\lambda}(X_i) + \frac{1}{n+s_n}\sum_{i=1}^{s_n} \frac{\partial^2}{\partial \lambda^2} \log \mathbb{P}_{\lambda}(Y_i) \bigg|  \nonumber \\
 &\leq & \sup_{\lambda \in \Lambda} \bigg|\frac{1}{n+s_n}\sum_{i=1}^{n} \frac{\partial^2}{\partial \lambda^2} \log \mathbb{P}_{\lambda}(X_i) + \frac{1}{n+s_n}\sum_{i=1}^{s_n} \frac{\partial^2}{\partial \lambda^2} \log \mathbb{P}_{\lambda}(Y_i) \bigg| \nonumber \\
 &\leq &  C_2 \left(\frac{1}{n+s_n}\sum_{i=1}^{n} G_2(X_i) + \frac{1}{n+s_n}\sum_{i=1}^{s_n} G_2(Y_i) \right). \nonumber
\end{eqnarray} 
Moreover, for some $C>0$
\begin{eqnarray}
\frac{1}{n+s}\sum_{i=1}^{n} G_2(X_i) &\stackrel{\text{P}}{\to}&  EG_2(X_1)\ \ \text{and} \nonumber \\
\frac{1}{n+s}\sum_{i=1}^{s} G_2(Y_i)  &\stackrel{\text{P}}{\to} & 0 I(s_n = o(n)) + CEG_2(Y_1)I(s_n \neq o(n), s_n =O(n)).
\end{eqnarray}
Thus, by (3.3), % (\ref{eqn: sdv2}), 
for large $n$
\begin{eqnarray}
\epsilon_1/2 \leq \left(\frac{1}{n+s_n}\sum_{i=1}^{n} G_2(X_i) + \frac{1}{n+s_n}\sum_{i=1}^{s_n} G_2(Y_i) \right) \leq 3 \epsilon_2.
\end{eqnarray}
This completes the proof of Lemma \ref{lem: mlesdvbdd}.  \qed

\subsubsection{Proof of Lemmas 5.8 and 5.9} % \ref{lem: adap}} 
\label{subsec: lemadap}
Suppose  $\epsilon = C \frac{m}{\sqrt{n}}||\mu_1 - \mu_2||_2^{-2} \sqrt{\log m}$ and  $C>0$. Since $\{X_{kt}\}$ are Sub-Gaussian, for some $C_1,C_2,C_3>0$, we have
\begin{align*}
P \left(\bigg|\frac{||\hat{\mu}_1 - \hat{\mu}_2||_2^2}{||\mu_1 - \mu_2||_2^2} -1 \bigg| > \epsilon  \right) &= P \left( \bigg| ||\hat{\mu}_1 - \hat{\mu}_2||_2^2 -||{\mu}_1 - {\mu}_2||_2^2 \bigg|> \epsilon ||{\mu}_1 - {\mu}_2||_2^2 \right) \\
& \leq C_1 \sum_{k=1}^{m} P \left(\sup_{i=1,2} \sqrt{n} |\hat{\mu}_{ik} - \mu_{ik}| \geq C_2 \epsilon \sqrt{n}  ||{\mu}_1 - {\mu}_2||_2^2/m  \right)  \\
& \leq C_3 m\  \text{exp}\{-C_2^2 n\epsilon^2 ||\mu_1 - \mu_2||_2^4/m^2 \} \to 0.
\end{align*}
Therefore,
\begin{eqnarray}
&& \frac{||\hat{\mu}_1 - \hat{\mu}_2||_2^2}{||\mu_1 - \mu_2||_2^2} -1 \stackrel{\text{P}}{\to} 0,\ \  \sum_{k=1}^{m}|(\hat{\mu}_{1k} -\hat{\mu}_{2k})^2 -(\mu_{1k}-\mu_{2k})^2 | = o_{P}(||\mu_1 - \mu_2||_2^2).\ \ \  \label{eqn: lemadapi} 
\end{eqnarray}
Moreover, for some $C,C_1,C_2>0$ and $\epsilon = C n^{-1} \log m$, we have
\begin{eqnarray}
P(\sup_{k,i} |\hat{\sigma}_{ik}^2 - \sigma_{ik}^2| > \epsilon) &\leq & \sum_{i=1,2}\sum_{k=1}^{m}\bigg[ P(|\hat{\sigma}_{ik}^2 + \hat{\mu}_{ik}^2 - \sigma_{ik}^2 - \mu_{ik}^2| > \epsilon/2) + P(|\hat{\mu}_{ik}^2  - \mu_{ik}^2| > \epsilon/2)\bigg] \nonumber \\
& \leq & C_1m (e^{-C_2 n \epsilon} + e^{-C_2 n \sqrt{\epsilon}}) \leq 2 C_1m e^{-C_2 n \epsilon} \to 0. \nonumber
\end{eqnarray}
Therefore, 
\begin{eqnarray} \label{eqn: lemadapii} 
\sup_{k,i} |\hat{\sigma}_{ik}^2 - \sigma_{ik}^2| = o_{P}(1).
\end{eqnarray}
(\ref{eqn: lemadapii}) implies Lemma 5.8 and the second part of Lemma 5.9(a).  

Similarly as (\ref{eqn: lemadapii}), one can show that
\begin{eqnarray}
\sup_{k,t} |E((\tilde{X}_{kt} - E(\tilde{X}_{kt}|\{X_{kt}\}))^4|\{X_{kt}\}) - E(X_{kt}-EX_{kt})^4| = o_{P}(1), \nonumber
\end{eqnarray}
which implies the first part of Lemma $5.9$(a).

\noindent To prove Lemma 5.9(b), we shall show
\begin{align}
\frac{\sum_{k=1}^{m}(\hat{\mu}_{1k}-\hat{\mu}_{2k})^2 \hat{\sigma}_{1k}^2}{\sum_{k=1}^{m}(\hat{\mu}_{1k}-\hat{\mu}_{2k})^2} - \frac{\sum_{k=1}^{m}({\mu}_{1k}-{\mu}_{2k})^2 {\sigma}_{1k}^2}{\sum_{k=1}^{m}({\mu}_{1k}-{\mu}_{2k})^2} \stackrel{\text{P}}{\to} 0. \label{eqn: lemadape1}
\end{align}
Note that
\begin{align*}
& \frac{\sum_{k=1}^{m}(\hat{\mu}_{1k}-\hat{\mu}_{2k})^2 \hat{\sigma}_{1k}^2}{\sum_{k=1}^{m}(\hat{\mu}_{1k}-\hat{\mu}_{2k})^2} - \frac{\sum_{k=1}^{m}({\mu}_{1k}-{\mu}_{2k})^2 {\sigma}_{1k}^2}{\sum_{k=1}^{m}({\mu}_{1k}-{\mu}_{2k})^2}  \\
= &  \frac{\sum_{k=1}^{m}(\hat{\mu}_{1k}-\hat{\mu}_{2k})^2 (\hat{\sigma}_{1k}^2-\sigma_{1k}^2)}{\sum_{k=1}^{m}(\hat{\mu}_{1k}-\hat{\mu}_{2k})^2} + \left(\frac{||\mu_1-\mu_2||_2^2}{||\hat{\mu}_1 - \hat{\mu}_2||_2^2} \right)\frac{\sum_{k=1}^{m}(\hat{\mu}_{1k}-\hat{\mu}_{2k})^2 {\sigma}_{1k}^2}{\sum_{k=1}^{m}({\mu}_{1k}-{\mu}_{2k})^2} \\
& \hspace{9.5 cm} - \frac{\sum_{k=1}^{m}({\mu}_{1k}-{\mu}_{2k})^2 {\sigma}_{1k}^2}{\sum_{k=1}^{m}({\mu}_{1k}-{\mu}_{2k})^2} \\
=& \frac{\sum_{k=1}^{m}(\hat{\mu}_{1k}-\hat{\mu}_{2k})^2 (\hat{\sigma}_{1k}^2-\sigma_{1k}^2)}{\sum_{k=1}^{m}(\hat{\mu}_{1k}-\hat{\mu}_{2k})^2} + \left(\frac{||\mu_1-\mu_2||_2^2}{||\hat{\mu}_1 - \hat{\mu}_2||_2^2} -1 \right)\frac{\sum_{k=1}^{m}(\hat{\mu}_{1k}-\hat{\mu}_{2k})^2 {\sigma}_{1k}^2}{\sum_{k=1}^{m}({\mu}_{1k}-{\mu}_{2k})^2} \\ 
& \hspace{4 cm} + \left(\frac{\sum_{k=1}^{m}(\hat{\mu}_{1k}-\hat{\mu}_{2k})^2 {\sigma}_{1k}^2}{\sum_{k=1}^{m}({\mu}_{1k}-{\mu}_{2k})^2} - \frac{\sum_{k=1}^{m}({\mu}_{1k}-{\mu}_{2k})^2 {\sigma}_{1k}^2}{\sum_{k=1}^{m}({\mu}_{1k}-{\mu}_{2k})^2} \right) \\
= & o_{P}(1) + o_{P}(1)O_{P}(1) + o_{P}(1),\ \ \text{by (\ref{eqn: lemadapi}), (\ref{eqn: lemadapii}) and (A1)}.
\end{align*}
This completes the proof of  Lemma  5.9(b). %\ref{lem: adap}(a). 
Similar argument works for Lemma 5.9(c). % \ref{lem: adap}(b). 
Hence, Lemma 5.9 % \ref{lem: adap}
 is established.  \qed

%\noindent \textbf{Acknowledgment}.

\end{document}